%% file: symdk.tex
\documentclass[10pt]{article}

\input{preamble}

\title{Sheaves on polynomials and polynomial sheaves}
\author{Tobias Dyckerhoff\footnote{Universität Hamburg, Fachbereich Mathematik, Bundesstraße 55, 20146 Hamburg, Germany
	email: {\tt first.last@uni-hamburg.de}}, Till Heine\footnotemark[\value{footnote}], Simon Schneider\footnotemark[\value{footnote}]}

\begin{document}

\maketitle

\begin{abstract}
    We show that several Dold-Kan type equivalences can be understood as
    a passage from constructible sheaves on symmetric products (i.e. spaces of
    polynomials) to polynomial sheaves in the sense of Goodwillie-Weiss
    calculus, by means of a natural analytic expansion functor. This geometric
    perspective provides access to standard tools from sheaf theory such
    as the six functor formalism and $t$-structures. For example, the equivalence
    between filtered objects and coherent chain complexes in a stable
    $\infty$-category receives a geometric incarnation as Verdier duality
    between $!$-sheaves and $*$-cosheaves on an $\infty$-dimensional simplex.
    As a main application, we obtain an intrinsic geometric explanation for the
    appearance of vanishing cycle data for perverse sheaves in a paracyclic variant
    of the Dold--Kan correspondence pioneered by Dwyer and Kan, at the same
    time lifting their equivalence to stable $\infty$-categories.
\end{abstract}

\tableofcontents

\addcontentsline{toc}{section}{Introduction}

\numberwithin{equation}{subsection}

\newpage

\section*{Introduction}

One version of the ``cyclic'' Dold-Kan type correspondences established by
Dwyer and Kan in \cite{DK85} can be formulated (cf. \cite{heine:dk}) as an
equivalence between the category of paracyclic abelian groups and the
category $P$ of diagrams of abelian groups of the form
\begin{equation}
    \label{eq:diag_perverse}
    \begin{tikzcd}
        B_0 \ar[bend left=10]{r}{\delta} & \
        \ar[bend left=10]{r}{\delta}\ar[bend left=10]{l}{d} B_1 & \
        \ar[bend left=10]{r}{\delta}\ar[bend left=10]{l}{d} B_2 & 
        \ar[bend left=10]{l}{d}  \cdots 
    \end{tikzcd}
\end{equation}
subject to the conditions
\begin{enumerate}[label = (\alph *)]
    \item\label{it:1} $d^2 = 0$ and $\delta^2 = 0$,
    \item\label{it:2} the various endomorphisms $\id - d \delta$ and $\id - \delta d$ are invertible. 
\end{enumerate}
It is striking that, while the methods used in \cite{DK85} are of a purely
combinatorial nature, without any direct reference to geometry, the category
$P$ does have a very geometric interpretation. Namely denote, for $k \ge 0$, by
$P_k \subset P$ the full subcategory consisting of diagrams satisfying $B_i =
0$ for $i > k$. As can be seen, e.g. by iterated Beilinson gluing \cite{beilinson:howtoglue},
the category $P_k$ is equivalent to the category of perverse sheaves on affine
$k$-space $\CC^k$, constructible with respect to the stratification
corresponding to the hyperplane filtration
\begin{equation}
    \label{eq:cnfilt}
    \{0\} \subset \CC \subset \CC^2 \subset \cdots \subset \CC^k.
\end{equation}

The starting point of this work was the desire to provide a geometric context
for the above paracyclic Dwyer-Kan correspondence in which the appearance
of perverse sheaves becomes manifest -- without reference to the classification
data \eqref{eq:diag_perverse}.
A key observation towards this goal is the fact that the filtration
\eqref{eq:cnfilt} can, by the fundamental theorem of symmetric polynomials, be
reexpressed as the filtration of symmetric products
\begin{equation}
    \label{eq:spnfilt}
    \{0\} \subset \SP^1(\CC) \subset \SP^2(\CC) \subset \cdots \subset \SP^k(\CC).
\end{equation}
In a nutshell, it then turns out that the various power sum vanishing cycles 
\[
    \Phi_{z_1^n + z_2^n + \dots + z_k^n}(F)_0 \quad \text{, $n \ge 0$,}
\]
of a perverse sheaf $F$ on $\SP^k(\CC)$ at the origin $0$ organize into a
paracyclic abelian group and that this assignment describes the Dwyer-Kan
correspondence. 

These results can most naturally be developed within a rather general framework
of {\em analytic expansion} relating sheaves on symmetric products and
polynomial sheaves. This context does not only give a natural habitat for the
Dwyer-Kan correspondence -- as we will see, the construction also provides a
geometric perspective on Lurie's Dold-Kan correspondence for stable
$\infty$-categories making manifest its relation with the classical Dold-Kan
correspondence for abelian categories via Verdier duality. Beyond these
applications, which will be established in this work, the sheaf-theoretic
perspective points towards possible higher--dimensional Dold-Kan type
correspondences which we plan to explore in future work.

We proceed by surveying our geometric setup: Let $N$ be a closed real manifold,
denote by $C = C(N)$ its open cone with tip $*$, and let $\D$ be a pointed
$\infty$-category with limits. In this context, we construct an {\em analytic
expansion} functor
\[
    \An: \Shc^*(\SP^{\infty}(C); \D) \lra \Shan(C^*, \D), \; F \mapsto \Fan
\]
from constructible sheaves on the infinite symmetric product of $C$ to analytic
sheaves on $C^* = C \setminus \{*\}$, defined as locally constant presheaves on the
poset of finite unions of disjoint disks. A key feature is that the analytic
expansion of a sheaf with support on $\SP^k(C)$ is polynomial of degree $\le k$
in the sense of Goodwillie-Weiss calculus. For $C = \CC$ (the cone over a
circle), the space $\SP^k(\CC)$ can be identified with the space of (monic)
polynomials of degree $\le k$ so that the functor $\An$ carries sheaves on this
space of polynomials to polynomial sheaves on $\CC^*$.
The functor $\An$ is in fact very simple to define: Denoting the monoid
structure of $X := \SP^{\infty}(C)$ by $+$, the value of $\Fan$ on an open $U \subset
C^*$ is given by the formula
\begin{equation}\label{eq:an_expansion}
    \Fan(U) := \fib(F(X) \to F(U + X)) \quad (\simeq \Gamma_{\SPi(C \setminus U)}(F)).
\end{equation}

\begin{exa}
    \label{exa:an_expansion} To illustrate the nature of the analytic expansion
    functor consider the example $C = C(\pt) = \Rg$. Suppose $F \in \Shc(\SPi(\Rg);\D)$
    is a constructible sheaf with support on $\SP^2(\Rg)$. This latter space
    can be identified with an open corner of a $2$-simplex parameterized by
    $\{0 \le x \le y\}$ with the $3$ contractible strata $\{x = y = 0\}$, $\{ x
    = 0, y > 0 \}$, and $\{x > 0, y > 0\}$. The sheaf $F$ can be described by
    its exit path representation which amounts to a diagram $F_0 \to F_1 \to
    F_2$ in $\D$. We compute the value of $F^{\an}$ at a disjoint union $U =
    D_1 \cup D_2 \cup D_3$ of three intervals. The intersection of the open $U
    + \SPi(\Rg)$ with $\SP^2(\Rg)$ is depicted in Figure \ref{fig:open}. We
    deduce that $F^{\an}(U)$ is equivalent to the total fiber of the cube
    \[
    \begin{tikzcd}[
      column sep=small,
      row sep=small
    ]
        F_0 \arrow[dr]\arrow[rr]\arrow[dd] & & F_1 \arrow[dr]\arrow[dd] & \\
                                  & F_1 \arrow[rr, crossing over]\arrow[dd] & & F_2\arrow[dd]\\
        F_1 \arrow[dr]\arrow[rr] & & F_2 \arrow[dr] & \\
                   & F_2 \arrow[rr] & & *
    \end{tikzcd}
    \]
\end{exa}

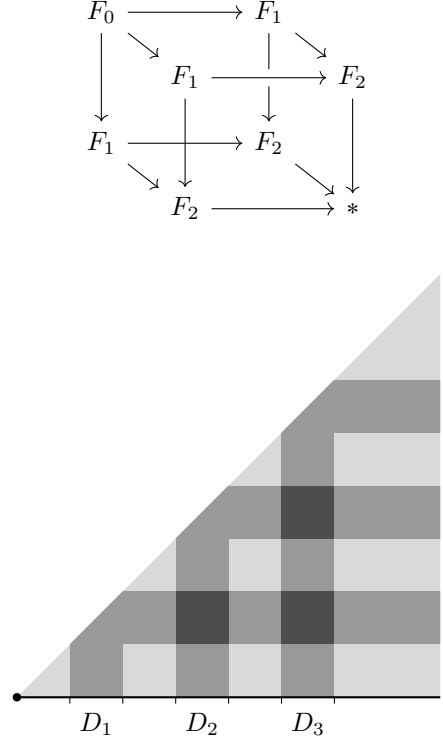
\begin{figure}[htpb]
    \centering
\begin{tikzpicture}[scale=.7]
  \def\xmax{8}
  \def\ymax{8}

  \fill[black!15] (0,0) -- (\xmax,0) -- (\xmax,\ymax) -- cycle;

  \filldraw[black] (0,0) circle (2pt);

  \fill[black!70] (3,1) rectangle (4,2);

  \begin{scope}
    \clip (0,0) -- (\xmax,0) -- (\xmax,\ymax) -- cycle;

    \fill[black!40] (1,0) rectangle (2,\ymax);
    \fill[black!40] (0,1) rectangle (\xmax,2);

    \fill[black!40] (3,0) rectangle (4,\ymax);
    \fill[black!40] (0,3) rectangle (\xmax,4);

    \fill[black!40] (5,0) rectangle (6,\ymax);
    \fill[black!40] (0,5) rectangle (\xmax,6);

    \fill[black!70] (3,1) rectangle (4,2);
    \fill[black!70] (5,1) rectangle (6,2);
    \fill[black!70] (5,3) rectangle (6,4);

  \end{scope}

  \draw[thick] (0,0) -- (\xmax,0);

  \foreach \t in {1,2,3,4,5,6} { \draw (\t,0) -- (\t,-0.12); }
  \node[below] at (1.5,-0.12) {$D_1$};
  \node[below] at (3.5,-0.12) {$D_2$};
  \node[below] at (5.5,-0.12) {$D_3$};

\end{tikzpicture}
    \caption{The intersection of the open $(D_1 \cup D_2 \cup D_3) + \SPi(\Rg)$ with $\SP^2(\Rg)$}%
    \label{fig:open}
\end{figure}

As we will demonstrate, when the $\infty$-category $\D$ is stable, then
analytic expansion provides a uniform geometric perspective on various Dold-Kan
type correspondences. The sheaf-theoretic point of view naturally comes with
concepts such as Verdier duality and geometrically defined $t$-structures which
turn out to play an appealing role. We provide an overview of our results:

\subsection{The symmetric product \texorpdfstring{$\SPi(\RR_{\ge 0})$}{spirg}}%
\label{sub:r}
Let $N = \{* \}$, and hence $C = \RR_{\ge 0}$. The following results are proved in \S \ref{sec:spr}. Statements and
references to proofs can be found in \S \ref{sub:spr_results}.

\begin{enumerate}
    \item Analytic expansion provides an equivalence 
        \begin{equation}\label{eq:stareq}
            \begin{tikzcd}
                \Shc^*(\SPi(\Rg); \D) \arrow[r, "\An"{above}, "\simeq"{below}] &  \Shan(\Rg; \D) 
            \end{tikzcd}
        \end{equation}
        where $\Shc^*(\SPi(\Rg); \D)$ denotes the $\infty$-category of $*$-sheaves
        obtained as an inverse limit over the $*$-restrictions along $\SP^k(\Rg)
        \subset \SP^{k+1}(\Rg)$. The equivalence \eqref{eq:stareq} can be interpreted
        in purely combinatorial terms yielding an equivalence
        \begin{equation}\label{eq:starcombinatorial}
           \begin{tikzcd}
              \Fun(\NN,\D) \arrow[r, "\simeq"{below}] &  \D_{\Delta_+} 
            \end{tikzcd}
        \end{equation}
        between filtered objects and augmented simplicial objects in $\D$,
        respectively. The equivalence \eqref{eq:starcombinatorial} can be
        regarded as an augmented variant of Lurie's Dold-Kan correspondence
        \cite[1.2.4]{lurie:ha}. 
    \item Restricting analytic expansion to the full subcategory
        $\Shc^!(\SPi(\Rg); \D) \subset \Shc^*(\SPi(\Rg); \D)$ of $!$-sheaves
        (defined as an inverse limit over $!$-restrictions), we obtain an equivalence 
        \begin{equation}\label{eq:shriekeq}
            \begin{tikzcd}
                \Shc^!(\SPi(\Rg); \D) \arrow[r, "\An"{above}, "\simeq"{below}] &  \Shan(\Rg; \D)^c
            \end{tikzcd}
        \end{equation}
        with {\em complete} analytic sheaves, i.e. sheaves whose value at the
        empty set is the colimit of the values on nonempty opens.
        Combinatorially, the subcategory of $!$-sheaves can be identified with the full subcategory
        \[
            \Fun(\NN,\D)^0 \subset \Fun(\NN,\D)
        \]
        consisting of filtered objects with zero colimit which, in turn, can
        again be identified with $\Fun(\NN,\D)$ (by associating to each term of
        a filtered object the fiber of the map to the colimit). 
        Upon this identification, the equivalence \eqref{eq:shriekeq} recovers Lurie's stable
        Dold-Kan correspondence
        \begin{equation}\label{eq:shriekcombinatorial}
           \begin{tikzcd}
              \Fun(\NN,\D) \arrow[r, "\simeq"{below}] &  \D_{\Delta}.
            \end{tikzcd}
        \end{equation}
    \item Composing the equivalence \eqref{eq:shriekeq} with Lurie-Verdier duality (\cite[5.5.5]{lurie:ha})
        \begin{equation}
            \begin{tikzcd}
                \Shc^!(\SPi(\Rg); \D) \arrow[r, "\V"{above}, "\simeq"{below}] & \coSh_{\PP}^*(\SPi(\Rg);\D)^0  
             \end{tikzcd}
        \end{equation}
        we obtain an equivalence
        \begin{equation}\label{eq:verdierequivalence}
            \begin{tikzcd}
                \coSh_{\PP}^*(\SPi(\Rg);\D)^0 \arrow[r, "\An\circ\V^{-1}"{above}, "\simeq"{below}] &  \Shan(\Rg; \D)^c
             \end{tikzcd}
        \end{equation}
        between constructible cosheaves supported off the discriminant and
        complete analytic sheaves. This translates to a combinatorial equivalence
        \begin{equation}\label{eq:verdiercombinatorial}
           \begin{tikzcd}
               \Ch_{\ge 0}(\D) \arrow[r, "\simeq"{below}] &  \D_{\Delta}.
            \end{tikzcd}
        \end{equation}
        between coherent connective chain complexes and simplicial objects in
        $\D$ which can be interpreted as a stable coherent variant of the classical Dold-Kan
        correspondence. 
    \item Suppose that $\D$ comes equipped with a $t$-structure with heart
        $\A$. Then the equivalence \eqref{eq:verdierequivalence} is $t$-exact
        with respect to standard $t$-structures and its restriction to hearts
        yields the classical Dold-Kan correspondence
        \begin{equation}\label{eq:classicalverdiercombinatorial}
           \begin{tikzcd}
               \Ch_{\ge 0}(\A) \arrow[r, "\simeq"{below}] &  \A_{\Delta}.
            \end{tikzcd}
        \end{equation}
        In other words, the equivalence \eqref{eq:shriekeq} is {\em not}
        $t$-exact with respect to the standard ``sheaf'' $t$-structure, but
        rather with respect to the Verdier dual ``cosheaf'' $t$-structure. This
        explains geometrically why the formulation of the stable Dold-Kan
        correspondence in terms of filtered objects does not admit a classical
        abelian analog. 
\end{enumerate} 

\begin{exa}
    \label{exa:verdier} 
    To give a sense of the insights brought about by the geometric perspective
    suppose, as in Example \ref{exa:an_expansion}, that $F \in
    \Shc(\SPi(\Rg);\D)$ is a constructible sheaf with support on $\SP^2(\Rg)$.
    As already mentioned in Example \ref{exa:an_expansion}, the sheaf
    $F$ can be described by its exit path representation $F_0 \to F_1 \to F_2$
    in $\D$. Its Verdier dual is given by the constructible cosheaf with entrance
    path representation
    \[
    \begin{tikzcd}
        0 \ar{d} & \ar{l} F_2[-2] \ar{d}\\
        \fib(F_0 \to F_1) & \ar{l} \fib(F_1 \to F_2)[-1]
    \end{tikzcd}
    \]
    as illustrated in Figure \ref{fig:verdier}. The zero relation comes about,
    since the Verdier dual is {\em not} constructible with respect to the
    $\NN$-stratification of $\SP^2(\Rg)$ but rather takes the value $0$ along
    the discriminant locus $\{ x = y \}$. This phenomenon continues to hold in
    higher dimensions so that Verdier duality provides a geometric model for
    the equivalence between filtered objects and coherent chain complexes in a
    stable $\infty$-category $\D$.
\end{exa}

\begin{figure}[htpb]
        \centering
\begin{tikzpicture}[scale=1.0]
  \begin{scope}
      \def\xmax{5}
      \def\ymax{3}

      \fill[black!15] (0,0) -- (\xmax,0) -- (\xmax,\ymax) -- cycle;

      \filldraw[black] (0,0) circle (2pt) node[anchor=north east] (A) {$F_0$};
      \node[below] (B) at (4,0) {$F_1$};
      \node (C) at (4,1.5) {$F_2$};

      \draw[->] (A.east) -- (B.west);
      \draw[->] (B.north) -- (C.south);

      \draw[thick] (0,0) -- (\xmax,0);
  \end{scope}

  \begin{scope}[shift={(6,1)}]
      \draw (0,0) edge[|->] node[above] {$\V$} (1,0);
  \end{scope}

  \begin{scope}[shift={(9,0)}]
      \def\xmax{5}
      \def\ymax{3}

      \fill[black!15] (0,0) -- (\xmax,0) -- (\xmax,\ymax) -- cycle;

      \filldraw[black] (0,0) circle (2pt) node[anchor=north east] (A) {$\fib(F_0 \to F_1)$};
      \node[below] (B) at (4,0) {$\fib(F_1 \to F_2)[-1]$};
      \node (C) at (4,1.5) {$F_2[-2]$};
      \node[anchor=east] (O) at (2,1.5) {$0$};

      \draw[<-] (A.east) -- (B.west);
      \draw[<-] (B.north) -- (C.south);
      \draw[<-] (O.east) -- (C.west);
      \draw[<-] (A.north) -- (O.west);

      \draw[thick] (0,0) -- (\xmax,0);
      \draw[thick, dotted] (0,0) -- (\xmax,\ymax);
  \end{scope}
\end{tikzpicture}
    \caption{A constructible sheaf on $\SP^2(\Rg)$ and its Verdier dual cosheaf}%
    \label{fig:verdier}
\end{figure}
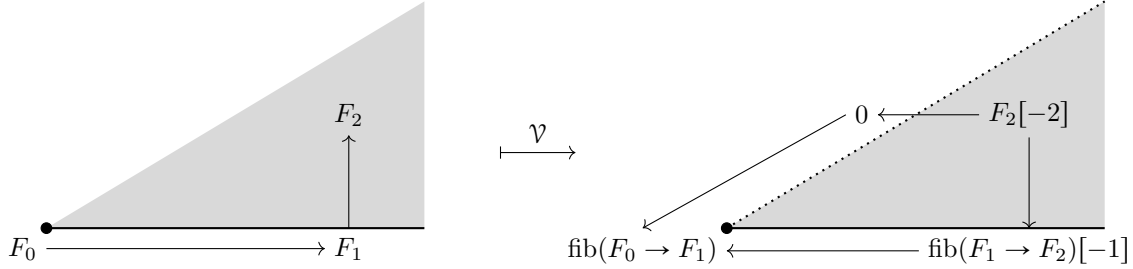

\subsection{The symmetric product \texorpdfstring{$\SPi(\CC)$}{spic}}%
\label{sub:c}

Let $N = S^1$, and hence $C = \CC$. Statements and proofs for the following can be found in \S \ref{sub:complex_results}.

\begin{enumerate}
    \item Analytic expansion provides an equivalence 
        \begin{equation}\label{eq:cstareq}
            \begin{tikzcd}
                \Shc^*(\SPi(\CC); \D) \arrow[r, "\An"{above}, "\simeq"{below}] &  \Shanr(\CC^*; \D) 
            \end{tikzcd}
        \end{equation}
        between constructible $*$-sheaves and {\em radial} analytic sheaves on
        $\CC^*$ defined as locally constant presheaves on the poset of those
        opens in $\DD(\CC^*)$ which are invariant under the action of the
        positive reals $\RR_{+}$ by multiplication. The $\infty$-category
        $\Shanr(\CC^*,\D)$ can be interpreted combinatorially as the
        $\infty$-category of augmented paracyclic objects in $\D$. The category
        $\Shc^*(\SPi(\CC); \D)$ may be described combinatorially in terms of
        the exit path category, given as an iterated join of circles. This
        gives a description
        \begin{equation}
           \begin{tikzcd}
                   \Fun(\Delta^0 * \bigstar_{n \in \NN} S^1,\D)  \arrow[r, "\simeq"{below}] &  \D_{\Lambda_{\infty,+}},
            \end{tikzcd}
        \end{equation}
        which can be regarded as an augmented stable homotopy-coherent variant of the paracyclic Dwyer-Kan correspondence
        outlined above. 
\item Restricting to {\em complete} radial analytic sheaves on $\CC^*$, analytic expansion provides an equivalence
        \begin{equation}
                \label{eq:cstarshriekeq}
                \begin{tikzcd}
                        \Shc^{*!}(\SPi(\CC);\D) \arrow[r, "\An"{above}, "\simeq"{below}] & \Shanr(\CC^*,\D)^c
                \end{tikzcd}
        \end{equation}
        with the subcategory of $*!$-sheaves on $\SPi(\CC)$. Combinatorially, this yields a paracyclic Dwyer-Kan correspondence
        \begin{equation}
           \begin{tikzcd}
                   \Fun(\bigstar_{n \in \NN} S^1,\D)  \arrow[r, "\simeq"{below}] &  \D_{\Lambda_{\infty}},
            \end{tikzcd}
    \end{equation}
        in the style of Lurie's Dold-Kan equivalence.

    \item Suppose that $\D$ comes equipped with a $t$-structure with heart
        $\A$. Then the equivalence \eqref{eq:cstarshriekeq} is $t$-exact with respect to the {\em perverse}
        $t$-structure. Its restriction to hearts recovers the paracyclic
        Dwyer-Kan correspondence
        \begin{equation}\label{eq:cstareq_heart}
            \begin{tikzcd}
                    \Perv^{*!}_{\NN}(\SPi(\CC); \A) \arrow[r, "\simeq"{above}] &  \A_{\Lambda_{\infty}} .
            \end{tikzcd}
        \end{equation}
        The terms of the paracyclic object corresponding to a perverse sheaf
        $F$ are equivalent to the vanishing cycles of $F$ at the origin $0$
        with respect to the power sums $z_1^n + z_2^n + \dots$ for varying $n \ge 0$. 
\end{enumerate}

\subsection{Outlook}

The following aspects will not be developed in detail in this work - we hope to
return to them in future work.

\begin{enumerate}

    \item {\bf Higher dimensions.} For $n \ge 2$, we obtain an analytic expansion functor
        \[
            \An: \Shc(\SP^{\infty}(\RR^{n+1}), \D) \lra \Shan(\RR^{n+1}\setminus\{0\}, \D)
        \]
        which will (rather clearly) not be an equivalence but may serve as a
        starting point for higher--dimensional Dold-Kan type correspondences. 

    \item {\bf Equivariance.} The equivalence \eqref{eq:cstareq} can be
        shown to be $\on{O}(2)$-equivariant with respect to the actions induced
        by the geometric action of $\on{O}(2)$ on $\RR^2 = \CC$. Given a
        homomorphism $G \to \on{O}(2)$, we obtain an equivalence
        \begin{equation}\label{eq:coheqdwyerkan}
            \Shc(\SP^{\infty}(\CC);\D)^G \simeq \Shanr(\CC^*; \D)^G
        \end{equation}
        between categories of $G$-equivariant sheaves. Further, the equivalence
        \eqref{eq:cstareq_heart} becomes
        \begin{equation}\label{eq:equi-dwyerkan}
            \Perv^G(\SP^{\infty},\A) \simeq \A_{\Lambda_{\infty}/G}
        \end{equation}
        identifying $G$-equivariant perverse sheaves with $\Lambda_{\infty}/G$-objects.
        For example, setting $G = \on{SO}(2)$, the category
        $\Lambda_{\infty}/\on{SO}(2)$ can be identified with Connes' cyclic category $\Lambda$.
        Working out the circle action in terms of the classifying data
        \eqref{eq:diag_perverse}, we recover the explicit combinatorial formulas from
        \cite{DK85} (also cf. \cite{heine:dk}). Varying the group $G$, we
        obtain Dold-Kan type equivalences for all planar crossed simplicial
        groups in the sense of \cite{DK15}.
\end{enumerate}

\subsection*{Funding}
The authors acknowledge support from the Deutsche Forschungsgemeinschaft (DFG,
German Research Foundation) under Germany's Excellence Strategy - EXC 2121
``Quantum Universe'' - 390833306 and the Collaborative Research Center - SFB
1624 ``Higher structures, moduli spaces and integrability'' - 506632645.

\section*{Glossary and conventions}

\begin{itemize}
    \item $M$ generic manifold
    \item $\DD(M)$ poset of finite, possibly empty, disjoint unions of open disks in $M$
    \item $\DD_{\le k}(M) \subset \DD(M)$ unions of cardinality $\le k$
    \item $\DD^*(M) = \DD(M) \setminus \emptyset$
    \item $N$ manifold
    \item $C = C(N)$ open cone over $N$ with tip $*$
    \item $C^*$ punctured cone
    \item $\U(-)$ poset of open subsets of a topological space
\end{itemize}

\begin{itemize}
    \item We call a functor $f: \C \to \D$ of $\infty$-categories {\em
        cofinal} if, for every $y \in \D$, the category $f_{y/}$ is weakly
        contractible. In particular, $f: \Delta^0 \to \D$ is cofinal if an
        only if $f(0)$ is a final object of $\D$. 
    \item We call a functor $f: \C \to \D$ of $\infty$-categories {\em
        coinitial} if, for every $y \in \D$, the category $f_{/y}$ is weakly
        contractible. In particular, $f: \Delta^0 \to \D$ is coinitial if an only if
        $f(0)$ is an initial object of $\D$. 
\end{itemize}

\section{Analytic sheaves from constructible sheaves on symmetric products}

We will now introduce the sheaf--theoretic concepts needed to set up the
geometric perspective on Dold-Kan type correspondences outlined in the
introduction. 

\subsection{Analytic and polynomial sheaves on manifolds}
\label{sub:analytic_polynomial}

We collect some ingredients from the Goodwillie-Weiss calculus of functors (e.g.
\cite{arone-ching:handbook}), adapted to our purposes. Let $M$
be a real $n$-dimensional topological manifold and let $\D$ be a pointed
$\infty$-category with limits. We denote by $\DD(M)$ the poset of finite
disjoint, possibly empty, unions of open disks (i.e. open subsets of $M$
homeomorphic to $\RR^n$).

\begin{defi}
    \label{defi:analytic_sheaf}
    A $\D$-valued {\em analytic sheaf} on $M$ is a functor
    \[
        G: \DD(M)^{\op} \lra \D
    \]
    such that, for every homotopy equivalence $U \subset V$ in $\DD(M)$, the
    associated map $G(V) \to G(U)$ is an equivalence in $\D$. We denote the
    $\infty$-category of analytic sheaves, defined as a full subcategory of
    $\Fun(\DD(M)^{\op},\D)$, by $\Shan(M;\D)$.
\end{defi}

\begin{rem}
    \label{rem:weiss}
    Denoting by $\U(M)$ the poset of all open subsets of $M$, we may extend an
    analytic sheaf $G$ by right Kan extension to a presheaf $\widetilde{G}$ on
    $\U(M)$. By the pointwise formula, this amounts to setting
    \begin{equation}\label{eq:weissdescent}
        \widetilde{G}(U) \simeq \lim G | (\DD(M)/U).
    \end{equation}
    We may view the open cover $\DD(M)/U$ as a {\em Weiss (hyper)cover} of
    $U$: every finite subset of $U$ is contained in an open $V \in
    \DD(M)/U$. The equivalence \eqref{eq:weissdescent} can then be interpreted
    as the statement that $\widetilde{G}$ satisfies (hyper)descent with respect
    to this cover. In particular, analytic sheaves in our sense, can be
    (tautologically) regarded as presheaves on $M$ satisfying Weiss
    (hyper)descent in the sense of the equivalences \eqref{eq:weissdescent}.
    Vice versa, we may start with a $\D$-valued presheaf $H$ on $\U(M)$,
    assuming that it maps homotopy equivalences to equivalences in $\D$. Then
    we obtain a canonical map
    \[
            H \lra \widetilde{H | \DD(M)}
    \]
    referred to as the {\em analytic approximation} of $H$. While very important in
    Goodwillie-Weiss calculus, this approximation will not play any role in
    this work -- we will always consider analytic sheaves to begin with.
\end{rem}

\begin{exas}
    \label{exa:real}
    The following examples are of particular interest to us.
    \begin{enumerate}
        \item $M = \RR$. Let $\Delta$ denote the simplex category and
            $\Delta_+$ its augmentation by an initial object. As shown in Proposition
            \ref{prop:intervals_localization}, the functor
            $\DD(\RR) \to \Delta_{+}, U \mapsto \pi_0(U)$ exhibits $\Delta_+$
            as a localization along the homotopy equivalences in $\DD(\RR)$.
            Therefore, pullback along the functor establishes an equivalence
            \[
                \D_{\Delta_+} \simeq \Shan(\RR;\D)
            \]
            identifying analytic sheaves on $\RR$ with augmented simplicial objects.

        \item $M = S^1$. The localization of $\DD(S^1)$ along the homotopy
            equivalences can be identified with the augmented paracyclic category
            $\Lambda_+$, see Proposition \ref{prop:s1_intervals_localization}. In particular, we obtain an equivalence
            \[
                \D_{\Lambda_+} \simeq \Shan(S^1;\D).
            \]
    \end{enumerate}
\end{exas}

Given a finite set $L$ and a cubical diagram
\[
    f: \P(L)^{\op} \to \D
\]
parameterized by the poset $\P(L)$ of subsets of $L$, we define the {\em total fiber}
\[
    \tfib(f) := \fib(f(L) \to \lim_{I \subsetneq L } f(I)).
\]

\begin{defi}
    \label{defi:derivative}
    Let $G \in \Shan(M;\D)$, $U \in \DD(M)$ and $D_1, \dots, D_l \subset M$
    pairwise disjoint open disks which are disjoint from $U$.
    We define the {\em $l$th difference cube} of $G$ at $U$ with respect to $D_1, \dots, D_l$ as
    \[
        \delta^l G (U; D_1, \dots, D_l) = \{G(U \cup D_I)\}_{I \subset \{1,\dots,l\}}
    \]
    where for $I \subset \{1,\dots,l\}$ we set $D_I = \bigcup_{i \in I} D_i$. We
    further refer to the total fiber of the $l$th difference cube as the the
    corresponding {\em $l$th total difference.}
\end{defi}

\begin{exa}
    \label{exa:first_and_second}
    The first difference cube is given by the arrow
    \[
        \delta G (U; D) = G(U \cup D) \to G(U)
    \]
    and the second difference cube is the square
    \[
    \begin{tikzcd}
        G(U \cup D_1 \cup D_2) \ar{d}\ar{r} & \ar{d}G(U \cup D_2)\\
        G(U \cup D_1) \ar{r} & G(U).
    \end{tikzcd}
    \]
\end{exa}

\begin{defi}
    \label{defi:polynomial}
    An analytic sheaf $G \in \Shan(M;\D)$ is called {\em polynomial of degree
    $\le k$} if all $l$th difference cubes
    \[
        \delta^l G (\emptyset; D_1, \dots, D_l)
    \]
    for $l > k$ are Cartesian (i.e. limit cones). We denote the full subcategory of $\Shan(M;\D)$ spanned
    by the polynomial sheaves of degree $\le k$ by $\Shk(M;\D)$.
\end{defi}

\begin{rem}
    \label{rem:cart}
    Note that, if a cubical diagram is a limit diagram, then its total fiber
    vanishes. The converse statement does not hold in general, but, for
    example, for stable $\infty$-categories.
\end{rem}

\begin{prop} \label{prop:kan} For an analytic sheaf $G$ on $M$ the following are
    equivalent:
    \begin{enumerate}[label = (\roman *)]
        \item $G$ is polynomial of degree $\le k$.
        \item $G$ is a right Kan extension of its restriction to the poset
            $\DD_{\le k}(M)$ of unions of at most $k$ disks.
    \end{enumerate}
\end{prop}
\begin{proof}
    For $U = D_1 \cup \dots \cup D_l$ with $l > k$ set $L = \{1,\dots,l\}$ and consider the cube
    \[
        q: \P(L)^{\op} \to \D,\; I \mapsto G(D_I).
    \]
    Using the functoriality of Kan extensions, we obtain that the cube $q$ is a
    right Kan extension of its restriction to the poset $\P(L)^{\le k}$ of
    subsets of cardinality less than $k+1$ if and only if it is obtained by
    successive right Kan extensions along the steps of the filtration
    \[
        \P(L)^{\le k} \subset \P(L)^{\le k+1} \subset \dots \subset \P(L).
    \]
    Via the pointwise criterion, this latter condition is equivalent to
    requiring that all $m$th difference cubes ($k < m \le l$) of $G$ at $\emptyset$
    with respect to $m$ disks among $D_1, \dots, D_l$ are limit cubes. The claimed
    equivalence now follows by letting $q$ range over all cubes of the given
    form.
\end{proof}

\begin{prop}
    \label{prop:allderivatives}
    Let $G \in \Shk(M;\D)$ be polynomial of degree $\le k$. Then, for every $U
    \in \DD(M)$, every $l > k$, and every collection $D_1, \dots, D_l$ of
    pairwise disjoint disks satisfying $D_i \cap U = \emptyset$, the $l$th difference cube
    \[
        \delta^l G (U; D_1, \dots, D_l)
    \]
    is Cartesian.
\end{prop}
\begin{proof}
        By Proposition \ref{prop:kan}, the anlytic sheaf $G$ is a right Kan extension of
        the restriction to $\DD_{\le k}(M)$. For the punctured cube $C_U = \{ U \cup D_I \}_{I \subsetneq \{1, \dots,l\}}$
        the inclusion of full subcategories
        \[
                C_U \to C_U \cup \DD_{\le k}(M)/(U \cup D_1 \cup \dots \cup D_l)
        \]
        of $\DD(M)$ is cofinal because $k < l$. The value of $G$ at $U \cup D_1 \cup \dots \cup D_l$
        can be computed by iterated right Kan extension, so the $l$th difference cube
        $\delta^l G (U; D_1, \dots, D_l)$ is biCartesian.

\end{proof}

\begin{rem}
    \label{rem:weissk}
    Continuing the discussion thread of Remark \ref{rem:weiss}, using
    Proposition \ref{prop:kan} along with the pointwise formula for Kan
    extensions, we obtain that, for a polynomial sheaf $G$ of degree $\le
    k$, the value of the presheaf $\widetilde{G}$ on $\U(M)$ at an open subset
    $U$ is given by the formula
    \begin{equation}\label{eq:weissk}
        \widetilde{G}(U) \simeq \lim G | \DD_{\le k}(M)/U.
    \end{equation}
    The open cover $\DD_{\le k}(M)/U$ is a {\em Weiss (hyper)cover of $U$ of degree
    $k$}: every finite subset of cardinality $k$ is contained in an open set of
    the cover. Again, tautologically, we may view polynomial sheaves of degree
    $\le k$ as presheaves on $M$ satisfying (hyper)descent in the sense of the
    equivalence \eqref{eq:weissk}.
\end{rem}

For $k \ge 0$, consider the embedding $i_k: \DD_{\le k}(M) \subset \DD(M)$. We
obtain a (co)localization functor
\[
    \tau_k: \Shan(M;\D) \to \Shk(M;\D),\; G \mapsto (i_k)_*(i_k)^* G
\]
which we refer to as the {\em $k$th polynomial approximation}.

\begin{prop}\label{prop:approx}
    The diagram
    \[
    \begin{tikzcd}
        \Shan(M;\D) \ar[swap]{dr}{\tau_k} \ar[bend left=5]{drr}{\tau_{k-1}} \ar[bend left=10]{drrrr}{\tau_0} & & & \\
        \cdots \ar{r}& \Shk(M;\D) \ar[swap]{r}{\tau_{k-1}} &  \Sh^{\le k-1}(M;\D) \ar{r}  & \cdots \ar{r} & \Sh^{ \le 0}(M;\D)
    \end{tikzcd}
    \]
    is a limit cone.
\end{prop}
\begin{proof}
    This amounts to the statement that, for every analytic sheaf $G \in \Shan(M;\D)$, we have
    \begin{equation}\label{eq:approx}
        G \overset{\simeq}{\lra} \lim_{k \ge 0} \tau_k(G).
    \end{equation}
    In other words, $G$ is the limit of its sequence of polynomial approximations. But this
    latter statement is clear, since it can be verified pointwise. When
    evaluating on $U \in \DD_{\le l}(M)$ then the sequence $\{\tau_k(G)\}_{k \ge
    0}$ stabilizes for $k > l$ with value $G(U)$.
\end{proof}

\subsection{Constructible sheaves on symmetric products}

Let $N$ be a compact real topological manifold and let
\[
    C = C(N) =  \{*\} \amalg_{N \times \{0\}} N \times \RR_{\ge 0}
\]
denote its open cone. For $k \ge 0$, we denote by
\[
    \SP^k(C) = \overbrace{C \times \dots \times C}^k / S_k
\]
the $k$th symmetric product of $C$. We regard $C$ as pointed by the tip $* \in
C$ of the cone. This gives rise to a sequence of embeddings
\[
    \{*\} \subset C \subset \SP^2(C) \subset \SP^3(C) \subset  \dots
\]
given by adding the base point, and we define
\[
    \SPi(C) = \bigcup_{k \ge 0 } \SP^k(C).
\]
We abbreviate $X := \SPi(C)$ and keep track of the filtration $X_{\le k}
= \SP^k(C) \subset X$, $k \ge 0$. The space $X$ can thus be viewed as an
$\NN$-stratified space with strata $X_k = X_{\le k} \setminus X_{\le k-1}$, $k
\in \NN$. We call this stratification the {\em filtration stratification}.

We further set $C^* = C \setminus \{*\}$ to be the complement of the
tip of the cone.

\begin{exas}
    \label{exas:spks}
    \begin{enumerate}
        \item For $N = S^1$, we identify $\CC = C(S^1)$. The elementary symmetric
            polynomials give rise to a homeomorphism
            \begin{equation}\label{eq:elementary}
                \SP^k(\CC) \overset{\cong}{\lra} \CC^k,\; (z_1, \dots, z_k) \mapsto (e_1, \dots, e_k)
            \end{equation}
            parameterizing the space of monic complex polynomials of degree $k$ in
            terms of their zeros ($\SP^k(\CC)$) or their coefficients ($\CC^k$).
            Interpreting the symmetric product $\SP^k(\CC)$ as the space of monic
            polynomials of degree at most $k$ in the variable $T$, the inclusion $\SP^k(\CC)
            \subset \SP^{k+1}(\CC)$ is given by multiplication with $T$. Under the
            identification \eqref{eq:elementary}, this corresponds to the inclusion
            $\CC^k \subset \CC^{k+1}$ obtained by adjoining $0$ as the last coordinate.
            The $k$-stratum of the $\NN$-stratification of $\SPi(\CC)$ is therefore
            given by the space of monic polynomials of degree $k$ which, under
            \eqref{eq:elementary} gets identified with $\CC^{k-1} \times \CC^*$. In
            particular, all strata are homotopy equivalent to $S^1$.
        \item For $N = \pt$, we obtain $C = \RR_{\ge 0}$. We may identify
            \[
                \SP^k(\RR_{\ge 0}) = \{ (x_1, \dots, x_k) \in \RR^k \; | \; x_1 \ge x_2 \ge \dots \ge x_k \ge 0 \}
            \]
            so that we may view this space as the complement of the face
            opposite to the vertex $0$ of a geometric $k$-simplex (i.e. an open
            corner of $0$). Again, the inclusion is given by setting the last
            coordinate to zero, which amounts to the inclusion of the face
            opposite to the vertex $k+1$. In particular, the $k$-stratum of
            $\SPi(\RR_{\ge 0})$ can be identified with the complement
            of the faces $\partial_0$ and $\partial_{k}$ of a geometric
            $k$-simplex, in particular, it is contractible.
        \item In the general situation, for $C = C(N)$, setting $C^*$ to be the
            complement of the cone point, the $k$-stratum of the stratification
            of $\SPi(C)$ is given by $\SP^k(C^*)$ which deformation
            retracts to $\SP^k(N)$. Note, in particular, that the discriminant
            stratification of $\SP^k(C)$ with open stratum given by the
            configuraion space is {\em not} reflected in our
            $\NN$-stratification.
    \end{enumerate}
\end{exas}

The symmetric product $X = \SPi(C)$ comes equipped with a commutative topological
monoid structure which, when viewing the class in $\SP^k(C)$ of a tuple $(x_1,
\dots, x_k) \in C^k$ as a formal sum $x_1 + \dots + x_k$, is given by formally
adding such expressions. This monoid structure can be used to conveniently
describe open subsets of $X$ which will be of crucial importance for us.
For example, given an open subset $U \subset C^*$, we denote by
\begin{equation}\label{eq:upx}
    U + X \subset X
\end{equation}
the open subset consisting of formal sums with at least one summand in $U$.

We also write
\[
	m_1 \cdot y_1 + \dots + m_r \cdot y_r
\]
for an element of $\SP^k(C)$ where the $y_i$ are distinct and $\sum_{i=1}^r m_i = k$.

We introduce a refinement of the filtration stratification which we shall require in \S \ref{sec:spr}.
\begin{defi}
Let $\PP$ denote the set of (unordered) partitions of natural numbers, i.e. decompositions $k=\sum_{i \in I} m_i$ with $I$ finite and
$m_i \in \NN_{>0}$. There is a natural poset structure on $\PP$ as follows:
\begin{itemize}
       \item If $\sum_{i \in I} m_i =k = \sum_{j \in J} n_j$, then there is an
               arrow $\sum_{i \in I} m_i \to \sum_{j \in J} n_j$ if and only if  $J$ admits a partition
               $J= \bigcup_{i \in I} J_i$ such that $\sum_{j \in J_i} n_j = m_i$. For instance, there is an arrow $2+2 \to 1+1+2$.
\item If $\sum_{i \in I} m_i =k < l = \sum_{j \in J} n_j$, then there is an arrow
        $\sum_{i \in I} m_i \to \sum_{j \in J} n_j$ if and only if the $(|I|+1)$-term partition $l= \sum_{i \in I} m_i +(l-k)$ can
        be obtained from $\sum_{j \in J} n_j$ as above.
\end{itemize}
The {\em discriminant stratification} $\SPi(C) \to \PP$ is given by
\[
        \sum_{i \in I} m_i \cdot x_i \mapsto \sum_{i \in I} m_i
\]
for $x_i \in C^*$ distinct and $m_i >0$.
\end{defi}
There is a natural map of posets $\PP \to \NN$ given by $k=\sum_{i \in I} m_i \mapsto k$. We denote by $\PP_k$ and $\PP_{\le k}$
the preimages of $k$ and $[k]=\{1,\dots,k\}$ respectively.

We now fix $n \in \NN$ and recall the structure of a (smooth) conical manifold on $\SP^{n}(C)$ (c.f. \S \ref{sub:conical_manifolds}).
\begin{prop}
\label{prop:smooth_structures}
Suppose that $N$ is a smooth compact manifold and let $C=C(N)$ denote the open cone.
The discriminant stratification
\[
	\SP^{n}(C) \to \PP
	\]
	naturally admits the structure of a smooth conical manifold.
\end{prop}
\begin{proof}
        Applying Proposition 3.5.18 of \cite{aft17} to the conical manifold $C(N) \to [1]$ yields a smooth conical
structure on an ordered version of the discrimant stratification and Example 3.5.15
therein shows that this structure descends along the action of the permutation group to a smooth structure on the discrimant
stratification itself.
\end{proof}
\begin{rem}
        \label{rem:basics}
A more precise examination of the arguments in \cite{aft17} allows an explicit description of a basis of conical disks of $\SP^n(C)$:
Such a basis $\B$ of $\SP^n(C)$ is given by sums
\[
        \sum_{i=0}^k \SP^{m_i}(D_i) =\{\sum_{i=1}^n x_i |  D_j \text{ contains
        $m_j$ of the $x_j$ for all $0 \le j \le k$} \}
\]
such that
\begin{itemize}
       \item $\sum_{i=0}^k m_i =n$ and $m_i>0$ for $i>0$,
        \item $D_1,\dots,D_k \in \DD_1(C^*)$ are pairwise disjoint smooth disks,
        \item if $m_0 >0$ then $D_0=C_{\epsilon}(N)\subset C$ is a cone disjoint from all the $D_i$.
\end{itemize}
(We allow $m_0=0$ to make the notation uniform for all basics.)

More generally, we also obtain a basis if we allow only smooth disks $D_i$ from some basis of smooth disks of $C^*$.
In particular, if $C^*$ admits an intersection-stable basis of smooth disks,
this yields an intersection-stable basis of $\SP^n(C)$.
\end{rem}
Let $\D$ be an $\infty$-category which is either bicomplete and stable or of the large class of presentable categories
described in Remark \ref{rem:non-stable-coefficients}.

From the general theory of conical manifolds, we obtain:
\begin{prop}
        \label{prop:discriminant_exodromy}
        Restriction induces an equivalence
        \[
                \Sh_{\PP}(\SP^n(C);\D) \to \Fun_{\PP}(\B^{\op},\D)
        \]
        between $\PP$-constructible sheaves (c.f. Definition \ref{defi:constructible_sheaf}) on $\SP^{n}(C)$
        and $\PP$-constructible presheaves (c.f. Definition \ref{defi:constructible_presheaf})
        on the intersection-stable basis $\B$ of Remark \ref{rem:basics}.

        The inverse is given by right Kan extension.
\end{prop}
\begin{proof}
        This is an application of the general exodromy statement Theorem \ref{thm:exodromy} for conical manifolds.
\end{proof}
There is a similar statement for the filtration stratification.
\begin{thm}\label{thm:filtration_exodromy}
        An inclusion $B \subset B'$ with $B,B' \in \B$ is an $\NN$-stratified homotopy equivalence if and only if they have the
        same minimal stratum over $\NN$.

        Restriction induces an equivalence
	\[
                \Sh_{\NN}(\SP^n(C);\D) \to \Fun_{\NN}(\B^{\op},\D).
	\]
        is an equivalence for the intersection-stable basis $\B$ of Remark \ref{rem:basics}.

        The inverse is given by right Kan extension.
\end{thm}
\begin{proof}
        The first part follows immediately from Lemma \ref{lem:filtration_basics_homotopy}.
        The coarsening $\SP^n(C) \to \PP \to \NN$ thus satisfies the additional hypothesis of Proposition
        \ref{prop:coarsen_stratification_cons_sheaves}, hence the equivalence.
\end{proof}
\begin{rem}
The filtration stratification $\SP^{n}(C) \to \NN$ does not generally itself make $\SP^n(C)$ a conical manifold,
though it is conical in the weaker sense of \cite[A.5.5]{lurie:ha}.
For instance, we have seen that the $2$-stratum of $\SP^2(\Rg)$ has a nonempty boundary given by the discriminant.
\end{rem}
\subsection{Analytic expansion}
In this subsection, let $\D$ be a pointed coefficient category (bicomplete and stable or of the class
in Remark \ref{rem:non-stable-coefficients}).

The following definition is of central importance for this work. 
\begin{defi}
    \label{defi:analytic_expansion}
    Let $X = \SPi(C)$ and let $F$ be a $\D$-valued $\NN$-constructible
    sheaf on $X$. We define the {\em analytic expansion} of $F$ to be the
    presheaf
    \[
        \Fan: \DD(C^*)^{\op} \lra \D, \; U \mapsto \fib(F(X) \to F(U + X)).
    \]
    where the fiber is taken along the restriction map corresponding to the
    open subset $U + X \subset X$ defined in \eqref{eq:upx}.
\end{defi}

\begin{exa}
    \label{exa:realanalyticexplansion}
    Let $C = C(\pt) = \RR_{\ge 0}$. An $\NN$-constructible sheaf $F$ on
    $\SPi(\RR_{\ge 0})$ is given, in terms of its exit path representation (c.f.
    Proposition \ref{prop:exit_paths_real_case}), as a sequence
    \[
        F_0 \to F_1 \to F_2 \to \dots
    \]
    The value of the analytic expansion $\Fan$ on the union $D_1 \cup \dots \cup
    D_k \in \DD(\RR_{>0})$ of $k$ disjoint intervals is equivalent to the
    totalization of the cube
    \[
        \P(\{1,\dots,k\})^{\op} \to \D, \; I \mapsto F_{|I|}.
    \]
\end{exa}

\begin{thm}
    \label{thm:polytopoly}
    Let $F \in \Shc(\SPi(C); \D)$. Then the following hold:
    \begin{enumerate}
        \item The presheaf $F^{\an}$ is an analytic sheaf.
        \item If $F$ has support on $\SP^{k}(C)$ then $F^{\an}$ is polynomial
            of degree $\le k$.
    \end{enumerate}
\end{thm}
\begin{proof}
    Let $U = D_1 \cup D_2 \cup \dots \cup D_l$ be an open in $\DD(C^*)$. We set $L = \{1, \dots, l\}$ and
    compute $F(U + X)$ via descent for the open cover $\{D_i + X\}_{i \in L}$. For $I \subset L$, we set
    \begin{equation}
        D^{\Sigma I} := \sum_{i \in I} D_i \quad \subset \SP^{|I|}(C)
    \end{equation}
    and note
    \[
        \bigcap_{i \in I} (D_i + X) = D^{\Sigma I} + X
    \]
    obtaining
    \begin{equation}\label{eq:coveringsym}
        F(U + X) \simeq \lim \{F(D^{\Sigma I} + X)\}_{\emptyset \neq I \in \P(L)}
    \end{equation}
    We thus obtain
    \[
        \Fan(U) = \fib(F(X) \to F(U + X)) \simeq \tfib \{F(D^{\Sigma I} + X)\}_{\emptyset \neq I \in \P(L)}
    \]
    where we set $D_{\emptyset} = X$. Now let $I \subset L$ with $|I| = k$ and denote $i_k: \SP^k(C) \to \SP^{\infty}(C)$.
    By Lemma \ref{lem:fundamental}, the restriction map
    \begin{equation}\label{eq:reseq}
        F(D^{\Sigma I} + X) \simeq i_k^* F(D^{\Sigma I})
    \end{equation}
    is an equivalence. Therefore, since $D^{\Sigma I} \subset \SP^k(C)$ is an open
    subset of the open stratum $\SP^k(C^*)$, and $F$ is assumed constructible,
    we deduce that $\Fan$ is analytic.

    Now suppose that $F$ has support on $\SP^k(C)$, suppose $l > k$ and, for $J
    \subset L$, set
    \[
        D_J = \bigcup_{i \in J} D_i \subset C^*.
    \]
    We need to show that the cube
    \[
        \{\Fan(D_J)\}_{J \in \P(L)^{\op}}
    \]
    is Cartesian which follows from Lemma \ref{lem:cubeinversion}, since it is
    the inverted cube corresponding to the cube 
    \[
        \{F(D^{\Sigma I} + X)\}_{I \in \P(L)}.
    \]
    Note that the hypothesis of the lemma is satisfied since $F$ vanishes on
    $\SP^l(C^*)$, as this guarantees that 
    \[
        F(D_1 + \dots + D_l + X) \simeq 0
    \]
    (cf. \eqref{eq:reseq}).
\end{proof}

\begin{lem}[(Cube inversion)]\label{lem:cubeinversion}
    Let $L$ be a finite set and denote by $\P(L)$ the poset of subsets of $L$.
    Let $\D$ be an $\infty$-category with finite limits and let
    \[
        q: \P(L) \to \D
    \]
    be a cube in $\D$ such that $q(L) \simeq *$. We define the {\em inverted cube}
    \[
        \check{q}: \P(L)^{\op} \to \D, \; J \mapsto \tfib q|\P(L)/J
    \]
    where $\P(L)/J$ denotes the subposet consisting of subsets of $L$ that are
    contained in $J$. Then the cube $\check{q}$ is Cartesian.
\end{lem}
\begin{proof}
    The claim reduces to the statement that the limit of the twice-punctured
    cube $q|T$ where $T = \P(L) \setminus \{\emptyset, L\}$ can be computed as
    a limit of limits by covering $T$ via
    \[
        \P(L) \to \Cat/\P(L), J \mapsto \{\emptyset \neq I \subset J \}
    \]
    and applying \cite[4.2.10]{lurie:htt}.
\end{proof}

\begin{rem}
    \label{rem:informal}
    Lemma \ref{lem:cubeinversion} states that the total fiber of the cube $q$
    (which is the initial vertex of the inverted cube $\check{q}$) can be
    computed as a limit of total fibers over all proper faces of the cube
    $\P(L)$ that contain the initial vertex. If $q$ does not satisfy the
    condition $q(L) \simeq 0$, then we may form a fiber sequence of cubes
    \[
        q' \to q \to q''
    \]
    where $q''$ is the cube $q$ but with $q(L)$ replaced by $0$, so that $q'$
    is a cube whose only nonzero value is $q'(L) \simeq q(L)$. Passing to total fibers, we obtain a fiber sequence
    \[
        \Omega^{|L|} q(L) \to \tfib(q) \to \tfib(q'')
    \]
    where $\tfib(q'')$ can be computed as a limit of total fibers using Lemma
    \ref{lem:cubeinversion}. 
\end{rem}

\begin{lem}
    \label{lem:fundamental}
    Let $F$ be a constructible sheaf on $\SPi(C)$. Let $D_1, D_2, \dots, D_k$ be
    pairwise disjoint open disks in $C^*$ and $i: \SP^k(C) \subset \SPi(C)$. Then the restriction map
    \[
        F(D_1 + \dots + D_k + X) \simeq i^*F(D_1 + \dots + D_k)
    \]
    is an equivalence in $\D$.
\end{lem}
\begin{proof}
        This is a consequence of Lemma \ref{lem:extend_zero_disk} and homotopy invariance Proposition
        \ref{prop:homotopy_equivalence_constructible_sheaves}.
\end{proof}

\subsection{Complete analytic sheaves}%
\label{sub:complete_analytic}

Note that everything we have discussed so far holds in the ``unstable'' context
for e.g. the categories of pointed spaces or $\infty$-categories (c.f. Remark \ref{rem:non-stable-coefficients}).
In particular, Theorem \ref{thm:polytopoly} holds in that situation. For the
remainder of this paper, we focus on the case where $\D$ is stable (and
bicomplete). This will be the standing assumption in what follows. In this
section, we document various basic statements that hold under this assumption.

\begin{rem}
    \label{rem:tfib_stable}
    For stable coefficients $\D$, in the context of Definition
    \ref{defi:derivative}, we have that an $l$th difference cube is Cartesian
    if and only if its total fiber, referred to as the corresponding {\em $l$th
    total difference} vanishes. 
\end{rem}

\begin{defi}
    \label{defi:complete_analytic}
    An analytic sheaf $G$ on $M$ is called {\em complete} if the corresponding
    diagram
    \[
        G: \DD(M)^{\op} \to \D
    \]
    is a colimit cone, exhibiting $G(\emptyset)$ as the colimit of $G|\DD^*(M)$.
    We denote the full subcategory of $\Shan(M;\D)$
    consisting of complete analytic sheaves by $\Shan(M;\D)^c$.
\end{defi}

\begin{prop}
    \label{prop:stable_complete}
    Assume that $\D$ is stable. Then every $\D$-valued polynomial sheaf on $M$
    is complete.
\end{prop}
\begin{proof}
    Switching from $\D$ to $\D^{\op}$, we may, more conveniently, view an
    analytic sheaf as a covariant functor $G: \DD(M) \to \D$. Such an analytic
    sheaf is polynomial of degree $\le k$ if, for every pairwise disjoint
    collection $D_1, \dots, D_l \in \DD(M)$ of disks, $l>k$, the cube
    \[
        q: \P(\{1,\dots,l\}) \to \D,\; I \mapsto G(D_I)
    \]
    is biCartesian, in particular, it exhibits $G(\emptyset)$ as the limit of
    the punctured cube. We need to show that $G(\emptyset)$ is the limit over
    all of $\DD^* = \DD(M) \setminus \emptyset$. To this end, we define, for a finite
    subset $T \subset M$, the subposet $\DD_T \subset \DD(M)$ consisting of $U
    \in \DD(M)$ such that $U \cap T \neq \emptyset$. Let $J$ denote the
    poset of finite subsets of $M$ of cardinality $>k$. Covering the diagram
    $\DD^*$ by $J$ via $T \mapsto \DD_T$, we may compute the limit over
    $\DD^*$ as a limit of limits over the various posets $\DD_T$. The poset $J$
    is contractible (since it is filtered) and we claim that the limits of all
    restrictions $G|\DD_T$ are equivalent to $G(\emptyset)$. Indeed, $\DD_T$ is
    the union of the subposets $\DD_{\{t\}}$, $t \in T$, so that we may compute
    the limit over $\DD_T$ as the limit of the punctured cube
    \[
        \P(T) \to \D, I \mapsto \lim G|\bigcap_{t \in I} \DD_{\{t\}}.
    \]
    Let $D_1, \dots, D_l$ be a collection of disjoint disks such that each
    element of $T$ is contained in precisely one of them. Then a simple
    coinitiality argument (using that $G$ is analytic) implies that the the
    punctured cube is equivalent to $I \mapsto G(D_I)$ so that its limit is
    $G(\emptyset)$ as claimed.
\end{proof}

\subsection{The cup-cap equivalence}

Let $k \le n$ and let $D_1, \dots, D_k$ be disjoint open disks in $C^*$. There are two assoctiated open subsets of $\SP^n(C)$: 
\begin{enumerate}[label=(\Roman *)]
    \item\label{type:1} the open $D_1 \cup  D_2 \cup \dots \cup D_k + \SP^{n-1}(C) \subset \SP^n(C)$,
    \item\label{type:2} the open $D_1 + D_2 + \dots + D_k + \SP^{n-k}(C) \subset \SP^n(C)$.
\end{enumerate}
The main result of this section, Theorem \ref{thm:stable_interchange}, implies
that the values of a sheaf $F \in \Sh(\SP^n(C);\D)$ on opens of type \ref{type:1} uniquely determine its
values on opens of type \ref{type:2} and vice versa. This phenomenon lies at
the heart of our argument to establish the equivalences described in \S
\ref{sub:r} and \S \ref{sub:c}. The result only holds for a {\em stable}
coefficient category $\D$ which explains why we cannot expect analytic expansion
to establish an equivalence for unstable coefficients.\\

For reasons of flexibility, we will formulate our result for an arbitrary subposet $\BB \subset \DD^*_{\le n}(C^*)$, which need only satisfy the following {\em decomposability} condition:
\begin{equation}
        \label{eq:partition_closure}
        \text{If } \bigcup_{i \in S}D_i \in \BB \text{ for } S \text{ finite then }
        \bigcup_{i \in T}D_i \in \BB\text{ for all }\emptyset \neq T \subset S.
\end{equation}
(Here and in the following we use the convention that the $D_i$ in a term $\bigcup_{i \in S}D_i$ are pairwise disjoint disks.)
Given such a poset, we consider the two families of opens in the symmetric product $\SP^n(C)$ of the above types:
\[
	\begin{split}
                \BB_{\cup}&= \{\bigcup_{i \in S} D_i+\SP^{n-1}(C) | \bigcup_{i \in S} D_i \in \BB \} \cup \{\emptyset\}, \\
                \BB_{\cap}&= \{\bigcap_{i \in S} (D_i+\SP^{n-1}(C)) | \bigcup_{i \in S} D_i \in \BB  \} \cup \{\SP^n(C)\}.
	\end{split}
\]

The map $U \mapsto U+\SP^{n-1}(C)$ furnishes an isomorphism of posets $\BB \cup \{\emptyset\} \cong \BB_{\cup}$, whose inverse image is given by restricting to the $1$-stratum.
Note again that for an element of $\BB_{\cap}$ we have
\[
        \bigcap_{i \in S} (D_i+\SP^{n-1}(C)) = D^{\Sigma S} + \SP^{n-|S|}(C)
\]
For $\bigcup_{i \in S} D_i \in \BB$, let
\[
        \{D^{\Sigma I} + \SP^{n-|I|}(C)\}_{\emptyset \neq I \subset S}
\]
denote the punctured \v{C}ech cube associated to the covering of
$\bigcup_{i \in S} D_i +\SP^{n-1}(C)$ by the opens $D_i+\SP^{n-1}(C)$. By property
\eqref{eq:partition_closure} this indeed takes values in $(\BB_{\cup} \cup
\BB_{\cap})^{\op}$.
For a stable $\infty$-category $\D$, we denote by 
\[
        \Fun((\BB_{\cup} \cup \BB_{\cap})^{\op},\D)^{\square} \subset \Fun((\BB_{\cup} \cup \BB_{\cap})^{\op},\D)
\]
the full subcategory consisting of presheaves $F$ satisfying $F(\emptyset) \simeq 0$ and the following
condition: for every $\bigcup_{i=1}^k D_i \in \BB$, we have
\[
        F(\bigcup_{i=1}^k D_i +\SP^{n-1}(C)) \simeq \lim \{F(D^{\Sigma I} +
        \SP^{n-|I|}(C))\}_{\emptyset \neq I \subset \{1,\dots,k\}}
\]
or, put differently, the full \v{C}ech cube associated to the open covering $\{D_i+\SP^{n-1}(C)\}$ is biCartesian.

\begin{thm}[(Cup-cap equivalence)]\label{thm:stable_interchange}
	Let $\D$ be a stable $\infty$-category. Then the pullback functors 
    \[\begin{tikzcd}
        {\Fun(\BB_{\cap}^{\op},\D)} & {\Fun((\BB_{\cup} \cup \BB_{\cap})^{\op},\D)^{\square}}
                                    & {\Fun((\BB_{\cup}^{\triangleleft})^{\op},\D)^{\emptyset}}
            \arrow["r"', from=1-2, to=1-1]
            \arrow["s", from=1-2, to=1-3]
    \end{tikzcd}\]
    form a span of trivial Kan fibrations. Here, $\BB_{\cup}^{\triangleleft} = \BB_{\cup} \cup \{\SP^n(C)\}$ and
$\Fun((\BB_{\cup}^{\triangleleft})^{\op},\D)^{\emptyset} \subset \Fun((\BB_{\cup}^{\triangleright})^{\op},\D)$
is the full subcategory of presheaves with $F(\emptyset) \simeq 0$.

The functors restrict to equivalences on $\NN$-constructible presheaves
(c.f. Definition \ref{defi:constructible_presheaf}).

\end{thm}
\begin{proof}
        We first prove that $r$ is a trivial Kan fibration by showing that the functors
        in $\Fun((\BB_{\cup} \cup \BB_{\cap})^{\op},\D)^{\square}$ are precisely the right Kan
        extensions of functors in $\Fun(\BB_{\cap}^{\op},\D)$. For this it suffices to show that, for every union of disjoint disks
        $U = \bigcup_{i \in S} D_i +\SP^{n-1}(C) \in \BB_{\cup}$ the punctured cube
\[
        \{F(D^{\Sigma I} + \SP^{n-|I|}(C))\}_{\emptyset \neq I \subset S}
\]
        is coinitial in the comma category $U/\BB_{\cap}^{\op}$:
        For
        \[
                W^{\Sigma J}+\SP^{n-|J|}(C) \in U/\BB_{\cap}^{\op}
        \]
        we show that $W_j \subset D_i$ for some pair of indices $i \in S, j \in J$. Assume, to the contrary, that $W_j \nsubset D_i$
        for all $i \in I, j \in J$. By connectedness of the $W_j$, we also have $W_j \nsubset \bigcup_{i \in I}D_i$, so there exists
        $y_j \in W_j \setminus \bigcup_{i \in S} D_i$ for all $j \in J$. But then
        \[
                \sum_{j=1}^{l}y_j+(n-k-1) \cdot * \in (W^{\Sigma J }+\SP^{n-|J|}(C) )\setminus U,
        \]
        a contradiction.

        The claim about coinitiality now follows from Quillen's Theorem A since the relevant overcategories are filtered.

        Secondly, we show that $s$ is a trivial Kan fibration by showing that the functors in
        $\Fun((\BB_{\cup} \cup \BB_{\cap})^{\op},\D)^{\square}$ are those arising from functors in
        $\Fun((\BB_{\cup}^{\triangleright})^{\op},\D)^{\emptyset}$ via left Kan extension.
        To this end, we first define, for $1 \le m \le n$, the subcategory $(\BB_{\cap})_{\le m} \subset \BB_{\cap}$ as the poset
        of opens $ D^{\Sigma I} + \SP^{n-|I|}(C)) \in \BB_{\cap}$ with $|I| \le m$.
        There is then a tower of inclusions
        \[
                \BB_{\cup}^{\triangleright} = \BB_{\cup} \cup (\BB_{\cap})_{\le 1} \subset \BB_{\cup} \cup (\BB_{\cap})_{\le 2} \subseteq
                \dots \subset \BB_{\cup}\cup (\BB_{\cap})_{\le n} = \BB_{\cup} \cup \BB_{\cap}.
        \]
        We show that, for $1 < m \le n$, a functor $F: (\BB_{\cup} \cup (\BB_{\cap})_{\le m})^{\op} \to \D$ is a left
        Kan extension of its restriction to $(\BB_{\cup} \cup (\BB_{\cap})_{\le m-1})^{\op}$ if and only if
        \[
                F(\bigcup_{i \in S}D_i +\SP^{n-1}(C)) \simeq \lim \{F(D^{\Sigma I} +
                \SP^{n-|I|}(C))\}_{\emptyset \neq I \subset \{1,\dots,m\}}
        \]
        for any $\bigcup_{i \in S}D_i \in \BB \cap \DD_m(C)$.
        By stability of $\D$, this is equivalent to
        \[
                F(D^{\Sigma S} + \SP^{n-m}(C)) \simeq \colim (\{\bigcup_{i=1}^{m}D_i +\SP^{n-1}(C) \} \cup \{F(D^{\Sigma I} +
                \SP^{n-|I|}(C))\}_{\emptyset \neq I \subsetneq \{1,\dots,m\}})
        \]
        To establish this, we show that the punctured cube is cofinal in
        $(\BB_{\cup} \cup (\BB_{\cap})_{\le m})^{\op}/(D^{\Sigma S} + \SP^{n-m}(C))$.
        It again suffices to show that any object of the comma category lies above an object of the tip-less cube. Given
        $\bigcup_{j \in J}V_j \in \BB_{\le m}$ such that
        \[
                D^{\Sigma S}+\SP^{n-m}(C) \subset V^{\Sigma J}+\SP^{n-|J|}(C),
        \]
        Lemma \ref{lem:intersection_inclusions} yields        \[
                D^{\Sigma I}+\SP^{n-|J|}(C) \subset V^{\Sigma |J|}+\SP^{n-|J|}(C).
        \]
        for some $I \subset S$ with $|I|=|J|$.
        Since Kan extensions compose, the above implies that the functors in $\Fun((\BB_{\cup} \cup \BB_{\cap})^{\op},\D)^{\square}$
        are the left Kan extensions along $(\BB_{\cup}^{\triangleright})^{\op} \to (\BB_{\cup} \cup \BB_{\cap})^{\op}$ of
        presheaves with $F(\emptyset) \simeq 0$.

         For both restrictions, we have shown that the value of a Kan extension $F$ at the new objects is determined by a biCartesian
        \v{C}ech cube. The claim about constructible presheaves therefore follows from Lemmas \ref{lem:intersection_inclusions}
        and \ref{lem:union_inclusion} and homotopy invariance of limits and colimits.
\end{proof}
Finally, we normalize:
\begin{lem}\label{lem:normalization}
        The normalization functor
\[
        N:\Fun((\BB_{\cup}^{\triangleright})^{\op},\D)^{\emptyset} \to \Fun(\BB_{\cup}^{\op},\D)
\]
given by
\[
        F \mapsto (U \mapsto \fib(F(\SP^n(C)) \to F(U)
\]
is an equivalence which respects $\NN$-constructible presheaves.
\end{lem}
\begin{proof}
        An inverse is clearly given by the map
        \[
                F \mapsto (U \mapsto \cof(F(U) \to F(\emptyset)))
        \]
        with $F(\SP^{n}(C) \simeq 0$ by left Kan extension.
        The statement about constructibility follows immediately from the characterization in Lemma \ref{lem:union_inclusion}.
\end{proof}
\begin{cor}
        \label{cor:analytification_equivalence_criterion}
        For a decomposable poset $\BB \subset \DD^*_{\le n}(C^*)$ of nonempty unions of disks the restricted analytic expansion
        functor
\[
        \AA_{\BB}:\Sh_{\NN}(\SP^n(C);\D) \to \Fun_{\NN}((\BB^{\triangleright})^{\op},\D)
\]
given by
\[
        \AA_{\BB}(F)(U)= \fib(F(\SP^n(C)) \to F(U + \SP^{n-1}(C)))
\]
is an equivalence if and only if the restriction functor
\[
        \Sh_{\NN}(\SP^n(C);\D) \to \Fun_{\NN}(\BB_{\cap}^{\op},\D)
\]
is an equivalence.
\end{cor}
\begin{proof}
        This is an immediate consequence of Lemma \ref{lem:normalization} and Theorem \ref{thm:stable_interchange} because
        $\AA_{\BB}$ may be identified with the composite
\[\begin{tikzcd}
        {\Sh_{\NN}(\SP^n(C);\D)} & {\Fun((\BB_{\cup} \cup \BB_{\cap})^{\op},\D)^{\square}} & {\Fun((\BB_{\cup}^{\triangleright})^{\op},\D)^{\emptyset} } & {\Fun(\BB_{\cup}^{\op},\D).}
	\arrow[from=1-1, to=1-2]
	\arrow[from=1-2, to=1-3]
	\arrow["N", from=1-3, to=1-4]
\end{tikzcd}\]
\end{proof}

\section{\texorpdfstring{$\SPi(\Rg)$}{spirg} and the simplicial Dold-Kan correspondence}
\label{sec:spr}

In this section, we let $\D$ be a bicomplete {\em stable} $\infty$-category and
discuss how the analytic expansion functor is related to the simplicial Dold-Kan
correspondence. 

Let $C = C(\pt) = \Rg$. 
\subsection{The results}
\label{sub:spr_results}

We state the main results of \S \ref{sec:spr}.

\begin{thm}
    \label{thm:real_analytic}
    Let $\D$ be a stable bicomplete $\infty$-category.
    \begin{enumerate}
        \item\label{it:remain1} Analytic expansion induces a tower of equivalences
            \begin{equation}\label{eq:real_augmented_thm}
                 \begin{tikzcd}
                     \Shc^*(\SPi(\Rg); \D) \arrow[r, "\An"{above}, "\simeq"{below}] &  \Shan(\Rg, \D) \\
                     \Shc^!(\SPi(\Rg); \D) \arrow[u, phantom, sloped, "\subset"] \arrow[r, "\simeq"{above}]& \Shan(\Rg, \D)^c \arrow[u,phantom, sloped, "\subset"]\\
                     \Shc(\SP^{k}(\Rg); \D) \arrow[u,phantom, sloped, "\subset"]\arrow[r, "\simeq"{above}] &  \Shk(\Rg, \D)\arrow[u,phantom, sloped, "\subset"]
                 \end{tikzcd}
            \end{equation}
            where 
            \begin{itemize}
                \item $\Shc^*(\SPi(\Rg); \D)$: $*$-sheaves, i.e. inverse limit of $\{\Shc(\SP^k(\Rg); \D)\}_{k \ge 0}$ along $*$-restrictions,
                \item $\Shc^!(\SPi(\Rg); \D)$: $!$-sheaves, i.e. inverse limit of $\{\Shc(\SP^k(\Rg); \D)\}_{k \ge 0}$ along $!$-restrictions,
                \item $\Shan(\Rg, \D)^c$: complete analytic sheaves (cf. \S \ref{sub:complete_analytic}),
                \item $\Shk(\Rg, \D)$: polynomial sheaves of degree $\le k$ (cf. \S \ref{sub:analytic_polynomial}).
            \end{itemize}
        \item\label{it:remain2} Via passage to exit path representations, we obtain the tower of equivalences
            \begin{equation}\label{eq:real_augmented_ident_thm}
                 \begin{tikzcd}
                    \Fun(\NN,\D) \arrow[r, "\simeq"{above}] & \D_{\Delta_+} \\
                    \Fun(\NN,\D)^0 \arrow[r, "\simeq"{above}]\arrow[u, phantom, sloped, "\subset"]  & (\D_{\Delta_+})^{\on{c}}\arrow[u, phantom, sloped, "\subset"]  \\
                    \Fun(\NN_{\le k},\D) \arrow[u, phantom, sloped, "\subset"] \arrow[r, "\simeq"{above}] & \D_{\Delta_+}^{(k)} \arrow[u, phantom, sloped, "\subset"]
                 \end{tikzcd}
            \end{equation}
            where 
            \begin{itemize}
                \item $\Fun(\NN,\D)^0$: filtered objects $F$ satisfying $\colim F \simeq 0$,
                \item $\D_{\Delta_+}$: augmented simplicial objects valued in $\D$,
                \item $(\D_{\Delta_+})^{\on{c}} \subset \D_{\Delta_+}$ consists of simplicial objects $G$ satisfying $G(\emptyset) \simeq \colim G|\Delta$,
                \item $\D_{\Delta_+}^{(k)} \subset \D_{\Delta_+}$ consists of
                    simplicial objects $G$ which are right Kan extensions of
                    their restrictions to the full subcategory of $\Delta_+$ spanned by $\{\emptyset, [0], [1], \dots, [k-1]\}$.
            \end{itemize}
            The equivalences in \eqref{eq:real_augmented_ident_thm} can be described explicitly as follows:
            For $[n] \in \Delta_+$, we consider the cube
            \[
                q_n: \P([n]) \to \NN, I \mapsto |I|
            \]
            of cardinalities of subsets of $[n]$. Then the simplicial
            object $G_{\bullet}$ corresponding to a filtered object $F$ is
            determined by the formula
            \[
                G_n = \tfib q_n^* F
            \]
            noting its functoriality in $[n] \in \Delta_+^{\op}$.
    \end{enumerate}
\end{thm}
\begin{proof}
        The exit path representations of \ref{it:remain2} are proved in \S \ref{subsec:real_constructible}.
        The equivalence of \ref{it:remain1} for the truncated case and $*$-sheaves is shown in \S \ref{sub:analytic_expansion_real}. The statement for $!$-sheaves is the content of \S \ref{sub:shriek_real}.
\end{proof}

\begin{thm}
    \label{thm:real_verdier}
    Let $\D$ be a stable bicomplete $\infty$-category.
    \begin{enumerate}
        \item\label{it:remain3} Composing the equivalence for $!$-sheaves from
            Theorem \ref{thm:real_analytic} with Verdier duality yields an
            equivalence
            \[
                 \coSh^*_{\PP}(\SPi(\Rg); \D)^0 \simeq \Shan(\Rg, \D)^c
            \]
            which upon
            passage to exit path representations yields an equivalence
            \begin{equation}\label{eq:homotopy_dk}
                \Ch_{\ge 0}(\D) \simeq \D_{\Delta}
            \end{equation}
            where $\Ch_{\ge 0}(\D)$ denotes the $\infty$-category of
            homotopy-coherent connective chain complexes in $\D$. 
        \item\label{it:remain4} In the presence of a $t$-structure on $\D$ with
            heart $\A$, the restriction of the equivalence
            \eqref{eq:homotopy_dk} to the hearts of the standard $t$-structures
            yields the classical Dold-Kan correspondence
            \begin{equation}\label{eq:homotopy_dk_heart}
                \Ch_{\ge 0}(\A) \simeq \A_{\Delta}.
            \end{equation}
    \end{enumerate}
\end{thm}
\begin{proof}
    This is established in \S \ref{sub:realverdier}.
\end{proof}

\subsection{Exit path representations for constructible sheaves on \texorpdfstring{$\SP^{\infty}(\Rg)$}{spirg}}
\label{subsec:real_constructible}

We begin by computing the exit path representations for constructible sheaves
on $X=\SP^n(\Rg)$, both with respect to the discriminant stratification over
$\PP$ and with respect to the filtration representation over $\NN$.

A basis $\B$ of conical disks of $\SP^n(\Rg)$ is given by the sums $\sum_{i=0}^r \SP^{m_i}(I_i)$ of Remark \ref{rem:basics},
where the  $I_i \subset \Rg$ are open (in $\Rg$) intervals such that $0 \in I_0$.
Recall the identification of $\SP^n(\Rg)$ with the half-open infinite $n$-simplex:
\[
	\SP^n(\Rg) \cong \{(x_1, \dots, x_n) \in \Rg^n : x_1 \ge x_2 \dots \ge x_n \} \subset
	\Rg^n,
\]
Using this description, we define a map
\[
	\pi:\SP^n(\Rg) \to [1]^n, \, \pi(x_1, \dots, x_n) = (\pi_1(x),
	\dots,\pi_n(x))
\]
with
\[
	\pi_k(x)= \begin{cases}
                0, &\text{ if } x_k=x_{k+1}\\
			1, &\text{ if } x_k \neq x_{k+1}\\

	\end{cases}
\]
with the convention that $x_{n+1}=0$.
The map $\pi$ stratifies $\SP^n(\Rg)$ by the different types of { \em ordered}
partitions, as opposed to the unordered partitions of $\PP_{\le n}$.
Consequently, there is a map $\rho:[1]^n \to \PP_{\le n}$ of posets
exhibiting $\pi$ as a refinement of the discriminant stratification:
\[\begin{tikzcd}
	{\SP^n(\Rg)} & {[1]^n} \\
	& {\PP_{\le n}}
	\arrow["\pi", from=1-1, to=1-2]
	\arrow[from=1-1, to=2-2]
	\arrow["\rho", from=1-2, to=2-2]
\end{tikzcd}\]
\begin{lem}\label{lem:contractible_decomposition_discriminant_strata}
        Let $Q \in \PP_{\le n}$ denote the unordered partition $k = m_1 + \dots + m_r$.
	The stratum $X_{Q}$ of the dicriminant stratification admits a decomposition into contractible open subsets
	\[
		X_{Q} = \bigsqcup_{q \in \rho^{-1}(Q)}
		X_{q},
	\]
    which are the strata of the stratification by $\pi$.
	For any standard basic $B \in \B$,
	the intersection $B \cap X_q$ is either empty or contractible.
\end{lem}
\begin{proof}
	The fiber $\rho^{-1}(Q)$ is given by the set
	of ordered partitions which induce $Q$. These are incomparable, which implies
        that $X_q \subset X_Q$ is open. The space $X_q$ is convex as a subspace of $\Rg^{n}$ and so is $B \cap X_q$ if it
        is nonempty.
\end{proof}
\begin{prop}\label{prop:real_discriminant_exit_path_cube}
	There is a natural equivalence
	\[
		\Sh_{\PP}(\SP^n(\Rg);\D) \simeq \Fun([1]^n,\D).
	\]
\end{prop}
\begin{proof}
	By Proposition \ref{prop:discriminant_exodromy}, there is an equivalence
	\[
		\Sh_{\PP}(\SP^n(\Rg);\D) \simeq \Fun_{\PP}(\B^{\op},\D).
	\]
	It thus suffices to show that the induced functor
	\[
		\pi:\B^{\op} \to [1]^n
	\]
	which takes a basic to its minimal stratum in the $\pi$-stratification
    is an $\infty$-categorical localization at stratified homotopy equivalences.
	We verify this by using the criterion \cite[A.6.4]{DKS:milnor}. It follows from basic
	facts on smooth conical manifolds that the morphisms which are inverted by
	$\pi$ are precisely the inclusions of conical disks of the same type.
	For $q \in [1]^n$ let $\B_q \subset \B$ denote the subposet of
	basics with minimal stratum $q$.
	To establish the second condition we utilize the open cover
	\[
		\B_q \to \U(X_q), B \mapsto B \cap X_q
	\]
	which takes basics of type $q$ to their intersection with the
	$q$-stratum. Both these intersections and the stratum $X_q$ itself are contractible by
        Lemma \ref{lem:contractible_decomposition_discriminant_strata}. It follows from Lurie's van-Kampen theorem
        that $\B_q$ is weakly contractible.

	To establish the cofinality condition of the third point, we prove
	that the comma categories are weakly contractible, again by a
	van-Kampen argument. For $B \in \B$ with $\pi(B) \le q$ this comma category
        is $(\B_q/B)^{\op}$, i.e. the poset of basics with minimal stratum $q$ which are contained in $B$. There is
	again a covering
	\[
		\B_q/B \to \U(B \cap X_q), \, B' \mapsto B' \cap X_q
	\]
	and $B \cap X_q$ is a contractible as above, so $(\B_q/B)^{\op}$ is weakly contractible.
\end{proof}
We also compute the exit path representation for the filtration stratification on $\SP^n(\Rg)$.
\begin{rem}\label{rem:strata_computation}
        Let $X=\SP^n(\Rg)$. The above exit path equivalence may also be identified with the pullback along $q \mapsto X_{\ge q}$,
        as the inclusion $B \subset X_{\ge q}$ of any conical basic $B \in \B_q$ with minimal stratum $q$ can be seen to be an
        equivalence.
\end{rem}
\begin{prop}\label{prop:exit_paths_real_case}
	There is a natural equivalence
	\[
		\Sh_{\NN}(\SP^n(\Rg);\D) \simeq \Fun(\Delta^n,\D).
	\]
\end{prop}
\begin{proof}
        Due to Theorem \ref{thm:filtration_exodromy} this can be seen by localizing $[1]^n$ at the appropriate class of morphisms or
        by utilizing the equivalence
	\[
                \Sh_{\NN}(\SP^n(\Rg),\D) \simeq \Fun_{\NN}(\B^{\op},\D).
	\]
	and showing that the functor $\B^{\op} \to \Delta^n$ is an
        $\infty$-categorical localization at the set of $\NN$-stratified homotopy equivalences.
	A proof is via the criterion of \cite[A.6.4]{DKS:milnor}. The first condition
        also follows from Theorem \ref{thm:filtration_exodromy}.
        By Lurie's van-Kampen theorem, the other two parts again reduce to the observation that
	for any basic open $B \in \B$ and $0 \le k \le n$ the intersection $B \cap
        X_k$ with the $k$-stratum is convex as a subset of $\Rg^{n}$ (hence contractible) or empty.
\end{proof}

\begin{cor}\label{cor:exit_filtered_star}
	There is an equivalence
	\[
		\Sh_{\NN}^*(\SPi(\Rg);\D) \simeq \Fun(\NN,\D).
	\]
\end{cor}
\begin{proof}
    Follows immediately by passing to the $*$-limit, observing that the
    $*$-restrictions are given by restricting along the face inclusions $\partial_{n+1}: \Delta^n \to
    \Delta^{n+1}$.
\end{proof}

We also compute the localizations on the other side:
\begin{prop}\label{prop:intervals_localization}
        For $k \in \NN \cup \{\infty\}$ the functor
\[
        \pi_0:\DD_{\le k}(\Rg) \to (\Delta_{+})_{\le k-1}
\]
exhibits $(\Delta_{+})_{\le k}$ as an $\infty$-categorical localization at the class of homotopy equivalences. Here,
$\pi_0(U)$ is regarded as a finite poset via the natural order induced from $\Rg$. (There is a unique isomorphism with the standard
simplex $\pi_0(U) \cong [| \pi(U)| -1]$.)
\end{prop}
\begin{proof}
To show that $\pi_0$ is an $\infty$-localization, we utilize the criterion from \cite[A.6.2]{DKS:milnor}.

Clearly $\pi_0$ inverts precisely the homotopy equivalences, since intervals are contractible.

To prove the second condition, we note that, for $0 \le l < k$, the category
\[
        ([l]/\pi_0)^{\iso} \cong \DD_{l+1}(\RR_{>0})^{\simeq}
\]
is given by the category of $l+1$-fold disjoint unions of intervals with morphisms given by homotopy equivalences.
To establish weak contractibility, we examine the functor
\[
        \chi:\DD_{l+1}(\RR_{>0})^{\simeq} \to \U(\SP^{l+1}(\mathbb{R}_{>0})), \bigcup_{i=0}^{l} I_i \mapsto \sum_{i=0}^l I_i
\]
These sums $\sum_{i=0}^l I_i$ are contractible and they cover the
unordered configuration space $C \subset \SP^{l+1}(\mathbb{R}_{>0})$, which is contractible because it is convex (as a subset of
$\Rg^n$).
For $c \in C$, the category $\chi^{-1}(c)$ is filtered: If $c \in \sum_{i=0}^l I_i$ and $c \in \sum_{i=0}^l I_i'$, then also
$c \in \sum_{i=0}^l (I_i \cap I_i')$, where we order the intervals by the natural order of $\Rg$.
This implies that $\chi^{-1}(c)$ is weakly contractible.
Applying \cite[A.3.1]{lurie:ha}, we conclude that $([l]/\pi_0)^{\iso}$ is weakly contractible.

Thirdly, we show that for $0 \le l \le k$, the inclusion
\[
j_{[l]}:([l]/\pi_0)^{\iso} \to [l]/\pi_0
\]
is $\infty$-coinitial. By Quillen's Theorem A, it suffices to show that the comma categories 
\[
        ([l]/\pi_0)^{\iso}/(\bigcup_{i=0}^r I_i,f:[l] \to [r])
\]
are weakly contractible. The restriction of $\chi$ to such a comma category covers the contractible (again by convexity) space
$C \cap (I_{f(0)}+ \dots +I_{f(l)})$. By the same filteredness argument as above, we may again apply the van Kampen theorem
to deduce weak contractibility.
\end{proof}
\subsection{The analytic expansion equivalence}
\label{sub:analytic_expansion_real}

We begin by showing that analytic expansion defines an equivalence on finite symmetric products. 
\eqref{eq:real_augmented_thm}.
\begin{prop}\label{prop:Rg_polynomification_equivalence}
        The truncated analytic expansion functor
        \[
                \An: \Sh_{\NN}(\SP^n( \Rg);\D) \to \Fun_{\NN}(\DD_{\le n}^{\op}(\RR_{>0}),\D)
        \]
        taking a sheaf $F \in \Sh_{\NN}(\SP^n (\Rg);\D)$ to the polynomial sheaf
        \[
                U \mapsto \Gamma_{\SP^n(\Rg \setminus U)}(F)
        \]
        is an equivalence of $\infty$-categories.
\end{prop}
Before giving the proof, we
note the following final ingredient, which which captures the contractibility of the configuration spaces of $\Rg$.
\begin{lem}\label{lem:Rg_ignore_discriminant_equivalence}
        Let $\B \subset \U(\SP^n(\Rg))$ denote the standard basis of sums
        \[
                B=\SP^{m_0}(D_0) + \dots + \SP^{m_r}(D_r).
        \]
        Let also $\B_{c} \subset \B$ denote the subcategory spanned by those $B$ with $m_1 = \dots =m_r =1$, i.e. without
        higher multiplicities except near $0$. The restriction functor
        \[
                \Fun_{\NN}(\B^{\op},\D) \to \Fun_{\NN}(\B_{c}^{\op},\D)
        \]
        is a trivial Kan fibration. The functors in the former category are precisely
        the right Kan extensions of functors in the latter.
\end{lem}
\begin{proof}
        We analyze the comma categories. For $ B=\sum_{i=1}^r \SP^{m_i}(D_i) +\SP^{m_0}(\Rg)$
        let $C_B = B/\B_{c}^{\op} \cap (\B_{c})_{m_0}$ denote the subposet of the comma category spanned by basics of
        the form $\SP^{m_0}(D_0')+\sum_{i=1}^{n-m_0} D_i \subset B$.
        This inclusion $C_B \to B/\B^{\op}_c$ is coinitial, because the relevant overcategories have a final object  and all elements of $C_B$ are basics of the same type, hence their inclusions are equivalences.
        The poset $C_B$ is also weakly contractible by Lurie's van-Kampen theorem: It can be interpreted as an open cover of the subspace $B_c \subset B$ given by divisors without multiplicities outside
        $D_0$. The space $B_c$ is convex as a subset of $\Rg^n$, hence contractible, and the cover is filtered over points,
        because the nonempty intersection of elements of $C_B$ is again an element of this poset.

        Now consider $F \in \Fun(\B^{\op},\D)$ with $\NN$-constructible restriction $F|\B_{c}^{\op}$. By
        the above, $F$ is a right Kan extension of its restriction if and only if $F(B) \simeq F(B')$
        for any $B \in \B$ and $B \supset B' \in C_B$, i.e. if and only if $F$ is $\NN$-constructible.
\end{proof}
\begin{proof}[Proof of Proposition \ref{prop:Rg_polynomification_equivalence}]
        Let $\BB = \DD^{*}_{\le n}(\RR_{>0})$. We verify the criterion of Corollary \ref{cor:analytification_equivalence_criterion}
        for $\AA = \AA_{\BB}$, so we must show that the diagonal arrow in the following commutative diagram is an equivalence:
\[\begin{tikzcd}
	{\Fun_{\NN}(\B^{\op},\D)} & {\Sh_{\NN}(\SP^n(\Rg);\D)} & \\
        {\Fun_{\NN}(\B_{c}^{\op},\D)}& {\Fun_{\NN}((\B_{c}\cup \BB_{\cap})^{\op},\D)} & {\Fun_{\NN}(\BB_{\cap}^{\op},\D)}
	\arrow[from=1-2, to=1-1]
	\arrow[from=1-2, to=2-2]
	\arrow["{r_{c}}", from=1-1, to=2-1]
	\arrow["{r_L}", from=2-2, to=2-3]
	\arrow["{r_R}"', from=2-2, to=2-1]
	\arrow[ from=1-2, to=2-3]
\end{tikzcd}\]
Here $\BB_{\cap} \subset \U(X)$ is the poset of sums of disks from Theorem \ref{thm:stable_interchange}.
The sheaves on $\SP^n(\Rg)$ are the right Kan extensions along the inclusion $\B^{\op} \subset \U(X)^{\op}$ by
Theorem \ref{thm:filtration_exodromy}.
By Lemma \ref{lem:Rg_ignore_discriminant_equivalence} the constructible presheaves on $\B$ are themselves the right
Kan extensions along $(\B_{c})^{\op} \subset \B^{\op}$. By transitiviy of Kan extensions, the right Kan extensions
of constructible presheaves along $(\B_{c})^{\op} \subset (\B_{c} \cup \BB_{\cap})^{\op}$ are constructible. Conversely,
any $F \in \Fun_{\NN}((B_{c}\cup \BB_{\cap})^{\op},\D)$ is then also canonically equivalent to the Kan extension of its restriction,
so that $r_R$ is a trivial Kan fibration.

We next show that $r_L$ is a trivial Kan fibration because the functors in $\Fun_{\NN}((\B_{c}\cup \BB_{\cap})^{\op},\D)$
are the left Kan extensions of functors in $\Fun_{\NN}(\BB_{\cap}^{\op},\D)$:
For any $B=\SP^{m_0}(D_0)+\sum_{i=1}^{n-m_0}D_i\in \B_{c}$
the open $D_1 + \dots +D_r + \SP^{m_0}(\Rg) \in \BB_{\cap}$ is a final object in $\BB_{\cap}^{\op}/B$.

By $2$-out-of-$3$, the restriction $\ShnRgC \to \Fun_{\NN}(\BB_{\cap}^{\op},\D)$ is an equivalence, so $\AA$ is too.
\end{proof}

\begin{cor}
    \label{cor:star_equivalence}
    Let $\D$ be a stable bicomplete $\infty$-category. Then analytic expansion induces an equivalence
    \[
         \begin{tikzcd}
             \Shc^*(\SPi(\Rg); \D) \arrow[r, "\An"{above}, "\simeq"{below}] &  \Shan(\Rg, \D).
         \end{tikzcd}
    \]
\end{cor}
\begin{proof}
    Obtained from the truncated equivalences of Proposition
    \ref{prop:Rg_polynomification_equivalence} by passing to the $*$-limits.
\end{proof}

\subsection{Verdier duality and coherent chain complexes}

By Theorem \ref{thm:verdier_duality_constructibility}, Verdier duality yields an equivalence
\[
	\V:\Sh_{\PP}(\SP^n(\Rg);\D) \to \coSh_{\PP}(\SP^n(\Rg);\D).
\]
We have previously computed the exit path category of this conical manifold in
Proposition \ref{prop:real_discriminant_exit_path_cube}, so we obtain the
following square of natural equivalences
\[\begin{tikzcd}
	{\Sh_{\PP}(\SP^n(\Rg);\D)} & {\coSh_{\PP}(\SP^n(\Rg);\D)} \\
    {\Fun([1]^n,\D)} & {\Fun(([1]^n)^{\op},\D)}
	\arrow["\V",from=1-1, to=1-2]
	\arrow[from=1-1, to=2-1]
	\arrow[from=1-2, to=2-2]
	\arrow["\V",from=2-1, to=2-2]
\end{tikzcd}\]
In this section we again denote $X=\SP^n(\Rg)$.
\begin{lem}\label{lem:greaterq_conical}
        Let $q \in [1]^n$ be a point in the discriminant of the stratifying poset, i.e. $q$ corresponds to a partition
        $n-m_0=\sum_{i=1}^r m_i$ with $m_i \ge 2$ for some $m_i >0$. Let $p$ denote the partition $n-m_0 =1 + \dots + 1$.
        Let $B \in \B_p$ be a conical basic with minimal stratum $p$. Then the inclusion
        \[
                B \subset X_{>q}
        \]
        is an $\NN$-stratified homotopy equivalence.
\end{lem}
\begin{proof}
        Suppose that $B=\SP^{m_0}(I_0) + \sum_{i=1}^{n-m_0} I_i$.
        A homotopy inverse is given by the restriction of
        \[
                g_1 \times \dots  \times g_n:  \times \RR_{>0}^{n-m_0} \times \Rg^{m_0} \to \Rg^n
        \]
        with orientation-preserving homeomorphisms  $g_{i}:\RR_{>0} \cong I_i$ for $i \le n-m_0$ and
        $g_{n-m_0+1}= \dots = g_{n}:\Rg \cong I_0$. A witnessing homtopy is given by straight-line homotopies in each factor.
\end{proof}
\begin{prop}\label{prop:verdier_filtration_image}
	Verdier duality induces an equivalence
        \[
                \V:\Sh_{\NN}(\SP^n(\Rg);\D) \simeq \coSh_{\PP}(\SP^{n}(\Rg);\D)^0
        \]
        between the category of $\NN$-constructible sheaves and the subcategory of $\PP$-constructible cosheaves
        \[
                G:\U(\SP^n(\Rg)) \to \D
        \]
        such that $G_x \simeq 0$ for any point $x$ in the discrimant,
        i.e. $x= m_0 \cdot 0 + m_1 \cdot y_1 \dots + m_r \cdot y_r$ with $m_i \ge 2$ for some $i \ge 1$.
\end{prop}
\begin{proof}
        For an ordered partition $q \in [1]^n$ let $i_q:X_q \to X$ denote the inclusion of the stratum.
        Because $\V \circ i_q^! \simeq i_q^* \circ \V$ and $\V$ is an equivalence, the preimage of
        $\coSh_{\PP}(\SP^{n}(\Rg);\D)^0$ under $\V$ is given by precisely those $F \in \Sh_{\PP}(\SP^n(\Rg);\D)$
        such that $i_q^!F \simeq 0$ for all strata $q \in [1]^n$ lying in the discriminant.
        For arbitrary $q \in [1]^n$ the sheaf $i_q^! F$ is locally constant on a contractible and we compute
        \[
                i_q^!F(X_q) \simeq \fib(F(X_{\ge q}) \to F(X_{>q})).
        \]
        by Remark \ref{rem:strata_computation}.
        
	Let now first $F \in \ShnRgC$ be a constructible sheaf with respect to the filtration stratification. Then, for
        a stratum $q \in [1]^n$ in the discriminant, the map $F(X_{\ge q}) \to F(X_{>q})$ is an equivalence because
        the inclusion $X_{>q} \to X_{\ge q}$ is an $\NN$-stratified homotopy equivalence by Lemma \ref{lem:greaterq_conical}.

	Conversely, suppose that $F \in \ShPnRgC$ is a constructible sheaf
	with respect to the discriminant stratification $\SP^n(\Rg) \to \PP$,
	and that
	\[
		i_q^! F \simeq 0
	\]
        for all $q \in [1]^n$ in the discriminant.
        We wish to show that $F$ is $\NN$-constructible. We will establish this
        inductively by showing that, for $q \in [1]^n$, the restriction of $F$
        to the open subspace $X_{\geq q}$ is $\NN$-constructible with respect to the filtration stratification,
	assuming that $F|X_{> q}$ is already $\NN$-constructible.

        We distinguish two cases. If $q \in [1]^n$ does not lie in the discriminant, then
        $\NN$-constructibility on $X_{\ge q}$ imposes no additional condition and follows immediately from
        $\NN$-constructibility on $X_{>q}$ and $\PP$-constructibility on $X_{\ge q}$.

        In the second case $q \in [1]^n$ does lie in the discriminant. In this case, we conclude
        \[
                0 \simeq i_q^!F(X_q) \simeq \fib(F(X_{\ge q}) \to F(X_{>q})),
        \]
        so that $F(X_{\ge q}) \to F(X_{>q})$ is an equivalence.
        It follows from Lemma \ref{lem:greaterq_conical} and Remark \ref{rem:strata_computation} that $F|X_{\ge q}$ is
        $\NN$-constructible.
\end{proof}
\begin{defi}\label{defi:chain_complex}
	We denote by
\[
        \Ch_{[0.n]}(\D) \subset
	\Fun(([1]^n)^{\op},\D)
\]
the full subcategory spanned by the diagrams $D: ([1]^n)^{\op} \to \D$
concentrated in the $n+1$ vertices of the form $(1, \dots, 1, 0,
\dots, 0)$. It is the category of chain complexes
concentrated in the interval $[0,n]$.
\end{defi}
\begin{ex}
	We explain in the case $n=3$ how Definition \ref{defi:chain_complex}
	captures the notion of a coherent chain complex. The objects of
        $\Ch_{[0,3]}(\D)$ are diagrams of the form
\[\begin{tikzcd}[sep=small]
	& 0 && 0 \\
	{A_3} && {A_2} \\
	& 0 && {A_0} \\
	0 && {A_1}
	\arrow[from=1-2, to=1-4]
	\arrow[from=1-2, to=3-2]
	\arrow[from=1-4, to=3-4]
	\arrow[from=2-1, to=1-2]
	\arrow[from=2-1, to=2-3]
	\arrow[from=2-1, to=4-1]
	\arrow[from=2-3, to=1-4]
	\arrow[from=2-3, to=4-3]
	\arrow[from=3-2, to=3-4]
	\arrow[from=4-1, to=3-2]
	\arrow[from=4-1, to=4-3]
	\arrow[from=4-3, to=3-4]
\end{tikzcd}\]
with $A_i \in \D$. The face
\[\begin{tikzcd}
	{A_3} & {A_2} \\
	0 & {A_1}
	\arrow[from=1-1, to=1-2]
	\arrow[from=1-1, to=2-1]
	\arrow[from=1-2, to=2-2]
	\arrow[from=2-1, to=2-2]
\end{tikzcd}\]
of the cube encodes a null-homotopy for the morphism $d^2:A_3 \to A_1$.
Similarly, the face
\[\begin{tikzcd}
	{A_2} & {A_1} \\
	0 & {A_0}
	\arrow[from=1-1, to=1-2]
	\arrow[from=1-1, to=2-1]
	\arrow[from=1-2, to=2-2]
	\arrow[from=2-1, to=2-2]
\end{tikzcd}\]
encodes a null-homotopy for $d^2: A_2 \to A_0$. By composition with identity morphisms,
these null-homotopies then
yield null-homotopies for $d^3: A_3 \to A_0$, and the total cube ought to be understood
as a homotopy between these homotopies.
\end{ex}
It remains to explicitly compute the chain complex associated to a filtered
object. At the level of triangulated
categories a procedure for constructing a chain complex from a filtered object
is described in \cite[Remark 1.2.4.3]{lurie:ha}. We now show that Verdier
duality realizes this construction geometrically at an $\infty$-categorical level.
\begin{prop}
	Let $F \in \ShnRgC$ be a constructible sheaf.
        Then
        \[
                \V(F)(X_{\ge k}) \simeq \fib(F(X_{\ge k}) \to F(X_{\ge k+1}))[-k].
        \]
        (If $k =n$, then $X_{\ge n+1} = \emptyset$.)
\end{prop}
\begin{proof}
        Let $Q$ denote the partition $k=1 + \dots +1$. By remark \ref{rem:strata_computation} and the computations
        in the proof of Proposition \ref{prop:verdier_filtration_image} we have
        \[
                \V(F)(X_{\ge k}) \simeq i_Q^*\V( F)(X_Q) \simeq \V(i_Q^!F)(X_Q)
        \]
        Since $F$ is $\NN$-constructible, we have $i_Q^!F(X_Q) \simeq \fib(F(X_{\ge k}) \to F(X_{\ge k+1}))$.
        The formula now follows from the well-known fact that Verdier duality for locally constant sheaves on $X_Q \cong \RR^k$ is
        given by the $k$-fold shift on the stalk.
\end{proof}

\subsection{\texorpdfstring{$!$}{shriek}-sheaves and completeness}
\label{sub:shriek_real}

The {\em right} adjoint $(i_k)^!$ of $(i_k)_*$ can be described as follows: it
associates to a $(k+1)$-step filtration object $\{F_i\}_{0 \le i \le k+1}$ the
$k$-step filtration
\begin{equation}
    \label{eq:right}
    \fib(F_0 \to F_{k+1}) \to \fib(F_1 \to F_{k+1}) \to \dots \to  \fib(F_k \to F_{k+1}).
\end{equation}
Define the $\infty$-category 
\[
    \Shc^!(\Rg;\D) := \lim_{k \ge 0} i^!_k
\]
of constructible $!$-sheaves as the inverse limit over these right adjoints
$i_k^!$. We further note the adjunction
\[
    (j_k)_* : \Shc(X_{\le k};\D) \longleftrightarrow \Shc^*(\RR_{\ge 0};\D) : j_k^!
\]
where, in terms of filtrations, $(j_k)_*$ is given by extending a filtration by
zeros and $j_k^!$ is given by formula \eqref{eq:right}.

\begin{prop}
    \label{prop:realshriek_sub}
        We may identify the $\infty$-category $\Shc^!(\Rg;\D)$ of
        constructible $!$-sheaves with the full subcategory of
        $\Shc(\Rg;\D)$ consisting of those sheaves $F$ satisfying one of the following equivalent conditions:
        \begin{enumerate}[label=(\roman *)]
            \item\label{shriek:1} The canonical map
                \[
                   \colim_{k \in \NN} (j_k)_*(j_k)^! F  \simeq F 
                \]
                is an equivalence.
            \item\label{shriek:2} Expressing $F$ as a $\NN$-filtration $\{F_k\}$, we have
                \[
                    \colim_{k \in \NN} F_{k} \simeq  0.
                \]
        \end{enumerate}         
\end{prop}
\begin{proof}
    The canonical map in \ref{shriek:1} is in fact the counit of an adjunction
    \begin{equation}\label{eq:shriekadj}
        \Shc^!(\Rg;\D) \llra \Shc^*(\Rg;\D)
    \end{equation}
    which can be constructed as a special case of \cite[4.6]{tashi:cores}, so
    the fact that it is an equivalence amounts to the statement that we have a
    fully faithful embedding $\Shc^!(\Rg;\D) \subset \Shc^*(\Rg;\D)$. Instead
    of setting up the argument in this general formal context, we provide a
    quick direct argument in terms of the identification 
    \[
        \Shc^*(\Rg;\D) \simeq \Fun(\NN,\D)
    \]
    Namely, denote by $\NN^{\rhd}$ the poset of natural numbers with a final
    element $\infty$ adjoined, and denote by
    \[
        \widetilde{M} \subset \Fun(\NN^{\rhd} \times \NN, \D) 
    \]
    the full subcategory consisting of diagrams $F:\NN^{\rhd} \times \NN \to
    \D$ satisfying 
    \begin{enumerate}
        \item\label{sd:1} for $i<j \in \NN$, we have $F(i,j) \simeq 0$, 
        \item\label{sd:2} for $i \ge j \in \NN$, and $i < k \le \infty$ the square
                \[
                \begin{tikzcd}
                    F(i,j) \ar{r}\ar{d} & F(i, j+1)\ar{d} \\
                    F(k,j) \ar{r} & F(k, j+1) 
                \end{tikzcd}
                \]
                is biCartesian.
    \end{enumerate}
    Denote further by 
    \[
        M \subset \Fun(\NN \times \NN, \D) 
    \]
    the full subcategory of diagrams $F$ satisfying the conditions \ref{sd:1}
    and \ref{sd:2} (with $k < \infty$). We may interpret the $k$th row of a
    diagram in $M$ as an object in $\Shc(X_{\le k};\D)$ and the remaining data
    comprising the diagram as exhibiting an object in $\lim_{k} i_k^!$ yielding an equivalence
    \[
        M \simeq \Shc^!(\Rg;\D).
    \]
    Further, restriction to $\{\infty\} \times \NN$ yields an equivalence $M
    \simeq \Fun(\NN,\D)$. 
    The adjunction \eqref{eq:shriekadj} then corresponds to the adjunction
    \[
        j_!: M \llra \widetilde{M}: j^*
    \]
    given by restriction and left Kan extension along the inclusion $j: \NN
    \times \NN \subset \NN^{\rhd} \times \NN$. It is now straighforward to
    explicitly verify the statements of the proposition in terms of this
    description.
\end{proof}

\begin{thm}
    \label{prop:realshriek}
    The analytic expansion equivalence \eqref{cor:star_equivalence}
            restricts to an equivalence of $!$-sheaves with complete analytic
        sheaves so that \eqref{cor:star_equivalence} can be augmented to a
            diagram
            \begin{equation}\label{eq:real_augmented}
                 \begin{tikzcd}
                     \Shc(\SPi(\Rg); \D) \arrow[r, "\An"{above}, "\simeq"{below}] &  \Shan(\Rg, \D) \\
                     \Shc^!(\SPi(\Rg); \D) \arrow[u, phantom, sloped, "\subset"] \arrow[r, "\simeq"{above}]& \Shan(\Rg, \D)^c \arrow[u,phantom, sloped, "\subset"]\\
                 \end{tikzcd}
            \end{equation}
            which can be identified with the diagram
            \begin{equation}\label{eq:real_augmented_ident}
                 \begin{tikzcd}
                    \Fun(\NN,\D) \arrow[r, "\simeq"{above}] & \D_{\Delta_+} \\
                    \Fun(\NN,\D)^0 \arrow[r, "\simeq"{above}]\arrow[u, phantom, sloped, "\subset"]  & (\D_{\Delta_+})^{\on{c}}\arrow[u, phantom, sloped, "\subset"]  \\
                 \end{tikzcd}
            \end{equation}
            where \begin{itemize}
                \item $\Fun(\NN,\D)^0 \subset \Fun(\NN,\D)$ consists of filtrations $F$ satisfying $\colim F \simeq 0$,
                \item $(\D_{\Delta_+})^{\on{c}} \subset \D_{\Delta_+}$ consists of simplicial objects $G$ satisfying $G(\emptyset) \simeq \colim G|\Delta$,
            \end{itemize}
\end{thm}
\begin{proof}
    We may identify $\Fun(\NN,\D)^0$ with $\Fun(\NN,\D)$ via
    \begin{equation}\label{eq:lurieident1}
        \Fun(\NN,\D)^0 \to \Fun(\NN,\D), \{F_k\}_{k \in \NN} \mapsto \{ \fib(F_0 \to F_{k+1}) \}_{k \in \NN}
    \end{equation}
    with inverse given by
    \[
        \{\tilde{F}_k\}_{k \in \NN} \mapsto \{ \cofib(\tilde{F}_{k-1} \to \colim \tilde{F}) \}_{k \in \NN}
    \]
    setting $F_{-1} = 0$. These coordinate changes can be naturally seen in
    terms of constructible sheaves. Namely, for every sheaf $F \in
    \Shc(\Rg;\D)$ we have a diagram
    \[
    \begin{tikzcd}
        \Gamma_{X_{\le 0}}(F)\ar{d} \ar{r} & \Gamma_{X_{\le 1}}(F) \ar{dd}\ar{r} & \cdots \ar{r} & \Gamma_{X_{\le k}}(F) \ar{dddd} \ar{r} & \cdots \ar{r} & \Gamma(X;F) \ar{d} \\
                                    0\ar{rrrrr}&&&&& \Gamma(X \setminus X_{\le 0} ;F) \ar{d} \\
                                               &0\ar{rrrr} &&&&\Gamma(X \setminus X_{\le 1} ;F)\ar{d} \\
                                     &&&&& \vdots \ar{d}\\
                                     &&&0 \ar{rr} &&\Gamma(X \setminus X_{\le k} ;F)\ar{d} \\
                                     &&&&& \vdots
    \end{tikzcd}
    \]
    where the indicated squares are exact triangles. In fact, the category
    $\Shc(X;\D)$ is equivalent to the category of such diagrams. The
    subcategory $\Shc^!(X;\D)$ gets, via the projection 
    \[
        F \mapsto \{ \Gamma_{X_{\le k}}(F) \}_{k \in \NN}
    \]
    identified with $\Fun(\NN,\D)^0$ and, via the projection
    \[
        F \mapsto F_k := \{ \Gamma(X \setminus X_{\le k-1} ;F) \}_{k \in \NN}
    \]
    with $\Fun(\NN,\D)$. Here, we set $X_{\le -1} = \emptyset$. The equivalence
    between $!$-sheaves and complete sheaves now follows from 
    \begin{enumerate}
        \item the observation that the sections with support $\Gamma_{X_{\le k}}(F)$ can be
            identified with the colimit of the restriction to $\Delta_{\le k}$ of
            the analytic expansion of $F$ (which is clear from the
            proof of Proposition \ref{prop:Rg_polynomification_equivalence}),
        \item and the computation
            \begin{align*}
                \colim_k F_k & \simeq \colim_k \cof(\Gamma_{X_{\le k}}(F) \to F(X))\\
                            & \simeq \cof(\colim_k \Gamma_{X_{\le k}}(F) \to F(X))\\
                            & \simeq \cof(\colim F^{\an}|\Delta^{\op} \to F^{\an}(\emptyset)).
            \end{align*}
    \end{enumerate}
    
\end{proof}

\begin{rem}
    \label{rem:luriedk}
    Note that we have an equivalence
    \begin{equation}\label{eq:lurieident2}
        \D_{\Delta} \simeq (\D_{\Delta_+})^{\on{c}}
    \end{equation}
    by means of left Kan extension. 
    Under the identifications \eqref{eq:lurieident1} and
    \eqref{eq:lurieident2}, the middle row of \eqref{eq:real_augmented_ident}
    translates to an equivalence
    \begin{equation}\label{eq:luriedk}
        \Fun(\NN,\D) \simeq \D_{\Delta}
    \end{equation}
    which recovers Lurie's Dold-Kan correspondence from \cite{lurie:ha}.
\end{rem}

\subsection{Verdier duality for \texorpdfstring{$!$}{shriek}-sheaves and \texorpdfstring{$t$}{t}-exactness}
\label{sub:realverdier}

In this section, we explain how the classical Dold-Kan correspondence,
formulated for abelian categories as opposed to stable $\infty$-categories,
arises in our geometric setup. Given a $t$-structure on the stable
$\infty$-category $\D$ with heart $\A$, it induces a standard one (defined
levelwise) on the category $\D_{\Delta}$ of simplicial objects with heart
$\A_{\Delta}$. However, the equivalence \eqref{eq:luriedk} does not identify
this $t$-structure with the standard one on $\Fun(\NN,\D)$ (with heart
$\Fun(\NN,\A)$) but rather a non-standard one. As we will now see, this latter
$t$-structure arises geometrically in our sheaf-theoretic setup: it can be
described as the ''Verdier dual'' $t$-structure on the category of sheaves. 

As the space $\SPi(\Rg)$ is not locally compact, we cannot make use of Verdier
duality in its usual form. But, as argued in \cite[Lecture 21]{lurie:tamagawa},
the situation has a remedy. Above, we have defined $!$-sheaves as the inverse
limit over the sequence
\[
\begin{tikzcd}
    \cdots & \ar[swap]{l}\Sh(\SP^{k-1}(\Rg);\D)  & \ar[swap]{l}{i_{k-1}^!}  \Sh(\SP^k(\Rg);\D)  & \ar[swap]{l}{i_k^!} \Sh(\SP^{k+1}(\Rg);\D) & \ar[swap]{l}\cdots
\end{tikzcd}
\]
Since the spaces $\SP^k(\Rg)$, $k \in \NN$, are locally compact (and Hausdorff), we may apply
Verdier duality to each of the terms of this sequence, yielding an equivalence
with the inverse limit over the $*$-pullbacks
\[
\begin{tikzcd}
    \cdots & \ar[swap]{l}\coSh(\SP^{k-1}(\Rg);\D)  & \ar[swap]{l}{i_{k-1}^*}  \coSh(\SP^k(\Rg);\D)  & \ar[swap]{l}{i_k^*} \coSh(\SP^{k+1}(\Rg);\D) & \ar[swap]{l}\cdots
\end{tikzcd}
\]
of {\em cosheaf} categories. Thus, passing to limits, covariant Verdier duality
yields an equivalence
\begin{equation}\label{eq:shriek-verdier}
   \V:\Sh^!(\SPi(\Rg);\D) \simeq \coSh^*(\SPi(\Rg);\D) =: \coSh(\SPi(\Rg);\D).
\end{equation}

\begin{thm}
    \label{thm:shriekverdier}
    \begin{enumerate}
        \item The restriction of $\V$ to $\PP$-constructible $!$-sheaves induces an equivalence
            \[
                    \Sh_{\PP}^!(\SPi(\Rg),\D) \overset{\simeq}{\lra} \coSh_{\PP}(\SPi(\Rg),\D).
            \]
        \item Restricting further to the full subcategory of
            $\NN$-constructible $!$-sheaves, we obtain an equivalence
            \[
                \Sh_{\NN}^!(\SPi(\Rg),\D) \overset{\simeq}{\lra} \coSh_{\PP}(\SPi(\Rg),\D)^0
            \]
            where $\coSh_{\PP}(\SPi(\Rg),\D)^0$ denotes the full subcategory of
            cosheaves whose restriction to the discriminant is $0$.
        \item There is an equivalence
            \[
                \coSh_{\PP}(\SPi(\Rg),\D)^0 \overset{\simeq}{\lra} \Ch_{\ge 0}(\D)
            \]
            of stable $\infty$-categories, where $\Ch_{\ge 0}(\D) \simeq \lim_n \Ch_{[0,n]}(\D)$
            is the limit along restriction morphisms.
    \end{enumerate}
\end{thm}
\begin{proof}
        This is a consequence of Theorem \ref{thm:verdier_duality_constructibility}, Proposition
        \ref{prop:verdier_filtration_image} and the exit path computation in Proposition
        \ref{prop:real_discriminant_exit_path_cube}, along with the observation
        that $*$-restriction of cosheaves corresponds to the (stupid)
        truncation of chain complexes. 
\end{proof}

We conclude by establishing the $t$-exactness of $\An \circ \V$.

\begin{thm}
    \label{thm:texact}
    Suppose the $\D$ comes equipped with a $t$-structure with heart $\A$. 
    \begin{enumerate}
        \item The equivalence
        \begin{equation}\label{eq:sheaf_chaindk}
            \An \circ \V: \coSh_{\PP}(\SPi(\Rg),\D)^0 \to \Shan(\Rg;\D)^c, G \mapsto (\An \circ \V (G))|\DD^*(\RR_{>0})
        \end{equation}
        is $t$-exact with respect to standard $t$-structures where we identify
        complete analytic sheaves with their restriction to the poset of nonempty
        opens.
    \item Under the identifications from Theorem \ref{thm:shriekverdier} and
        Remark \ref{rem:luriedk}, the equivalence \eqref{eq:chaindk} induces an
        equivalence
        \begin{equation}
            \label{eq:chaindk}
            \Ch_{\ge 0}(\D) \to \D_{\Delta}, B_{\bullet} \mapsto A_{\bullet}
        \end{equation}
        where, for $[m] \in \Delta$, we have
        \begin{equation}
            \label{eq:dkformula}
            A_m \simeq \bigoplus_{0 \le k \le m} B_{k}^{\binom{ m }{ k }}
        \end{equation}
    \item By passage to hearts, the equivalence \eqref{eq:chaindk} induces an equivalence
        \[
            \Ch_{\ge 0}(\A) \to \A_{\Delta}, B_{\bullet} \mapsto A_{\bullet}
        \]
        which coincides with the classical Dold-Kan correspondence for the
        abelian category $\A$. 
    \end{enumerate} 
\end{thm}
\begin{proof}
    We will verify that the functor \eqref{eq:chaindk} followed by evaluation
    at $[m] \in \Delta$ is given by the claimed formula \eqref{eq:dkformula}.
    This implies all other claims of the theorem since finite direct sums
    are $t$-exact. It suffices to verify \eqref{eq:dkformula} for cosheaves
    supported on a finite symmetric product $\SP^n(\Rg)$ and for $m \le n$. To this end, let $G \in
    \coSh_{\PP}(\SP^n(\Rg),\D)^0$ be a cosheaf corresponding to a coherent
    chain complex $B_{\bullet}$. Let further $U \in \DD_{m+1}(\RR_{>0})$
    be an open subset consisting of $m+1$ open intervals and denote
    \[
        A := \Rg \setminus U = A_0 \cup A_1 \cup \dots \cup A_{m+1}
    \]
    where the components $A_i$ are ordered by the natural order of $\Rg$. In particular,
    $A_0 = [0, a]$ and $A_{m+1} = [b,\infty)$ for certain $a < b \in \Rg$. We compute
    \begin{align*}
        \An(\V(G))(U) & = \Gamma_{\SP^n(A)} \circ \V (G) \\
                    & \simeq \Gamma \circ i^! \circ \V (G)\\
                    & \simeq \Gamma_c(\SP^n(A); i^*G)
    \end{align*}
    where $i: \SP^n(A) \subset \SP^n(\Rg)$ denote the closed embedding. Now the
    closed subset $\SP^n(A)$ is a union of the components
    \[
        \underbrace{A_0 + \dots + A_0}_{r_0} + \underbrace{A_1 + \dots + A_1}_{r_1} + \dots + \underbrace{A_{m+1} + \dots + A_{m+1}}_{r_{m+1}}
    \]
    where $r_k \ge 0$ with $\sum_k r_k = n$ 
    so that $\An(\V(G))(U)$ splits into the direct sum of sections in $i^*G$
    with compact support on these components. These compactly supported
    sections vanish on components that satisfy one of the following conditions:
    \begin{itemize}
        \item $r_i > 1$ for some $0<i<m+1$, since the sections can be identified
            with a stalk of the cosheaf on the discriminant (where the cosheaf
            vanished by assumption), 
        \item $r_i \le 1$, for all $0<i<m+1$, and $r_{m+1} > 0$ since the corresponding sections can be identified with compactly
            supported sections of a constant cosheaf (with value $B_{n-r_0}$) on the noncompact subset 
            \[
                    \underbrace{0 + \dots + 0}_{r_0} + \dots + \underbrace{  {A_{m+1} + \dots +A_{m+1}}}_{r_{m+1}},
            \]
    \end{itemize}
    cf. Figure \ref{fig:components}. The arguments for these latter vanishing
    results use the same techniques as in the proof of Lemma
    \ref{lem:extend_zero_disk}, we omit the details.
    Therefore, the only contributions to $\An(\V(G))(U)$ come from components of the form
    \[
        \underbrace{A_0 + \dots + A_0}_{n-k} +  A_{j_1} + \dots + A_{j_k}
    \]
    where $0 \le k \le n$ and $0 < j_1 < \dots < j_k < m+1$. Compactly supported
    sections on such a component can be identified with sections of $i^*G$ on the compact orthotope
    \[
        \underbrace{0 + \dots + 0}_{n-k} +  A_{j_1} + \dots + A_{j_k}
    \]
    of the constant cosheaf with value $B_k$ yielding the desired formula
    \eqref{eq:dkformula}.
\end{proof}

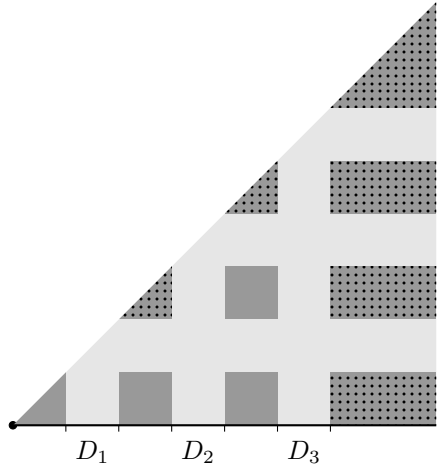
\begin{figure}[ht]
    \centering
\begin{tikzpicture}[scale=.7]
  \def\xmax{8}
  \def\ymax{8}

  \fill[black!10] (0,0) -- (\xmax,0) -- (\xmax,\ymax) -- cycle;

  \filldraw[black] (0,0) circle (2pt);

  \begin{scope}
    \clip (0,0) -- (\xmax,0) -- (\xmax,\ymax) -- cycle;

    \fill[black!40] (0,0) rectangle (1,1);
    \fill[black!40] (2,2) rectangle (3,3);
    \fill[pattern=dots] (2,2) rectangle (3,3);
    \fill[black!40] (4,4) rectangle (5,5);
    \fill[black!40, pattern=dots] (4,4) rectangle (5,5);
    \fill[black!40] (6,6) rectangle (\xmax,\ymax);
    \fill[black!40, pattern=dots] (6,6) rectangle (\xmax,\ymax);

    \fill[black!40] (2,0) rectangle (3,1);
    \fill[black!40] (4,0) rectangle (5,1);
    \fill[black!40] (6,0) rectangle (\xmax,1);
    \fill[black!40, pattern=dots] (6,0) rectangle (\xmax,1);

    \fill[black!40] (4,2) rectangle (5,3);
    \fill[black!40] (6,2) rectangle (\xmax,3);
    \fill[black!40, pattern=dots] (6,2) rectangle (\xmax,3);

    \fill[black!40] (6,4) rectangle (\xmax,5);
    \fill[black!40, pattern=dots] (6,4) rectangle (\xmax,5);

  \end{scope}

  \draw[thick] (0,0) -- (\xmax,0);

  \foreach \t in {1,2,3,4,5,6} { \draw (\t,0) -- (\t,-0.12); }
  \node[below] at (1.5,-0.12) {$D_1$};
  \node[below] at (3.5,-0.12) {$D_2$};
  \node[below] at (5.5,-0.12) {$D_3$};

\end{tikzpicture}
    \caption{The components of the complement of $U +
    \Rg$ in $\SP^2(\Rg)$, the ones with trivial contributions to $\An(\V(G))(U)$ are dotted.}%
    \label{fig:components}
\end{figure}

\newpage

\section{\texorpdfstring{$\SP^{\infty}(\CC)$}{spic} and the paracyclic Dold-Kan correspondence}

Let again $\D$ be a bicomplete {\em stable} $\infty$-category.
We discuss how the analytic expansion functor is related to the paracyclic Dold-Kan
correspondence.

Let $C = C(S^1) \cong \CC$. To obtain an equivalence from analytic expansion,
we will have to use a slightly modified notion of analytic sheaf, as explained now.
We consider the action of the group $\RR_+$ of positive reals on $\CC^*$ by
multiplication. We call $U \in \DD(\CC^*)$ {\em radial} if it is
$\RR_+$-invariant and denote by $\DDr(\CC^*) \subset \DD(\CC^*)$ the subset of
radial opens. We refer to the homotopy invariant presheaves on $\DDr(\CC^*)$
as {\em radial analytic sheaves}. They form an $\infty$-category denoted by
$\Shanr(\CC^*;\D)$.

Note that we obtain an identification of posets
\[
    \DDr(\CC^*) \lra \DD(S^1), U \mapsto U \cap S^1
\]
so that we may identify radial analytic sheaves on $\CC^*$ with analytic
sheaves on $S^1$. 

\subsection{The results}
\label{sub:complex_results}

Let $\D$ be a bicomplete stable $\infty$-category.

\begin{thm}
    \label{thm:complexanalytic}
        The analytic expansion functor induces a tower of equivalences

            \begin{equation}\label{eq:complex_expansion}
                 \begin{tikzcd}
                     \Shc^*(\SPi(\CC);\D) \arrow[r, "\An"{above}, "\simeq"{below}] &  \Shanr(\CC^*;\D) \\
                     \Shc^{*!}(\SPi(\CC); \D) \arrow[u, phantom, sloped, "\subset"] \arrow[r, "\simeq"{above}]& \Shanr(\CC^*, \D)^c \arrow[u,phantom, sloped, "\subset"]\\
                     \Shc^*(\SP^{k}(\CC); \D) \arrow[u,phantom, sloped, "\subset"]\arrow[r, "\simeq"{above}] &  \Shkr(\CC^*, \D)\arrow[u,phantom, sloped, "\subset"]
                 \end{tikzcd}
            \end{equation}
            \begin{itemize}
                    \item $\Shc^{*!}(\SPi(\CC); \D)$: $*!$-sheaves (c.f. \S \ref{sub:starshriek})
                    \item $\Shanr(\CC^*;\D) \cong \Shan(S^1,\D)$: radial analytic sheaves,
                    \item $\Shanr(\CC^*;\D)^{\c} \cong \Shan(S^1,\D)^{\c}$: complete radial analytic sheaves,
                    \item $\Shkr(\CC^*;\D) \cong \Sh^{\le k}(S^1,\D)^{\c}$: radial $k$-polynomial sheaves,
            \end{itemize}
            Via exit path representations, these yield the following combinatorial tower of equivalences:

            \begin{equation}\label{eq:complex_combinatorial}
                 \begin{tikzcd}
                         \Fun(\NN_{S^1}^{\triangleleft},\D) \arrow[r, "\An"{above}, "\simeq"{below}] &  \D_{\Lambda_{\infty,+}} \\
                         \Fun(\NN_{S^1}^{\triangleleft},\D)^0 \arrow[u, phantom, sloped, "\subset"] \arrow[r, "\simeq"{above}]
                        & (\D_{\Lambda_{\infty}})^{\c} \arrow[u,phantom, sloped, "\subset"]\\
                         \Fun((\Delta^{k-1}_{S^1})^{\triangleleft},\D) \arrow[u,phantom, sloped, "\subset"]\arrow[r, "\simeq"{above}] &  \D_{\Lambda_{\infty,+}}^{(k)} \arrow[u,phantom, sloped, "\subset"]
                 \end{tikzcd}
         \end{equation}
         where
        \begin{itemize}
        \item $\NN_{S^1} \cong \bigstar_{n \in \NN} \Bop \ZZ$: infinite join of $\Bop \ZZ$,
        \item $\Fun(\NN_{S^1},\D)^0$: diagrams $F:\NN_{S^1}^{\triangleleft} \to \D$ satisfying $\colim F \simeq 0$.
        \item $\Delta^{k-1}_{S^1} \cong \bigstar_{i=1}^{k} \Bop \ZZ$: $k$-fold join of $\Bop \ZZ$,
        \item $(\D_{\Lambda_{\infty,+}})^{\on{c}} \subset \D_{\Delta_+}$ consists of paracyclic objects $G$ satisfying $G(\emptyset) \simeq \colim G|\Lambda_{\infty}$,
        \item $\D_{\Lambda_{\infty,+}}^{(k)} \subset \D_{\Lambda_{\infty,+}}$ consists of
                    simplicial objects $G$ which are right Kan extensions of
                    their restrictions to the full subcategory of $\Lambda_{\infty,+}$ spanned by $\{\emptyset, [0], [1], \dots, [k-1]\}$.
        \end{itemize}
\end{thm}
\begin{proof}
        The exit path computations for the truncated case and $*$-sheaves are contained in \S \ref{sub:exit_paths_complex_case}.
        The proof of the analytic expansion equivalence in these cases is given in \S \ref{sub:complex_analytic}.
        The equivalence for $*!$-sheaves is discussed in \S \ref{sub:starshriek}.
\end{proof}

\begin{thm}
    \label{thm:perverset}
    Suppose $\D$ is a presentable stable $\infty$-category
    and comes equipped with an accessible left complete $t$-structure with heart $\A$. Then
   the equivalence \eqref{eq:complex_expansion} yields an
    equivalence 
    \[
            \Perv^{*!}(\SPi(\CC);\A) \overset{\simeq}{\lra} \A_{\Lambda_{\infty}}
    \]
    of abelian categories identifying perverse sheaves with paracyclic objects.
    The terms of the paracyclic object corresponding to a perverse sheaf $F$
    are given by the vanishing cycles $\{\Phi_{z_1^k + z_2^k + \dots}(F)_0\}_{k
    \in \NN}$ at the origin with respect to the power sums.
\end{thm}
\begin{proof}
        The proof is given in \S \ref{sub:complex_perverse}.
\end{proof}

\subsection{Exit path representations for constructible sheaves on \texorpdfstring{$\SP^{\infty}(\CC)$}{spic}}
\label{sub:exit_paths_complex_case}
In this section we analyze constructible sheaves on $\SP^{\infty}(\CC)$.
We begin the analysis for $\SP^n(\CC)$, fixing $n \in \NN$ for now.
Let $r:\CC^* \to S^1, z \mapsto \frac{z}{|z|}$ denote the standard retraction.
A basis of disks in $\CC^*$ is given by the collection of radial segments, i.e.
\[
        D(I,a,b) = \{z \in \CC | |z| \in (a,b), r(z) \in I\}
\]
for $ 0 \le a < b \le \infty$ and $I \subsetneq S^1$ an open interval.
Let $\B$ denote the corresponding basis of conical disks $\sum_{i=0}^k \SP^{m_i}(D_i) \subset \SP^n(\CC)$ of
Remark \ref{rem:basics}.

We shrink this basis slightly by imposing the additional requirement
that $r(D_i) \cap r(D_j) \neq \emptyset$ implies $r(D_i) = r(D_j)$.

 We set $\BB = \DDrn(\CC) \setminus \{\emptyset\}$, the poset of nonempty radial opens,
yielding again posets $\BB_{\cap},\BB_{\cup} \subset \U(\SP^n(\CC))$.
We also introduce a superposet of $\BB_{\cap}$, which more closely resembles $\B$ by
admitting multiplicities:
\[
        \BB_{\m} = \{\sum_{i \in I} \SP^{m_i}(D_i) + \SP^{n-|I|}(\CC) | \bigcup_{i \in I} D_i \in \BB, \sum_{i \in I} m_i \le n\}.
\]
where, as usual, we require $m_i \ge 1$ for $i \ge 1$.
We may compute exit paths via this poset as well:
\begin{lem}
        \label{lem:complex_case_expand_disks}
        Restriction induces an equivalence
        \[
                \Sh_{\NN}(\SP^n(\CC);\D) \simeq \Fun_{\NN}(\BB_{\m}^{\op},\D).
        \]
\end{lem}
\begin{proof}
        By Exodromy, the restriction functor
        \[
                \Sh_{\NN}(\SP^n(\CC);\D) \to \Fun_{\NN}(\B^{\op},\D)
        \]
        is an equivalence.
        We proceed by showing that both restrictions in the span
        \[
            \begin{tikzcd}
            {\Fun_{\NN}(\B^{\op},\D)} & {\Fun_{\NN}((\B \cup\BB_{\m})^{\op},\D)} & {\Fun_{\NN}(\BB_{\m}^{\op},\D)} 
            \arrow[from=1-2, to=1-1]
            \arrow[from=1-2, to=1-3]
        \end{tikzcd}
        \]
    are trivial Kan fibrations. The claim then follows by $2$-out-of-$3$.

        We show that the functors in $\Fun_{\NN}((\B \cup\BB_{\m})^{\op},\D)$ are precisely
        the right Kan extensions of functors in $\Fun_{\NN}(\B^{\op},\D)$.  By Theorem \ref{thm:exodromy}, the right Kan
        extensions of constructible presheaves on $\B$ are constructible sheaves. Then so is any intermediate Kan extension.

        Conversely, a constructible presheaf $F$ on $\B \cup \BB_{\m}$ is equivalent to the right Kan extension of $F|\B^{\op}$
        by Corollary \ref{cor:extended_disk_type}.

        We show that the functors in $\Fun_{\NN}((B \cup\BB_{\m})^{\op},\D)$ are the left Kan extensions of constructible presheaves
        on $\BB_{\m}$: For any $B=\sum_{i=0}^r \SP^{m_i}(D_i) \in \B$ the open set
        \[
                \sum_{i=1}^r \SP^{m_i}(r^{-1}(r(D_i))) + \SP^{m_0}(C) \in \BB_{\m}
        \]
        is initial in the comma category $B/\BB_{\m}$. (Note that in the above we may have $r(D_i)=r(D_j)$ in which case the
        multiplicities are understood to be added together.)
        This inclusion is again an $\NN$-stratified homotopy equivalence by
        Lemma \ref{lem:extend_zero_disk}. By Lemma \ref{lem:union_inclusion}, this implies that a presheaf on $B \cup\BB_{\m}$ is
        $\NN$-constructible if and only if it is a left Kan extension.
\end{proof}
By Lemma \ref{lem:complex_case_expand_disks} and Exodromy, we may compute the category of constructible sheaves by localizing
$\BB_{\m}$:
\begin{prop}
        \label{prop:exit_paths_complex_case}
        Let $\Delta^{n-1}_{S^1} =(\Bop \ZZ)^{*n}$ denote the $n$-fold join of $\Bop \ZZ$.
        The objects of $(\Delta^{n-1}_{S^1})^{\triangleleft}$ are the natural numbers $0 \le k \le n$.
        There is a functor
        \[
                L:(\BB_{\m})^{\op} \to (\Delta^{n-1}_{S^1})^{\triangleleft}
        \]
        of $1$-categories, which is an $\infty$-categorical localization at the class of stratified homotopy equivalences.
\end{prop}
\begin{proof}
        The functor $L$ is given on objects by
        \[
                L(\sum_{i=1}^k \SP^{m_i}(D_i) +\SP^{m_0}(\CC)) = n-m_0 \in (\Delta^{n-1}_{S^1})^{\triangleleft}
        \]
        On morphisms, if $i:B \hookrightarrow B' \in \BB_{\m}$, there are two cases. If $L(D) \neq L(V)$, then
        $L(i):L(B) \to L(B')$ is the unique arrow between these objects. If $L(B)=L(B')=t$,
        then we may write
        \[
                B=\sum_{i=1}^k \SP^{m_i}(D_i)+\SP^{n-t}(\CC) \subset B' =\sum_{j=1}^r \SP^{l_j}(D_j') + \SP^{n-t}(\CC).
        \]
        Before defining $L$ on these morphisms, we introduce some auxillary notation. For $B \in \BB_{\m}$ given as above, let
        \[
                U_{S^1}(B) = \bigcup_{i=1}^k r(D_i) \in \DD(S^1)
        \]
        denote the union of the projections of the $D_i$ to $S^1$. Letting $p:\RR \to S^1, x \mapsto e^{2 \pi i x}$ denote the
        standard universal covering, we also denote $U_{\RR}(B)= p^{-1}(U_{S^1}(B))$.
        To define $L$ we choose the cut $\Rg \subset \CC$.
        If $1 \notin U_{S^1}(B')$, we set $L(i)=\id$.
        Otherwise, we may assume $1 \in D_1'$. Counting with multiplicities, there are $l_1$ disks among the $D_i$ with
        $D_i \subset D'_1$. The morphism $L(i) \in \Aut(t) \cong \ZZ$ is then given by counting,
        with multiplicities, the number of the $D_i$ lying entirely in the
        component of $D_1' \setminus \Rg$ above the cut. It is easy to see that this assignment
        yields a functor.

        We show that $L$ is an $\infty$-categorical localization by checking the criterion
        \cite[A.6.2]{DKS:milnor}. It follows from Lemma \ref{lem:intersection_inclusions}, that
        the morphisms inverted by $L$ are precisely the stratified homotopy equivalences.
        For the second condition, we must show that the groupoid $L^{-1}(t)/\Bop \ZZ$ of morphisms
\[\begin{tikzcd}
	B & {t}
	\arrow["a", from=1-1, to=1-2]
\end{tikzcd}\]
        is weakly contractible.
        We introduce a functor $\chi:L^{-1}(t)/\Bop \ZZ \to \U(\SP^{t}(\RR))$.
        For $B \in L^{-1}(t)$ as above there is a unique map $\phi_B:\ZZ \to \pi_0(U_{\RR}(B))$ with the following properties:
        \begin{enumerate}
                \item $\phi_B$ is weakly increasing with respect to the natural order on
                        $\pi_0(U_{\RR}(B))$.
                \item $|\phi_B^{-1}(W)| = m_i$ for a component $W \in \pi_0 (U_{\RR}(B))$ with
                        $p(W)=r(D_i)$.
                \item $\phi_B(0) \subset \RR_{>0}$ and $\phi_B(-1) \nsubset \RR_{>0}$.
        \end{enumerate}
        We now define
        \[
                \chi:L^{-1}(t)/\Bop \ZZ \to \U(\SP^{t}(\RR)), \,
                (B,a) \mapsto \sum_{i=0}^{t-1}\phi_B(a+i) \subset \SP^{t}(\RR).
        \]
        where $\sum_{i=0}^{t-1}\phi_B(a+i) \cong \{x_1 \ge \dots \ge x_t \in \RR^t | x_i \in \phi_B(a+i) \}$ under the standard
        embedding of $\SP^t(\RR)$.
        Functoriality of $\chi$ is a consequence of our definitions.
        We have
        \[
                \bigcup_{(B,a) \in L^{-1}(t)/\Bop \ZZ }\chi(B,a)
                = \{ \sum_i x_i \in \SP^{t}(\RR) : |x_i -x_j| < 2 \text{ for all } i,j \}.
        \]
        because any point on the right-hand side is contained in $\chi(B,a)$ for some
        $B=\SP^{t}(r^{-1}(S^1 \setminus \{y\})) + \SP^{n-t}(\CC)$ and $a \in \ZZ$.
        This space is convex, hence contractible.

        The functor $\chi$ is filtered over points: If
        \[
                s \in (\sum_{i=0}^{t-1} \phi_B(a+i)) \cap  (\sum_{i=0}^{t-1}\phi_{B'}(b+i)) \subset \SP^{t}(\RR)
        \]
        for $(B,a),(B',b) \in L^{-1}(t)/\Bop \ZZ$, then $s \in \sum_{i} \phi_B(a+i) \cap \phi_{B'}(b+i)$.
        Therefore $B \cap B' \in L^{-1}(t)$, $x \in \chi(B \cap B', \max(a,b))$ and there are unique morphisms
        $(B \cap B',\max(a,b)) \to (B,a)$ and $(B \cap B',\max(a,b)) \to (B',b)$. By Lurie's van-Kampen theorem, the
        category $L^{-1}(t)/\Bop \ZZ$ is weakly contractible.

        Thirdly, for $V \in \BB_{\m}$ with $L(V) =v$ and $v <t$ we must show that the
        full subcategory $M_V \subset L^{-1}(t)/\Bop \ZZ$ spanned by objects $(B,a)$ with $B \subset V$ is weakly contractible.
        It suffices to show that the restriction $\chi|M_V$ covers a contractible subspace $Y$.
        To show that $Y$ is contractible, we consider another cover of $Y$ by opens $Y_a$ for $a \in \ZZ$.
        To define $Y_a$, we denote
        \[
                \sum_{i=0}^{v-1}\phi_V(a+i) + \SP^{t-v}(\RR) = \sum_{j=1}^{l} \SP^{m_j}(I_j) + \SP^{t-v}(\RR),
        \]
        where the $I_j$ are ordered by the natural order of $\RR$.
        Now
        \[
                \begin{split}
                        Y_a&=\sum_{j=1}^{l} \SP^{m_j}(I_j) + \SP^{t-v}(\RR) \\
                &\cap \{x_1 \ge \dots \ge x_{t} \in \SP^t(\RR) |  \exists S \subset \{1,\dots,t\} \text{ with }
                |S| =m_1, x_i \in I_1 \text{ for } i \in S \text{ and } x_1 < x_i +1 \text{ for }  x_i \in S \} \\
                &\cap \{x_1 \ge \dots \ge x_{t} \in \SP^t(\RR) |  \exists S \subset \{1,\dots,t\} \text{ with }
                |S| =m_l, x_i \in I_l \text{ for } i \in S \text{ and } x_t + 1 > x_i \text{ for }  x_i \in S \}. \\
                \end{split}
        \]
        To show that $Y_a$ is contractible, we note that the inclusion
        $Y_a \subset \sum_{j=1}^{l} \SP^{m_j}(I_j) + \SP^{t-v}(\RR)$ is a homotopy equivalence:
        Suppose $I_1 = (y,z)$. The inclusion is an equality unless $I_l=(y+1,z+1)$ (the translate of $I_1$).
        In that case, a homotopy inverse is induced by a homeomorphism
        $g:\RR \to \RR$ such that
        \[
                \begin{split}
                        g((-\infty,y)) &\subset (\frac{y+z}{2},z),\\
                        g((z,\infty)) &\subset (y+1,1+\frac{y+z}{2}),\\
                        g|_{[z,y+1]} &=\id.
                \end{split}
        \]
        Contractibility of $Y_a$ follows because $\sum_{j=1}^{l} \SP^{m_j}(I_j) + \SP^{t-v}(\RR)$ is contractible
        by Lemma \ref{lem:extend_zero_disk}.
        For $a \in \ZZ$ the intersection $Y_a \cap Y_{a+1}$ is also contractible by essentially the same argument
        (and because $v +1 \le t$).
        Finally, if $Y_a \cap Y_{b}$ for $a < b$, then $|\phi_V(\{a,a+1, \dots,b\})| \le 2$ and we find
        $Y_a \cap Y_b = Y_a \cap Y_{a+1} \cap \dots \cap Y_{b}$. By another application of the van-Kampen theorem,
        the space $Y$ is also contractible.

\end{proof}
\begin{cor}
        There is an equivalence $\Sh_{\NN}(\SP^{n}(\CC); \D) \simeq \Fun((\Delta^{n-1}_{S^1})^{\triangleleft},\D)$ for each
        $n \in \NN$ and an equivalence $\Sh^*_{\NN}(\SPi(\RR); \D) \simeq \Fun(\NN_{S^1}^{\triangleleft},\D)$.
\end{cor}
\begin{proof}
        The finite version follows from Lemma \ref{lem:complex_case_expand_disks} and Proposition \ref{prop:exit_paths_complex_case}.
        The version for $*$-sheaves follows by passing to the limit along $*$-restrictions.
\end{proof}
We also analyze the localization of radial analytic sheaves:
\begin{prop}\label{prop:s1_intervals_localization}
        For $k \in \NN \cup \{\infty\}$ there is a functor
\[
L: \DD_{\le k}(S^1) \to (\Lambda_{\infty,+})_{< k}
\]
exhibiting $(\Lambda_{\infty,+})_{<k}$ as the $\infty$-categorical localization along the class of homotopy equivalences.
\end{prop}
\begin{proof}
        We again denote $p:\mathbb{R} \to S^1, x \mapsto \exp(2\pi i x)$ and fix for each $U \in \DD_{\le k+1}(S^1)$
        the unique isomorphism $\phi_U: \ZZ \cong \pi_0(p^{-1}(U))$ which is weakly increasing with respect to the
        natural order and such that $\phi_U(0) \subset \RR_{>0}$, but $\phi_U(-1) \nsubset \RR_{>0}$.
        The functor $L$ is given on objects by mapping $U \in \DD_l(S^1)$ to $\langle l -1 \rangle$.
        On a morphism $\iota:U \subset V$ the map $L(\iota)$ is defined by commutativity of the following square:
        \[\begin{tikzcd}
	{\pi_0(p^{-1}(U))} & {\pi_0(p^{-1}(V))} \\
	\ZZ & \ZZ
	\arrow["{\pi_0(\iota)}", from=1-1, to=1-2]
	\arrow["{\phi_U}", from=2-1, to=1-1]
	\arrow["{L(\iota)}"', from=2-1, to=2-2]
	\arrow["{\phi_V}"', from=2-2, to=1-2]
\end{tikzcd}\]

The first condition of the localization criterion \cite[A.6.2]{DKS:milnor} is immediate. To prove the other two points
of the criterion, we first give a description of the categories $\langle l \rangle/L$. An object of $\langle l \rangle/L$
is given by a tuple $(U \in \DD_m(S^1),f:\ZZ \to \ZZ)$ with $m \le k+1$ and $f$ weakly increasing and periodic (in the sense
that $f(i+l+1)=f(i)+m+1$ for all $i \in \ZZ$.)
A morphism $(U,f:\mathbb{Z} \to \mathbb{Z}) \to (V,g: \mathbb{Z} \to \mathbb{Z})$ amounts to a commutative diagram
\[
\begin{tikzcd}
        \mathbb{Z} \arrow[r, "f"] \arrow[rd, "g"'] & \mathbb{Z} \arrow[r, "\phi_U"] \arrow[d, "L(\iota)"] & \pi_0(p^{-1}(U))
        \arrow[d, "\pi_0(\iota)"] \\
                                                   & \mathbb{Z} \arrow[r, "\phi_V"']                      & \pi_0(p^{-1}(V))
\end{tikzcd}
\]
where $\iota:U \to V$ is the inclusion map.
The remainder of the proof mirrors that of Proposition \ref{prop:exit_paths_complex_case}.
To establish the weak contractibility of $(\langle l \rangle /L)^{\iso}$, we utilize the functor
\[
\chi:(\langle l \rangle/L)^{\iso} \to \U(\SP^{l+1}(\mathbb{R})),
        (V \in \DD_{l+1}(S^1),f:\mathbb{Z} \to \mathbb{Z}) \mapsto \sum_{i=0}^{l} \phi_V(f(i)) \subset \SP^{l+1}(\RR)
\]
which may be regarded as a restriction of the the functor in that proof.
This map is a cover of the contractible (convex) configuration space
\[
        C=\{x_0 + \dots + x_{l} \in \SP^{l+1}(\RR) | 0 < |x_i-x_j|<1 \text{ for all } i \neq j\},
\]
by contractibles and is filtered over points, hence $(\langle l \rangle /L)^{\iso}$ is weakly contractible by \cite[A.3.1]{lurie:ha}.
%

Finally, the inclusion $j_{\langle l \rangle }: (\langle l \rangle /L)^{\iso} \to \langle l \rangle /L$
is cofinal, because each comma category
\[
        E=(\langle l \rangle /L)^{\iso}/(U,f:\mathbb{Z} \to \ZZ)
\]
is weakly contractible:
This follows by another application of the van-Kampen theorem to $\chi|E$. The image of this functor covers the space
\[
        C \cap \sum_{i=0}^l \phi_U(f(i)) \subset \SP^{l+1}(\RR),
\]
which is contractible. Indeed, it is nonempty because $\inf(\phi_U(f(l))) < \max(\phi_U(f(0)))+1$ (at worst $\phi_U(f(l))$ may be the
shift of $\phi_U(f(0))$), and it is convex as a subset of $\RR^{l+1}$.

\end{proof}
\subsection{The analytic expansion equivalence}
\label{sub:complex_analytic}
In this section we first show that the truncated radial analytic expansion functors are an equivalence by using the criterion from
Corollary \ref{cor:analytification_equivalence_criterion} with $\BB=\DDrn(\CC^*) \setminus \{\emptyset\}$ as before.
\begin{lem}\label{lem:drop_multiplicities_complex_case}
        The restriction functor yields an equivalence
        \[
                \Fun_{\NN}(\BB_{\m}^{\op},\D) \simeq \Fun_{\NN}(\BB_{\cap}^{\op},\D).
        \]
\end{lem}
\begin{proof}
The restriction $L|(\BB_{\cap} \setminus \{\emptyset\})$ of the localization functor
        \[
                L: \BB_{\m} \to (\Delta^{n-1}_{S^1})^{\triangleleft}
        \]
        from Proposition \ref{prop:exit_paths_complex_case} is also a localization functor: The proof is largely the same as
        the one given there: The opens in a fiber now cover the (contractible) configuration space
        \[
                C=\{\sum_{i}x_i \in \SP^{t}(\RR): 0 < |x_i -x_j |< 1  \text{ for all }i \neq j\}.
        \]
        and in the last part of the proof the space $Y_a$ is replaced with the intersection
\[
Y_a = C \cap \sum_{i=0}^{v-1}\phi_V(a+i) + \SP^{t-v}(\RR).
\]
(We now only need to consider the case where $V$ has no multiplicies.)
By the proof of Proposition \ref{prop:exit_paths_complex_case}, the space
\[
        \tilde{Y}_a   = \{\sum_{i}x_i \in \SP^{t}(\RR): |x_i -x_j |< 1  \text{ for all }i,j\} \cap
        \sum_{i=0}^{v-1}\phi_V(a+i) + \SP^{t-v}(\RR)
\]
(allowing multiplicities) is contractible. The inclusion $Y_a \subset \tilde{Y}_a$ is a homotopy equivalence.
A homotopy inverse is given by
\[
        \tilde{Y}_a \to Y_a, x_1 \ge \dots \ge x_t \mapsto \sum_i (x_i + \frac{t-i}{t}\epsilon(x_1,\dots,x_t))
\]
where $\epsilon:\tilde{Y}_a \to \RR_{>0}$ is a sufficiently small continuous function such that the above is well-defined.
The consecutive intersections $Y_a \cap Y_{a+1}$ are contractible by essentially the same argument and all other intersections
are empty.

\end{proof}
\begin{cor}
        Radial analytic expansion yields an equivalence
    \[
         \begin{tikzcd}
                 \Sh_{\NN}(\SP^{n}(\CC); \D) \arrow[r, "\An"{above}, "\simeq"{below}] &  \Shnr(\CC, \D)
         \end{tikzcd}
    \]
    and in the limit we have
    \[
         \begin{tikzcd}
                 \Shc^*(\SPi(\CC); \D) \arrow[r, "\An"{above}, "\simeq"{below}] &  \Shanr(\CC, \D).
         \end{tikzcd}
    \]
\end{cor}
\begin{proof}
        The truncated claim follows from the criterion in Corollary \ref{cor:analytification_equivalence_criterion},
        Lemma \ref{lem:complex_case_expand_disks} and Lemma \ref{lem:drop_multiplicities_complex_case}. The version for $*$-sheaves follows by taking inverse limits along $*$-restrictions.
\end{proof}
\subsection{Comparison with the real equivalence}
The classical paracyclic Dwyer-Kan equivalence descends to the simplicial correspondence via restrictions.
In this section, we indicate how one can achieve this in the geometric context.
On the side of polynomial sheaves, we utilize the description via $S^1$ and the inclusion
\[
        j:\RR_{>0} \to S^1, t \mapsto \exp(2i\pi \frac{t}{1+t})
\]
(corresponding to a cut at the positive real axis.)
For $n \in \NN$, we denote by $s_n=\SP^n(s)$ the $n$-fold symmetric product of the standard inclusion.

\begin{prop}\label{prop:restrict_to_simplex_finite}
        For $n \in \NN$ the diagram
\[\begin{tikzcd}
	{\Sh_{\NN}(\SP^{n}(\CC);\D)} & {\Sh^{\le n}(S^1,\D)} \\
	{\Sh_{\NN}(\SP^{n}(\Rg);\D)} & {\Sh^{\le n}(\RR_{>0},\D)}
	\arrow["\AA", from=1-1, to=1-2]
	\arrow["{s_n^*}"', from=1-1, to=2-1]
	\arrow["{j^*}", from=1-2, to=2-2]
	\arrow["\AA", from=2-1, to=2-2]
\end{tikzcd}\]
        commutes up to homotopy.
\end{prop}
\begin{proof}

Of course, $s_n$ is stratified-homotopic to the standard embedding, so the following comparison result also holds for
the standard embedding by homotopy-invariance.
On the other side, we must give a corresponding stratified inclusion $\Rg \to \CC$.
To mirror the behaviour of $j$, we use another embedding
\[
        t:\Rg \to \CC, t \mapsto t\exp(2i\pi \frac{t}{1+t}).
\]
given by a spiral.
Note that $\Sp^n(t) \simeq s_n$ in an $\NN$-stratified way, so by homotopy invariance \cite[Theorem 3.4]{volpe2025verdier},
we have $s_n^* \simeq \SP^n(t)^*$ on $\NN$-constructible sheaves, so we may as well
prove the proposition for $t_n :=\SP^n(t)$
        For an $D \in \DD(\RR_{>0})$ let $D^{\RR_+} \subset \CC$ denote the radial open with
        $D^{\RR_+} \cap S^1 = j(D)$. By the choice of our inclusions, we have
        $t_n^{-1}(D^{\RR_+}+\SP^{n-1}(\CC))=D+\SP^{n-1}(\Rg)$, so for
        $F \in \Sh_{\NN}(\SP^{n}(\CC);\D)$:
        \[
                t_n^* F(D + \SP^{n-1}(\Rg) ) =  (t_n)_{*} t_n^* F(D^{\RR_+}+\SP^{n-1}(\CC))  .
        \]
        It thus suffices to show that the natural arrow
        \[
                F(D^{\RR_+}+\SP^{n-1}(\CC)) \to  (t_n)_{*} t_n^* F(D^{\RR_+}+\SP^{n-1}(\CC))
        \]
        is an equivalence. This follows easily using the standard covering of $D^{\RR_+}+ \SP^{n-1}(\CC)$ by
        the opens $\sum_{i \in I} D^{\RR_{+}}_i + \SP^{n-|S|}(\CC)$ where $D^{\RR_{+}}_i$ are the components of $D^{\RR_{+}}$,
        since the value of a constructible sheaf on these opens is given by an appropriate stalk.
\end{proof}

\begin{cor}\label{cor:star-para-simplex-comparison}
        The square
\[\begin{tikzcd}
	{\Sh^*_{\NN}(\SP^{\infty}(\CC);\D)} & {\Sh^{\an}(S^1,\D)} \\
	{\Sh^*_{\NN}(\SP^{\infty}(\Rg);\D)} & {\Sh^{\an}(\RR_{>0},\D)}
	\arrow["\AA", from=1-1, to=1-2]
        \arrow["{s_{\infty}^*}"', from=1-1, to=2-1]
	\arrow["{j^*}", from=1-2, to=2-2]
	\arrow["\AA", from=2-1, to=2-2]
\end{tikzcd}\]
        commutes up to homotopy.
\end{cor}
\begin{proof}
        This follows from Proposition \ref{prop:restrict_to_simplex_finite} by noting that $s_n^*$ and $j^*$ commute with
        $*$-restrictions.
\end{proof}

\begin{rem}
        Tracing through the localization arguments, one sees that on the level of combinatorial (exit path) categories,
        the restriction $i_{\infty}^*$ and $j^*$ correspond to restrictions along the inclusions
        $\NN \to \NN_{S^1}^{\triangleleft}$ (taking initial object to initial object)
        and $\Delta_+ \subset \Lambda_{\infty,+}$ respectively.
\end{rem}

\subsection{\texorpdfstring{$*!$}{starshriek}-Sheaves}
\label{sub:starshriek}

Recall that, for a complex manifold $X$ and a holomorphic function $f$, we may
define a vanishing cycles functor as follows. We set $X_0 = f^{-1}(0)$ and
$X_{\Rg} = f^{-1}(\Rg)$ and consider the natural filtration
\[
    X_0 \overset{a}{\subset} X_{\Rg} \overset{b}{\subset} X.
\]
We define
\[
    \Phi_{f}: \Sh(X;\D) \to \Sh(X_0;\D), F \mapsto a^!b^*F.
\]

\begin{defi}
    \label{prop:vanishing_completion}
    For $k \ge 1$, let $e_k = z_1 z_2 \dots z_k$ denote the $k$th elementary symmetric
    polynomial in $k$ variables. We consider the sequence
    \[
        \Sh(\SP^0(\CC);\D) \overset{\Phi_{e_1}}{\lla} \Sh(\SP^1(\CC);\D) \overset{\Phi_{e_2}}{\lla} \Sh(\SP^2(\CC);\D) \lla  \dots
    \]
    and define its inverse limit
    \[
        \Sh^{*!}(\SP^{\infty}(\CC);\D) := \lim_{k \ge 0} \Sh(\SP^k(\CC);\D).
    \]
\end{defi}
We also introduce a restriction map on polynomial functors on $S^1$ which will correspond to the above vanishing cycles:
Let $\DD_{k+1}^{1}(S^1) \subset \DD_{k+1}(S^1)$ denote the set of those unions $\bigcup_{i=1}^{k+1} D_i \subset S^1$
such that
\[
        \sum_{i=1}^{k+1} D_i \cap \{\sum z_i \in \SP^{k+1}(S^1)| \prod_i z_i = 1\} \neq \emptyset.
\]
There is a diagram
\[\begin{tikzcd}
	& {\DD_{\le k}(S^1)} \\
        {\DD^{*}_{\le k}(S^1) \cup \DD_{k+1}^{1}(S^1)} & {\DD_{\le k}(S^1) \cup \DD_{k+1}^{1}(S^1)} \\
	{\DD(S^1)}
	\arrow["c", hook, from=1-2, to=2-2]
	\arrow["d", hook, from=2-1, to=2-2]
	\arrow["e"', hook, from=2-1, to=3-1]
\end{tikzcd}\]
We define
\[
i^{*!}_k:\Sh^{\an}(S^1;\D) \to \Sh^{\le k}(S^1,\D)\]
by right Kan extension of the composite
\[
        i^{*!}_kF=c^* d_{!} e^* F.
\]
In other words, we first restrict to the values on unions of at most $k$ disjoint disks and certain unions of $k+1$ disjoint disks,
then extend to $\emptyset$ by means of a left Kan extension and finally forget the values at the remaining unions of $k+1$ disjoint
disks (before extending to all disks by right Kan extension).

To better understand the above formula we make the following computation:
\begin{lem}
        \label{lem:starshriekempty}
        We denote $\DD=\DD^{*}_{\le k}(S^1) \cup \DD_{k+1}^{1}(S^1)$.
        For $D_1 \cup \dots \cup D_{k+1} \in \DD_{k+1}^1(S^1)$.and $F \in \Sh^{\an}(S^1;\D)$. A functor
        $G:(\DD^{\triangleleft})^{\op} \to \D$ with $G|\DD^{\op}=e^*F|\DD^{\op}$ is a left Kan extension of $e^* F$
        if and only if the difference cube $\delta^{k+1}(G)(\emptyset;D_1,\dots,D_{k+1})$ is biCartesian.
\end{lem}
\begin{proof}
       We utilize the natural embedding $\DD \subset \U(\SP^{k+1}(S^1))$. Letting $\UU = \{ \sum_i D_i' | \bigcup D_i' \in \DD_{k+1}^1 \} \subset \U(\SP^{k+1}(S^1))$ denote the sums corresponding to the unions of the special unions of $k+1$ disks, we study the inclusion $u:\DD \to \DD \cup \UU$.
        We observe that, for $\bigcup_i D_i' \in \DD_{k+1}^1(S^1)$, there is a natural equivalence
        \[
                u_! F (\sum_i D_i') \simeq \colim_{\emptyset \neq I \subset [k+1]} F(\bigcup_{i \in I} D_i')
        \]
        because the punctured difference cube is cofinal in the comma category. In particular, $u_! F | \UU$ takes values
        in $\D^{\simeq}$.
        To establish the lemma, it therefore suffices to prove the following:
        \begin{enumerate}
                \item The inclusion $\UU^{\op} \to (\DD \cup \UU)^{\op}$ is cofinal.
                \item \label{item:UU_contractible} The poset $\UU$ is weakly contractible.
        \end{enumerate}
        For the proof, we use a covering argument. Let
        \[
                \Conf_{k+1}^1(S^1) = \Conf_{k+1}(S^1) \cap \{\sum_i z_i | \prod_i  z_i =1\} \subset \SP^{k+1}(S^1)
        \]
        denote the space of configurations obeying the additional restriction $\prod_i z_i =1$.
        There is a covering $\UU \to \Conf_{k+1}^1(S^1), U \mapsto U \cap \Conf_{k+1}^1(S^1)$. It is clearly a covering by contractibles that is filtered over points, so to prove the second point it remains to note that the space $\Conf_{k+1}^1(S^1)$ is contractible.
        To prove the first point, we use Quillen's Theorem A and restrict the above covering to the relevant comma category.
        For $\bigcup_{i=1}^m D_i'$ with $m \le k$, the restriction to the comma category covers the contractible space
        \[
                \{\sum_i z_i \in \Conf_{k+1}^1(S^1) | \exists i,j: z_i \in D_j'\}.
        \]

\end{proof}
\begin{prop}\label{prop:starshriek_analytic}
        There is a canonical equivalence
        \[
                \Sh^{\an}(S^1;\D)^c \simeq \lim_k i^{*!}_{k-1}|\Sh^{\le k}(S^1,\D)
        \]
        between complete analytic sheaves on $S^1$ and the inverse limit.
\end{prop}
\begin{proof}
        The $i^{*!}_k$ induce a map which we show to be an equivalence.
        By Lemma \ref{lem:starshriekempty} restriction yields an equivalence
        \[
                \Sh^{\le k}(S^1,D) \simeq \Fun^{\an}((\DD_{\le k}^*(S^1) \cup \DD_{k+1}^1(S^1))^{\op},\D).
        \]
        where $\Fun^{\an}((\DD_{\le k}^*(S^1) \cup \DD_{k+1}^1(S^1))^{\op},\D)$ is the subcategory of functors inverting homotopy
        equivalences.
        This equivalence commutes with $i^{*!}_{k-1}$ and restriction along the inclusion
        \[
        j_{k-1}:\DD_{\le k-1}^*(S^1) \cup \DD_{k}^1(S^1) \subset \DD_{\le k}^*(S^1) \cup \DD_{k+1}^1(S^1) \]
        so that
        \[
        \lim_k i^{*!}_{k-1}  \simeq \lim_k j_{k-1}^* \]
        The induced map $\Sh^{\an}(S^1)^{\op},\D)^{\c} \to \lim_k j_{k-1}^*$ is plainly an equivalence as it is given at each
                level by restriction.
\end{proof}
\begin{lem}
        \label{lem:explicit_vanishing}
        The diagram
        \[\begin{tikzcd}
	{\Sh(\SP^k(\CC);\D)} & {\Sh^{\le k}(S^1,\D)} \\
	{\Sh(\SP^{k-1}(\CC);\D)} & {\Sh^{\le k-1}(S^1,\D)}
	\arrow["\AA", from=1-1, to=1-2]
	\arrow["{\Phi_{e_k}}", from=1-1, to=2-1]
        \arrow["{i^{*!}_{k-1}}", from=1-2, to=2-2]
	\arrow["\AA", from=2-1, to=2-2]
\end{tikzcd}\]
commutes up to homotopy.
\end{lem}
\begin{proof}
        For a sheaf $F \in \Sh(\SP^k(\CC);\D)$ there is a natural diagram
        \[
                (i_{k-1})_* \Phi_{e_k} F \to b_* b^* F \leftarrow F
        \]
        of sheaves on $\SP^k(\CC)$. 
        To simplify notation, we introduce $G=(i_{k-1})_* \Phi_{e_k} F$ and write $X=\SP^k(\CC)$.
        We make the following observations:
        \begin{enumerate}
                \item For $U = \bigcup_i D_i + \SP^{k-1}(\CC)$ with $\bigcup_i D_i \in \DD^{\Rg,*}_{\le k-1}(\CC) \cup \DD^{\Rg,1}_{k}(\CC)$ the map
                        \[
                                F(U) \to b_*b^*F(U)
                        \]
                        is an equivalence.
                \item For $\bigcup_i D_i \in \DD^{\Rg,*}_{\le k-1}(\CC)$ and $U = \bigcup_i D_i + \SP^{k-1}(\CC)$ the induced map
                \[
                        \fib(G(\SP^k(\CC)) \to G(U)) \to \fib(b_* b^* F(\SP^k(\CC)) \to b_* b^* F(U))
                \]
                is an equivalence.
        \item For $\bigcup_i D_i \in \DD^{\Rg,1}_k(\CC)$ all arrows in the induced diagram
\[\begin{tikzcd}
	{\fib(G(X) \to G(\sum_i D_i))} & {\fib(b_*b^*F(X) \to b_*b^*F(\sum_i D_i))} & {\fib(F(X) \to F(\sum_i D_i))} \\
	{\fib(G(X) \to G(\emptyset))}
	\arrow[from=1-1, to=1-2]
	\arrow[from=1-1, to=2-1]
	\arrow[from=1-3, to=1-2]
\end{tikzcd}\]
are equivalences.
        \end{enumerate}
       Because $F$ is a sheaf the punctured cube
        \[
                \{F(\bigcup_{i \in I}D_i + \SP^{k-1}(\CC) ) \}_{\emptyset \neq I \in \P(\{1,\dots,k\})^{\op}}
        \]
        has colimit $F(D_1 + \dots + D_k)$ for $\bigcup_i D_i \in \DD^{r,1}_k(\CC)$. It follows from the above observations and
        the description of $i^{*!}$ in Lemma \ref{lem:starshriekempty} that there are natural isomorphisms
       \[
               \AA(\Phi_{e_k}F) \simeq \AA(G) \simeq i^{*!}\AA(F).
       \]
\end{proof}
\begin{cor}
    \label{cor:starshriekequiv}
    Analytic expansion yields an equivalence
    \[
\Sh^{*!}(\SP^{\infty}(\CC);\D)  \simeq \Sh^{\an}(S^1;\D)^c
\]
\end{cor}
\begin{proof}
        The equivalence follows from Proposition \ref{prop:starshriek_analytic} and Lemma \ref{lem:explicit_vanishing}.
\end{proof}
\begin{cor}
        The induced inclusion
\[
        \Sh^{*!}(\SP^{\infty}(\CC);\D) \subset \Sh^{*!}(\SP^{\infty}(\CC);\D) \simeq \Fun(\NN_{S^1}^{\triangleleft},\D)
\]
identifies $\Sh^{*!}(\SP^{\infty}(\CC);\D)$ with the subcategory of functors
\[
        F: \NN_{S^1}^{\triangleleft} \to \D
\]
such that $\colim F \simeq 0$.
\end{cor}
\begin{proof}
        This is a consequence of the corresponding statement for the real case in Proposition \ref{prop:realshriek_sub} and the
        comparison of Corollary \ref{cor:star-para-simplex-comparison}. To complete the proof, we note that
        the inclusion $\NN \to \NN_{S^1}^{\triangleleft}$ is cofinal (the commma categories are filtered), and the inclusion
        $\Delta \subset \Lambda_{\infty}$ is well-known to be coinitial.
\end{proof}
\begin{rem}
        As in Remark \ref{rem:luriedk}, we may use the equivalence $\Fun(\NN_{S^1}^{\triangleleft},\D)^0 \simeq \Fun(\NN_{S^1},\D)$
        to obtain an equivalence
        \[
                \Fun(\NN_{S^1},\D) \simeq \D_{\Lambda_{\infty}}
        \]
        in the style of Lurie.
\end{rem}
The following computation shows that for a $*!$-sheaf $F$, the value of $\Fan$ at $D_1 \cup \dots \cup D_k$ can be
    identified with the (stalk at $0$ of the) vanishing cycle functor of the
    sum of $k$th powers
    \[
        \Phi_{z_1^k + z_2^k + z_3^k + \dots}(F)_0.
    \]

    \begin{lem}\label{lem:vanishing_homotopy}
        For $n \ge k \ge 1$ we have a decomposition $D=\{z: \Re z^{k} < 0\} = D_1 \cup \dots \cup D_k$ into components. The inclusion
        \[
                \{\sum_{i=1}^n z_i |\Re \sum_{i=1}^n z_i^k < 0\} \subset D + \SP^{n-1}(\CC)
        \]
        is an $\NN$-stratified homotopy equivalence.
\end{lem}
\begin{proof}
        Let $D'=\{z: \Re z^{k} < -1\}$. Then the inclusion $D'+\SP^{n-1}(\CC) \subset D + \SP^{n-1}(\CC)$ is a stratified homotopy
        equivalence by Lemma \ref{lem:shrink_disks} (for appropriately chosen charts). Inspection reveals that the homotopies
        descend to the intersection with $\{\sum_{i=1}^n z_i |\Re \sum_{i=1}^n z_i^k < 0\}$.
        It therefore suffices to show that
        \[
                \{\sum_{i=1}^n z_i |\Re \sum_{i=1}^n z_i^k < 0\} \cap (D' + \SP^{n-1}(\CC)) \subset D' + \SP^{n-1}(\CC)
        \]
        is a stratified homotopy equivalence. A homotopy inverse may be found by using the parametrized version of Lemma
        \ref{lem:symmetric_homotopy}, with $(H_i)_{\sum_i z_i}(y)=y \cdot (1+\sum_i |z_i|^k)$.
\end{proof}
\begin{cor}\label{cor:analytification_is_vanishing}
        For $F \in \Sh_{\NN}(\SP^n(\CC);\D)$ we have
        \[
                \AA(F)(\{z: \Re z^{k} < 0\}) \simeq (\Phi_{z_1^k + \dots + z_n^k} F)_0
        \]
        and the sheaf $\Phi_{z_1^k + \dots + z_n^k} F$ is concentrated in $0 \in \SP^n(\CC)$.
\end{cor}
\begin{proof}
        Let $\sum_i z_i \in \SP^n(\CC)$ be a nonzero point with $\Re \sum_i z_i^k = 0$. It follows that there is a summand $z_1$ with
        $\Re z_1^k < 0$. Let $B = \sum_{i=0}^k \SP^{m_i}(D_i) \in \B$ denote a conical basic centered on $\sum_i z_i$. We assume
        $z_1 \in D_1$ and write $D_1 = D(I,a,b)$ for the radial segment (c.f. \S \ref{sub:exit_paths_complex_case}).
        We consider the homotopy given by moving points in $D_1$
        into the top third $D(I,\frac{2b+a}{3},b)$ of $D_1$ via straight lines and keeping all the other summands constant.
        If $I$ is sufficiently small and $b$ is sufficiently
        large compared to the diamteter of the other $D_i$, then this homotopy witnesses that the inclusion
        \[
                B \cap \{z: \Re z^{k} < 0\} \subset B
        \]
        is an $\NN$-stratified homotopy equivalence.
        Such basics $B$ form a neighborhood basis, so we may compute the stalk with them and see that it vanishes as claimed.

        The other computation follows immediately from Lemma \ref{lem:vanishing_homotopy}.
\end{proof}

\subsection{Perverse \texorpdfstring{$*!$}{*!}-sheaves and \texorpdfstring{$t$}{t}-exactness}
\label{sub:complex_perverse}
In this section, we assume that $\D$ is presentable and comes equipped with an accessible left complete $t$-structure with
heart $\A$. For $k \in \NN$, there is a perverse $t$-structure on $\Sh_{\NN}(\SP^k(\CC);\D)$
induced by the middle perversity $\mathfrak{p}(m) = -m$, see \cite[Proposition A.14]{dyckerhoff2026laxdescentfoliatedverdier}
(corresponding to the perversity $S \to - \dim_{\CC} S$ in the notation of \cite{beilinson1982faisceaux}).
We denote by $\Perv(\SP^k(\CC);\A)$ the heart of the perverse $t$-structure.
The objects of $\Perv(\SP^k(\CC);\A)$ will be called perverse sheaves.

Because vanishing cycles functors are $t$-exact with respect to these perverse
$t$-structures (c.f. \cite[Theorem 6.0.2]{Schurmann2003}) we may define
perverse sheaves on the infinite symmetric product as
\[
        \Perv^{*!}_{\NN}(\SP^{\infty}(\CC);\D) := \lim_{k \ge 0} \Perv_{\NN}(\SP^k(\CC);\A)
\]
where the limit is taken over the various vanishing cycles functors $\Phi_{e_k}$.

We note that the perverse $t$-structure is compatible with restriction to the real symmetric product: 
\begin{lem}\label{lem:t-structure-comp-real}
        The functor
        \[
                s_k^*:\Sh_{\NN}(\SP^{k}(\CC);\D) \to \Sh_{\NN}(\SP^{k}(\Rg);\D).
        \]
        is $t$-exact with respect to the perverse $t$-structure on $\Sh_{\NN}(\SP^{k}(\CC);\D)$ and the perverse $t$-structure on
        $\Sh_{\NN}(\SP^{k}(\Rg);\D)$ described in \S \ref{sub:realverdier}.
\end{lem}
\begin{proof}
    By construction of the $t$-structures, it is immediate that $s_k^*$ is
    right $t$-exact. To verify that it is left $t$-exact, we proceed by
    induction on $k$ follows. 
    Passing to elementary symmetric coordinates, we identify $\SP^k(\CC)$
    with $\CC^k$, and consider the diagram 
    \begin{equation}
        \label{eq:texactsquare}
    \begin{tikzcd}
        \CC^{k-1} \ar[hook,r, "a"] & \CC^{k-1} \times \Rg \ar[hook, r, "b"] & \CC^{k}\\
        \SP^{k-1}(\Rg)  \ar[hook,u, "s_{k-1}"]\ar[hook,r,"r"] & \SP^{k}(\Rg) \ar[hook,u,"t"]\ar[hook,ur, swap, "s_{k}"]& 
    \end{tikzcd}
    \end{equation}
    of inclusions of spaces. Now suppose $F \in \Sh_{\NN}(\SP^{k}(\CC);\D)$ and
    $F \in \pD^{\ge 0}$. By definition of the perverse $t$-structure, to show $s_k^*F
    \in \pD^{\ge 0}$, we need to verify that 
    \begin{enumerate}[label=\arabic*.]
        \item\label{tex:1} $(s_k^*F)\SP^k(\RR_{>0}) \in D^{\ge -k}$, and 
        \item\label{tex:2} $r^!s_k^*F \in \pD^{\ge 0}$.
    \end{enumerate}  
    The first statement is clear, since $F|\SP^k(\CC \setminus \{0\}) \in D^{
    \ge -k}$ (with respect to the standard $t$-structure). To show the second
    statement, we first claim that $s_{k-1}^*a^! \simeq r^!t^*$ when
    restricted to $\NN$-constructible sheaves. To this end, we note that the
    commutative square in \eqref{eq:texactsquare} is pullback, so that base
    change provides an equivalence $(s_{k-1})_*r^! \simeq a^!t_*$. Therefore,
    restriction along $t$ provides a map $a^! \to a^! t_* t^* \simeq
    (s_{k-1})_*r^!t^*$ which is adjoint to our desired comparison map
    $s_{k-1}^*a^! \to r^!t^*$ . It is now readily verified (on stalks) that
    this map is indeed an equivalence: Namely, we use the recollement
    triangles for $a^!$ and $r^!$ expressing sections with support as fibers of
    the restriction map to the respective complements, noting that $F|\CC^{k-1}
    \times \RR_{>0}$ is locally constant and the inclusion $\SP^k(\RR_{>0})
    \subset \CC^{k-1} \times \RR_{>0}$ is a weak homotopy equivalence (both
    spaces are contractible). 

    This proves the claim. Statement \ref{tex:2} now follows by combining the
    perverse $t$-exactness of the vanishing cycles functor $a^!b^* \simeq
    \Phi_{e_{k}}$ and the perverse $t$-exactness of $s_{k-1}^*$ which holds by
    the induction hypothesis. 
\end{proof}

We will further consider the standard $t$-structure on each $\Sh^{\le
k}(S^1,\D)$ which is the one obtained by pulling back (via $j^*$) the
$t$-structure on $\Sh^{\le k}(\Rg,\D)$.

\begin{prop}\label{prop:complex_analytification_t_exact}
        The functor
\[\begin{tikzcd}
	{\Sh_{\NN}(\SP^{k}(\CC);\D)} & {\Sh^{\le k}(S^1,\D)}
	\arrow["\AA", from=1-1, to=1-2]
\end{tikzcd}\]
is $t$-exact.
\end{prop}
\begin{proof}
This is an immediate consequence of Corollary \ref{cor:analytification_is_vanishing}:
Vanishing cycles are $t$-exact with respect to the perverse $t$-structures and it follows immediately from the definition of the
perverse $t$-structure that it induces the standard $t$-structure on the subcategory of sheaves which are concentrated in $0$.
\end{proof}
\begin{thm}
        For a presentable stable $\infty$-category $\D$ equipped with an accessible left complete $t$-structure with heart $\A$, the
        analytic expansion equivalence
        \[
                \Sh^{*!}_{\NN}(\SP^{\infty}(\CC);\D)  \simeq \Sh^{\an}(S^1;\D)^c
        \]
        of Corollary \ref{cor:starshriekequiv} restricts to an equivalence
        \[
                \Perv^{*!}_{\NN}(\SP^{\infty}(\CC);\D) \simeq \A_{\Lambda_{\infty}}
        \]
        of abelian categories.
\end{thm}
\begin{proof}
        By Proposition \ref{prop:complex_analytification_t_exact} it suffices to note that, essentially by definition, the limit
        of hearts on the right-hand side of the equivalence (under the $i^{*!}$) is given by $\A_{\Lambda_{\infty}}$
        (c.f. Proposition \ref{prop:starshriek_analytic}).
\end{proof}
\appendix

\section{Exit paths on smooth conical manifolds}
\label{section:exit_paths}

\subsection{Conical manifolds}%
\label{sub:conical_manifolds}

We wish to place ourselves into a geometric context that allows us to study
constructible sheaves on suitable stratified spaces via exit paths. In this
section, we review the necessary background material.

\begin{defi}[\cite{lurie:ha}]
	Let $A$ be a partially ordered set. We endow the set $A$ with the Alexandroff topology by declaring:
	\[
	U\subset A \text{ is open, if and only if $U$ is upwards closed, i.e. $a\in U$ and $a\le b$ implies $b\in U$.}
	\]
    We define an {\em $A$-stratification} of $X$ to be a continuous map $s: X \to A$.
    For $a \in A$, the subspace $X_a = s^{-1}(a) \subset X$ is called the {\em
    $a$-stratum} of the stratification. We further denote e.g. $X_{\le a} =
    s^{-1}(A_{\le a})$ where $A_{\le a} = \{b\in A | b\le a\}$.
    A map of stratified spaces $X \to A$ and $Y \to B$ is a commutative square
    \[
    \begin{tikzcd}
        X \ar{r}\ar{d} & Y \ar{d}\\
        A \ar{r} & B
    \end{tikzcd}
    \]
    of continuous maps where the vertical maps are the respective stratifications of $X$ and $Y$.
\end{defi}

We will use the theory of conical manifolds (for terminology cf. \cite{KSW25assembly}), which
encompasses the symmetric product spaces we will study, and enjoys multiple
desirable properties. Fundamental to this theory is the notion of the cone of a
stratified space.

\begin{defi}
    Let $s:X\to A$ be a stratified space. The {\em open cone} $C(X)$ of $X$ is a stratified space with stratification
	\[
	s^\triangleleft : C(X) \to A^\triangleleft = A \cup \{-\infty\}.
	\]
	It is given as follows:
	\begin{enumerate}
		\item  As a set $C(X) = \{*\}\sqcup (X\times \RR_{>0})$.
		\item $U\subset C(X)$ is open if and only if $U\cap (X\times \RR_{>0})$ is open and if $* \in U$ implies
                        that $X\times (0,\epsilon)\subset U$ for some $\epsilon > 0$.
                \item The stratification is given by $s^\triangleleft(*)=-\infty \in A^{\triangleleft}$ and $s^\triangleleft(x,t) = s(x)$ for other points.
	\end{enumerate}
\end{defi}
We introduce conical disks, the theory's analogue of open balls.
\begin{defi}
    A stratified space $s:X \to A$ is a {\em conical disk}, if there exists a stratified second-countable compact Hausdorff space $r:Z \to A$ and $n \geq 0$ such that $X$ is isomorphic, as a stratified space, to the stratified space $r^{\triangleleft}: \RR^{n} \times C(Z) \to A^{\triangleleft}$,
	where the stratification is inherited from that on the cone, i.e. $r^{\triangleleft}(u,y)=r^{\triangleleft}(y)$.
\end{defi}
A conical manifold is a stratified space that is locally modeled on stratified disks.
\begin{defi}
    A stratified space $s:X \to A$ is a {\em conical manifold} if the closure of $\im s \subset A$ is a Noetherian poset and
        every point $x \in X$ has an open neighborhood $x \in U_x$ that,
        considered with the induced stratification, is a conical disk and such that $x$ in the minimal stratum of $U_x$.
        We call $U_x$ a conical neighborhood of $x$.
\end{defi}

\begin{rem}
    \label{rem:basis} Let $X$ be a conical manifold. Then the
    collection of open subsets which are conical disks forms a basis for the topology
    of $X$.
\end{rem}

In \cite{aft17} the authors also introduce (with somewhat different terminology) a theory of smooth conical manifolds
which refines the above notion of topological conical manifolds just as the notion of smooth manifolds refines that of
topological manifolds. All the conical manifolds we investigate in this paper naturally admit the structure of a
smooth conical manifold.

The following proposition summarizes the parts of the theory that are relevant for our discussion of exit path categories.
\begin{prop}\label{prop:basics_properties}
	\begin{enumerate}
		\item There is an $\infty$-category $\Mfld$ with objects given by smooth conical manifolds and $1$-morphisms
			given by smooth stratified open embeddings of conical manifolds.

			This category admits a full subcategory $\Disk \subset \Mfld$ whose objects are given by smooth conical disks.
		\item \label{item:unique_neighborhood} For a smooth conical manifold $X$ any two conical neighborhoods of a point $x \in X$ are equivalent in $\Disk$.

			For a conical disk $U \subset X$ we denote by $X_{[U]}$ the subspace of points $x \in X$ whose conical neighborhoods are equivalent to $U$.
		\item \label{item:equal_type_are_equiv} For an inclusion $B \subset B'$ of conical disks the following are equivalent:
			\begin{enumerate}
				\item The inclusion induces an equivalence in $\Disk$.
				\item $B$ and $B'$ have the same minimal stratum.
				\item $B \simeq B'$ in $\Disk$.
				\item The inclusion is a stratified homotopy equivalence.
			\end{enumerate}
		\item \label{item:disks_compute_stratum} For a smooth conical manifold $X$ there is an equivalence
			\[
			(\Disk/X)^{\simeq} \simeq \coprod_{[U] }X_{[U]}
			\]
			describing the groupoid core of the slice category as the disjoint union of subspaces running over all the equivalence types of conical disks in $X$.
	\end{enumerate}
\end{prop}
\begin{proof}
	Most of these statements are proved in section $4$ of \cite{aft17}, specifically
	Theorem 4.3.1 and Lemma 4.4.8 therein. The first three points of the equivalence \ref{item:equal_type_are_equiv}
        are also part of Theorem 4.3.1, the last equivalence follows
	straightforwardly from the considerations in the section, as observed in \cite{volpe2025verdier}.
\end{proof}

\subsection{Constructible sheaves and exit paths}
\label{subsec:cons_exit_paths}
Let $\D$ be a stable bicomplete $\infty$-category. We will make use of the six
functor formalism for sheaves with coefficients in $\D$ established in
\cite{volpe2025six}.
\begin{rem}\label{rem:non-stable-coefficients}
        The results of this subsection also hold for any presentable coefficient category for which
         Proposition \ref{prop:homotopy_equivalence_constructible_sheaves} holds, e.g. compactly generated presentable
         categories by \cite{haine2023homotopyinvariance}.
\end{rem}
\begin{defi}
	\label{defi:constructible_sheaf}
    Let $X$ be a locally compact space equipped with a stratification
    $X \to A$. An {\em $A$-constructible sheaf} $F: \U(X)^{\op} \to \D$ is a
    sheaf that is locally constant on each stratum, i.e., letting $i_a:X_a \to
    X$ denote the inclusion of the $a$-stratum, the sheaf $i_a^{*}F$ is a
    locally constant sheaf on $X_a$ for each $a \in A$.
	We denote by
	\[
		\Sh_A(X;\D) \subset \Sh(X;\D)
	\]
	the full subcategory spanned by the $A$-constructible sheaves.
    More generally, let $X \to P$ be a stratified space and suppose we are
    given a map of posets $P \to \NN$ such that
    \[
    X_{0} \subset X_{\le 1} \subseteq X_{\le 2} \subseteq \dots \subseteq X
    \]
    identifies $X$ with the seqential colimit of locally compact 
    spaces $X_{\le n}$. 
    Let $i_n:X_{\le n} \to X_{\le n+1}$ denote the inclusion of one closed
    subspace in the next higher subspace. We define the category of
    $P$-constructible $*$-sheaves $\Sh_P^*(X;\D)$ on $X$ as the inverse limit of the
    diagram
    \[
        \begin{tikzcd}
            \dots & {\Sh_{P_{ \le 2}}(X_{\le 2};\D)} & {\Sh_{P_{ \le 1}}(X_{\le 1};\D)} & {\Sh_{P_{0}}(X_{0};\D)}
        \arrow["{i_2^*}", from=1-1, to=1-2]
        \arrow["{i_1^*}", from=1-2, to=1-3]
        \arrow["{i_0^*}", from=1-3, to=1-4]
        \end{tikzcd}
    \]
\end{defi}
Constructible sheaves behave well with respect to stratified homotopy equivalences:
\begin{prop}\label{prop:homotopy_equivalence_constructible_sheaves}
Let $\D$ be a stable $\infty$-category.
\begin{enumerate} 
    \item Let $f:(X \to A) \to (Y \to B)$ be a stratified homotopy equivalence
        between locally compact stratified spaces, stratified over Noetherian
        posets. Then the induced functor $f^*:\Sh_B(Y;\D) \to \Sh_A(X;\D)$ is
        an equivalence.
    \item Let $U \subset V$ be an inclusion of open subsets of a stratified
        space $X \to A$ which is a stratified homotopy equivalence with
        respect to the induced stratifications. Then, for any constructible
        sheaf $F \in \Sh_A(X;\D)$, the restriction map
        \[
            F(V) \to F(U)
        \]
        is an equivalence.
\end{enumerate}
\end{prop}
\begin{proof}
	The first part is proved in \cite[Theorem 3.4]{volpe2025verdier} and
	the second part is a special case of Corollary 3.5 therein.
\end{proof}

We recall a characterization of constructible sheaves on (smooth) conical manifolds.
\begin{prop}\label{prop:constructible_sheaf_conical_manifold}
    Let $F:\U(X) \to \D$ be a sheaf on a (smooth) conical manifold $X \to A$.
    Then $F$ is constructible if and only if the following conditions holds:
    The restriction $F(B') \to F(B)$ along any inclusion $B \subset B'$ of conical disks which is a stratified homotopy
    equivalence is an equivalence.
\end{prop}
\begin{proof}
	This is \cite[Proposition 3.13]{volpe2025verdier}.
\end{proof}
This result leads us to make the following more general definition.
\begin{defi}\label{defi:constructible_presheaf}
        Let $X \to A$ be a stratified space, and $\O \subset \U(X)$ a subposet of opens.
        A presheaf $F: \O^{\op} \to \D$ is called $A$-constructible, if for any inclusion
        $U \subset U'$ in $\O$, which is a stratified homotopy equivalence, the induced
        morphism $F(U') \to F(U)$ is an equivalence in $\D$.
        The $\infty$-category of $A$-constructible presheaves
        \[
                \Fun_A(\O^{\op},\D) \subset\Fun(\O^{\op},\D)
        \]
        is spanned by the $A$-constructible presheaves.
\end{defi}

Following Lurie, we define:
\begin{defi}
        The category of {\em hypersheaves}
        \[
                \Hyp_A(X;\D) \subset \Sh_A(X;\D)
        \]
        is the full subcategory spanned by sheaves $F$ satisfying the following descent condition:
        \begin{itemize}
                \item For any open subset $U \subset X$ and diagram $\chi:D \to \U(U)^{\op}$ of $1$-categories
                        such that $\chi^{-1}(x)$ is weakly contractible for all $x \in U$, the restriction induces an equivalence
                        \[F(U) \simeq \lim F \circ \chi.\]
        \end{itemize}
\end{defi}

\begin{thm}\label{thm:exodromy}
	Let $X$ be a (smooth) conical manifold and $\B$ a basis consisting of conical disks.
	The restriction
	\[
                \Hyp_A(X,\D) \to \Fun_A(\B^{\op},\D)
	\]
	is an equivalence of $\infty$-categories. More precisely, the functors in $\Hyp_A(X,\D)$ are exactly the right Kan extensions of functors in $\Fun_A(\B,\D)$.
\end{thm}
Hyperdescent is therefore already encoded in the constructibility condition and constructibility can be checked on a basis.
Most of the remaining section is dedicated to the proof of Theorem \ref{thm:exodromy}. It is a variation of
\cite[Theorem 5.3.1]{KSW25assembly} and our proof largely follows theirs. However, parts of the arguments
are simplified by the fact that we are studying bases of disks rather than disjoint unions thereof.

The first step in the proof is a more explicit computation of the localization
of $\B$ at the set of stratified homotopy equivalences.
\begin{prop}\label{prop:explicit_localization}
	The functor $\B \to \Disk_{/X}$ is the localization at the subposet $\J_{\B}$ of
	inclusions of conical disks of the same isomorphism type.
\end{prop}
\begin{proof}
	We check the criterion of \cite[Lemma A.2.1]{KSW25assembly}. By statement \ref{item:equal_type_are_equiv} of Proposition
	\ref{prop:basics_properties}, the morphisms that are sent to equivalences are precisely those in $\J_{\B}$. Both parts of the
	criterion thus reduce to showing that the induced functor
	\[
		\J_{\B} \to (\Disk_{/X})^{\simeq}
	\]
	becomes an equivalence of groupoids after geometric realization. We note that for every conical disk $B \in \B$, the slice category $\Disk_{/B}$ is weakly contractible, since $\id_{B}$ is the final object. Therefore, we compute
	\[
		|\J_{\B}| \simeq \colim_{B \in \J_{\B}}(\Disk_{/B})_{\operatorname{term}}^{\simeq}
	\]
	where $(\Disk_{/B})_{\operatorname{term}}^{\simeq}$ is the contractible groupoid of disks of embeddings of disks
        of the same type as $B$ in $B$.

	To show that the map
	\[
		\colim_{B \in \J_{\B}}(\Disk_{/B})_{\operatorname{term}}^{\simeq} \to (\Disk_{/X})^{\simeq}
	\]
	is an equivalence, we utilize fact \ref{item:disks_compute_stratum} of Proposition \ref{prop:basics_properties}.
	It allows us to equivalently study the map
	\[
		\colim_{B \in \J_{\B}} B_{[B]} \to \coprod_{[U]} X_{[U]}
	\]
	This map is an equivalence by Lurie's van Kampen theorem. Indeed, it comes from the covering
	\[
		\J_{\B} \to \U(\coprod_{[U]} X_{[U]})
	\]
	obtained by regarding $B_{[B]}$ as an open subset of the factor $X_{[B]}$.
	The condition that the fiber over each point $x \in \coprod_{[U]} X_{[U]}$ must be weakly contractible is satisfied: Given
	$B, B' \in \B$ of the same isomorphism type such that $x \in B_{[B]} \cap B'_{[B]}$, there exists $V \in \B$ with
	$x \in V \subset B \cap B'$ since $\B$ is a basis. $V$ is then also of the same isomorphism type by \ref{item:unique_neighborhood} of Proposition \ref{prop:basics_properties}. Thus the fiber is cofiltered.
\end{proof}
\begin{cor}
	\label{cor:constructibility_Kan}
	Any right Kan extension $F: \U(X)^{\op} \to \D$ of a constructible $F|{\B^{\op}} \in \Fun_A(\B^{\op},\D)$ is itself constructible.
\end{cor}
\begin{proof}
	For any inclusion $B \hookrightarrow U$ of conical disks of the same type with $B \in \B$ the induced morphism $F(U) \to F(B)$ is an equivalence:
	Indeed, by the above theorem
	\[
		F(U) \simeq \lim_{({\B/U})^{\op}}F \simeq \lim_{({\Disk_{/U}})^{\op}} F
	\]
	and the arrow $U \to B$ is an equivalence in $\Disk^{\op}$, hence final in $(\Disk_{/U})^{\op}$.

	For an arbitrary inclusion $U \hookrightarrow V$ of conical disks of the same type, we choose $B,B' \in \B$ of the same type as $V$ and $U$ that fit into a commutative square
	\[\begin{tikzcd}
		U & V \\
		B & {B'}
		\arrow[from=1-1, to=1-2]
		\arrow[from=2-1, to=1-1]
		\arrow[from=2-1, to=2-2]
		\arrow[from=2-2, to=1-2]
	\end{tikzcd}\]
	The vertical and bottom arrows induce equivalences under $F$, so inclusion $U \subset V$ does too.
\end{proof}
\begin{prop}
	\label{prop:hypersheaf_condition}
	A functor $F:\U(X)^{\op} \to \D$ that is a right Kan extension of its restriction $F|{\B^{\op}} \in \Fun_A(\B^{\op},\D)$ is automatically a hypersheaf.
\end{prop}
\begin{proof}
	Suppose we are given a hypercovering $\chi:D \to (\U(X)/U)^{\op}$ of an open set $U \subset X$.
	We must show that $F$ satisfies descent for this hypercovering. Let $\G$ denote the Cartesian Grothendieck construction
	of the functor
	\[
		D \to \Cat_{\infty}, d \mapsto (\B/\chi(d))^{\op}
	\]
	and let $\rho: \G \to (\B/U)^{\op}$ denote the induced map.
	Then $\lim F \circ \chi \simeq \lim F \circ \rho$ since $F \circ \chi$ is a right Kan extension along the projection to $D$. Indeed, the pointwise formula for a Kan extension reduces to the fiberwise requirement $F(\chi(d)) \simeq \lim_{(\B/\chi(d))} F$, as the projection is Cartesian. This is satisfied because $F$ is a right Kan extension of its restriction to $\B^{\op}$.

	As localizations are coinitial, we also have
	\[
		\lim F \circ \chi \simeq \lim F|\colim_{d \in D^{\op}}(\B/\chi(d))^{\op}
	\]
	We now consider the following commutative diagram in $\Cat_{\infty}$:
	\[\begin{tikzcd}
		{\colim_{d \in D^{\op}}(\B/\chi(d))^{\op}} & {(\B/U)^{\op}} & \D \\
		{\colim_{d \in D^{\op}}(\Disk_{/\chi(d)})^{\op}} & {(\Disk_{/U})^{\op}}
		\arrow[from=1-1, to=1-2]
		\arrow[from=1-1, to=2-1]
		\arrow["F", from=1-2, to=1-3]
		\arrow[from=1-2, to=2-2]
		\arrow[from=2-1, to=2-2]
		\arrow[from=2-2, to=1-3]
	\end{tikzcd}\]
	The vertical arrows are localizations by Proposition \ref{prop:explicit_localization}. Since localizations are coinitial
	and
	\[
		F(U) \simeq \lim_{(\B/U)^{\op}} F,
	\]
	it suffices to show that the bottom vertical arrow is also coinitial.
	Since $(\Disk_{/U})^{\op} \to (\Disk)^{\op}$ is a left fibration, we may apply the dual of \cite[Lemma A.3.1]{KSW25assembly}
	to instead show that the induced map of groupoid cores is an equivalence, which is identified with
	\[
		\colim_{d \in D^{\op}}\coprod_{[W]}\chi(d)_{[W]} \to \coprod_{[V]}U_{[V]}
	\]
	This is again an equivalence by Lurie's van Kampen theorem: By definition of a hypercovering it is a Kan complex over each fiber, so the condition on the covering is satisfied.
\end{proof}

\begin{proof}[Proof of Theorem \ref{thm:exodromy}]
The Theorem follows immediately from Corollary \ref{cor:constructibility_Kan} and Proposition \ref{prop:hypersheaf_condition}.
\end{proof}

In the above, we have seen how to identify constructible hypersheaves with constructible presheaves on a basis.
Given an intersection-stable basis of conical disks, the same also holds true for ordinary sheaves.
\begin{cor}\label{cor:cons_sheaves_exodromy_intersection_stable}
	Let $\B$ be a basis of conical disks which is closed under intersections. Then the forgetful functor
	\[
		\Sh_A(X;\D) \to \Fun_A(\B^{\op},\D)
	\]
	is an equivalence of categories.
\end{cor}
\begin{proof}
	Any constructible hypersheaf is a sheaf. Conversely, a sheaf on an intersection-stable basis is the right Kan extension of its restriction, so any constructible sheaf is a constructible hypersheaf by the theorem.
\end{proof}
We are also interested in coarsening the stratification of a conical manifold, i.e. in diagrams
\[
	X \to A \to P
\]
where $X \to A$ is a conical manifold and $A \to P$ is a surjective map of posets.
Of course, in general the stratified space $X \to P$ need not be a conical manifold, but we nevertheless wish to express exit paths
on $X \to P$ by a variant of Theorem \ref{thm:exodromy}.
\begin{prop}\label{prop:coarsen_stratification_cons_sheaves}
	Let $X \to A \to P$ be a coarsening of a conical
	manifold, and let $\B$ denote an intersection-stable basis of conical disks with respect to $X \to A$.
        Suppose that the coarsening satisfies the following additional assumption:
        \begin{enumerate}[label=(M)]
        \item\label{item:conical_extra_condition}
	Any inclusion of conical disks $B \subset B'$ in $\B$ which have the same minimal stratum in $P$ (not necessarily in $A$) is
        a stratified homotopy equivalence with respect to the stratification induced by $P$.
\end{enumerate}
	Then, for any intersection-stable basis of conical disks $\B$ for $X \to A$, there is an equivalence
	\[
		\Sh_P(X;\D) \to \Fun_P(\B^{\op},\D).
	\]
\end{prop}
\begin{proof}
	By $\cap$-stability of $\B$ and Theorem \ref{thm:exodromy}, restriction yields an equivalence
	\[
		\Sh_A(X;\D) \to \Fun_A(\B^{\op},\D).
	\]
	It therefore remains to show that the essential image of the subcategory $\Sh_P(X;\D) \subset \Sh_A(X;\D)$ under
	the restriction functor is precisely $\Fun_P(\B^{\op},\D)$. It follows from Proposition
	\ref{prop:homotopy_equivalence_constructible_sheaves} that any $P$-constructible sheaf
        restricts to a $P$-constructible presheaf on the basis.
	Conversely, it suffices to establish that the right Kan extension $F$ of a $P$-constructible presheaf $F|\B^{\op}$ is a
	$P$-constructible sheaf. For $p \in P$ and $x \in X_p$ let $B$ be a conical disk containing $x$ in its minimal $A$-stratum.
	Then the pullback $j_p^*F$ along $j_p:X_p \to X$ is constant on the contractible open set $B \cap X_p \subset X_p$ by
        \ref{item:conical_extra_condition}.
\end{proof}

\subsection{Verdier duality}
To identify filtered objects with chain complexes via sheaf constructions, we will utilize the Verdier dual construction
from \cite[Section 5.5.5]{lurie:ha}, which we briefly recall. The coefficient category $\D$ is again assumed to be bicomplete and
stable.

\begin{thm}\cite[Theorem 5.5.5.1]{lurie:ha}
	Let $X$ be a locally compact Hausdorff space. There is a canonical equivalence of $\infty$-categories
	\[
		\V:\Sh(X;\D) \to \coSh(X;\D)
	\]
	where $\coSh(X;\D)$ denotes the $\infty$-category of $\D$-valued cosheaves on $X$, i.e. $\coSh(X;\D) = \Sh(X;\D^{\op})^{\op}$.
\end{thm}
Explicitly, this equivalence is given pointwise by taking compact sections with support, that is, for a
sheaf $F:\U(X)^{\op} \to \D$, the cosheaf $\V(F):\U(X) \to \D$ is given by the assignment
\[
	U \mapsto \Gamma_{c}(U;F) = \colim_{K \subset U} \Gamma_K(X;F) =
        \colim_{K \subset U} \fib( F(U) \to F(U \setminus K))
\]
where the indexing diagram runs over all compact subsets $K \subset U$.

By definition, Verdier duality interchanges inverse image and exceptional inverse image:
\begin{defi}
        \label{def:verdier_duality_interchange}
        Let $f:X \to Y$ be a map between locally compact Hausdorff spaces. The exceptional inverse
        image $f^!: \Sh(Y;\D) \to \Sh(X;\D)$ is the unique functor fitting into the natural equivalence
        \[
                \V \circ f^! \simeq f^* \circ \V.
        \]
        where $f^*:\coSh(Y;\D) \to \coSh(X;\D)$ denotes the inverse image functor
        of cosheaves.
\end{defi}

For smooth conical manifolds the Verdier dual contruction is compatible with
constructibility.
\begin{thm}\cite[Proposition 4.6]{volpe2025verdier}
	\label{thm:verdier_duality_constructibility}
Let $X \to A$ be a smooth conical manifold. Verdier duality restricts to an
equivalence
	\[
		\V:\Sh_A(X;\D) \simeq \coSh_A(X;\D)
	\]
between constructible sheaves and constructible cosheaves (defined
dually).
\end{thm}
Definition \ref{def:verdier_duality_interchange} implies that
for any map
\[\begin{tikzcd}
	X & Y \\
	A & B
	\arrow["f", from=1-1, to=1-2]
	\arrow[from=1-1, to=2-1]
	\arrow[from=1-2, to=2-2]
	\arrow[from=2-1, to=2-2]
\end{tikzcd}\]
of conical manifolds, the exceptional inverse image descends to a functor
$f^!:\Sh_B(Y;\D) \to \Sh_A(X;\D)$ of constructible sheaves.
\section{Homotopies on symmetric products}
\label{sec:technical}
Let $C=C(N)$ be an open cone on a closed manifold. In this subsection, we fix $n \in \NN$ and compile homotopy arguments
on $\SP^n(C)$. As in Remark \ref{rem:basics}, a basis $\B$ of $\SP^n(C)$ is given by sums
\[
        \sum_{i=0}^k \SP^{m_i}(D_i) =\{\sum_{i=1}^n x_i |  D_j \text{ contains
        exactly $m_j$ of the $x_j$ for all $k$} \}
\]
with $D_1, \dots,D_k$ pairwise disjoint disks, $m_i>0$ for all $i>0$, $D_0=C_{\epsilon}(N)$ and $D_1, \dots,D_k$ disjoint
from $D_0$ if $m_0 >0$.
\begin{lem}\label{lem:filtration_basics_homotopy}
        Let $B= \sum_{i=0}^k \SP^{m_i}(D_i) \in \B$ and $B' = \sum_{j=0}^r \SP^{l_j}(D_j') \in \B$ be basic opens in $\SP^n(C)$
        with $B' \subset B$. This inclusion is an $\NN$-stratified homotopy equivalence if and only if
        $B$ and $B'$ have the same minimal stratum, i.e. if $m_0 = l_0$.
\end{lem}

\begin{proof}
        The forward implication is obvious. To prove the other direction, we note that there is a stratified homeomorphism
        $B \cong \prod_i \SP^{m_i}(D_i) \subset \prod_i \SP^{m_i}(C)$. The projection $\prod_i \SP^{m_i}(D_i) \to \SP^{m_0}(D_0)$
        is an $\NN$-stratified homotopy equivalence, because $\SP^m(D)$ is contractible for a disk $D \in \DD_1(C^*)$
        by functoriality. If $m_0 = l_0$, it follows that there is a commutative square
\[\begin{tikzcd}
	{B'} & B \\
	{\SP^{m_0}(D_0')} & {\SP^{m_0}(D_0)}
	\arrow[from=1-1, to=1-2]
	\arrow[from=1-1, to=2-1]
	\arrow[from=1-2, to=2-2]
	\arrow[from=2-1, to=2-2]
\end{tikzcd}\]
where the vertical arrows are $\NN$-stratified homotopy equivalences. The bottom arrow is an $\NN$-stratified homotopy equivalence
as well because $D_0' \subset D_0$ is a $[1]$-stratified homotopy equivalence.
\end{proof}
We now turn to the related families of opens in $\SP^n(C)$.
\begin{lem}
        \label{lem:shrink_disks}
        For $\bigcup D_i \in \DD_{\le n}(C^*)$ choose charts $\phi_i$ for $D_i$.
        Let $D_i' = \phi^{-1}(B_K(0)) \subset C$ denote the
        open disk of radius $K$ in these charts. The following inclusions are stratified homotopy equivalences
        with respect to the discriminant stratification (for all multiplicities $m_i$):
\[
        \begin{split}
                i_{\cap}:\sum_{i=1}^k \SP^{m_i}(D_i') +\SP^{m_0}(C) &\subset \sum_{i=1}^k \SP^{m_i}(D_i) + \SP^{m_0}(C) \\
                i_{\cup}:\bigcup D_i' + \SP^{n-1}(C) &\subset \bigcup D_i + \SP^{n-1}(C)
        \end{split}
\]
\end{lem}
\begin{proof}
        We begin by constructing an auxilliary family of homotopies: For each collection of real numbers $R_1, \dots,R_k >0$
        there is a homotopy
        \[
                H_{R_1,\dots,R_k}:C \times [0,1] \to C
        \]
        such that
        \[
                H_{R_1,\dots,R_k}(x,t)=\begin{cases}
                        x \text{ if } x \notin \bigcup_i D_i \text{ or } x \in D_i \text{ and } ||\phi_i(x)|| \ge 2R_i \\
                        \phi_i^{-1}((1-t)\phi_i(x)+t\frac{K}{1+||\phi_i(x)||}\phi_i(x)) \text{ if } ||\phi_i(x)|| \le R_i
                \end{cases}
        \]
        Indeed, we can extend the above on each chart by linearly interpolating between the two cases in the remaining
        region. For such an extension the values of $H_{R_1,\dots,R_k}$ are continuous in the $R_1,\dots,R_k$.

        We now construct the witnessing homotopy for $i_{\cap}$:
        For $s=\sum_{j} x_j \in \sum_{i=1}^k \SP^{m_i}(D_i) + \SP^{m_0}(C) $ let $y_1, \dots,  y_l \in D_i$ denote the $l \geq m_i$  (not necessarily distinct) summands of $s$ which lie in $D_i$.
        Suppose further that they are ordered by ascending $||\phi_i(y_j)||$.
        We denote $R_i(s) = ||\phi_i(y_{m_i})||+1$. Note that these are continuous functions in $s$. The homotopy
        \[
                H_{\cap}:  (\sum_{i=1}^k \SP^{m_i}(D_i) + \SP^{m_0}(C)) \times [0,1] \to \sum_{i=1}^k \SP^{m_i}(D_i) + \SP^{m_0}(C)
        \]
        given by the continuous assignment
        \[
                (s= \sum_i x_i,t) \mapsto \sum_i H_{R_1(s),\dots,R_k(s)}(x_i,t)
        \]
        witnesses that $H_{\cap}(-,1)$ is a homotopy inverse to the $i_{\cap}$.

        Similarly, for $s\in \bigcup D_i + \SP^{n-1}(C)$ let $R(s) =1 + \min_{i,j} ||\phi_i(x_j)||$ (The minimum is taken over
        all pairs where $x_j \in D_i$). The homotopy

        \[
                H_{\cup}: (\bigcup D_i + \SP^{n-1}(C)) \times [0,1] \to \bigcup D_i + \SP^{n-1}(C)
        \]
        with
        \[
                (s=\sum_i x_i,t) \mapsto \sum_i H_{R(s),\dots,R(s)}(x_i,t)
        \]
        exhibits $H_{\cup}(-,1)$ as a homotopy inverse to $i_{\cup}$.
\end{proof}
        We note the following version of an argument appearing in the proof of Lemma \ref{lem:shrink_disks}:
\begin{lem}
        \label{lem:symmetric_homotopy}
        Let $D_i$ be open disks in $C^*$ and $D_i'=\phi_i^{-1}(B_K(0))$ as in Lemma \ref{lem:shrink_disks}.
        For homotopies $H_i:D_i' \times [0,1] \to D_i'$ there is a homotopy $H_i:C \times [0,1] \to C$ such that
        $H$ is constant on $C \setminus \bigcup_i D_i$ and $H|D_i' = H_i$.

        If the $H_i$ are continuously parameterized by a space $X$, then $H$ may also be chosen as a continuous function of $X$.
\end{lem}
\begin{proof}
        We specify $H(y,t)=y$ if $||\phi_i(y)|| \ge 2K$. Continuity is then apparent if we solve the ensuing extension problem
        in each chart individually. Here, existence follows from standard extension arguments for manifolds. Closer inspection
        shows the additional statement.
\end{proof}
\begin{lem}\label{lem:extend_zero_disk}
        For $\bigcup_{i=1}^{r}D_i \in \DD_{\le n} (C^{*})$ let $D_i' \subset D_i$ be open balls in a chart as above. There is
        $* \in D_0'=C_{\epsilon}(N)$ with $D_0' \cap D_i' = \emptyset$ such that the inclusion
        \[
                \SP^{m_0}(D_0') + \sum_{i=1}^r \SP^{m_i}(D_i') \subset \SP^{m_0}(C) + \sum_{i=1}^r \SP^{m_i}(D_i')
        \]
        in $\SP^n(C)$ is an $\NN$-stratified homotopy equivalence.
\end{lem}
\begin{proof}
        There is a stratified inclusion $j:\SP^{m_0}(C) \to \sum_{i=1}^r \SP^{m_i}(D_i')+ \SP^{m_0}(C)$ given by
        $s \mapsto s + \sum_{i \geq 1} m_i \cdot \phi_i^{-1}(0)$.
        Let $H_i$ denote a homotopy contracting $D_i'$ to the center point of the chart.
        Lemma \ref{lem:symmetric_homotopy} yields a homotopy $H$ on $C$ and the symmetric product of this homotopy witnesses
        that $j$ is a stratified homotopy equivalence.
        The set where $H$ is non-constant is contained in a closed set away from $*$ by construction, so there is
        an $\epsilon$-cone $* \in D_0' \subset C$ which is held constant by $H$.
        $H$ then also witnesses that the restriction $j|\SP^{m_0}(D_0')$ is a stratified homotopy equivalence.
        The inclusion $D_0' \subset C$ is also a stratified homotopy equivalence, so that $\SP^{m_0}(D_0') \subset \SP^{m_0}(C)$
        is a stratified homotopy equivalence as well. In conclusion, we have a commutative square
\[\begin{tikzcd}
        \SP^{m_0}(D_0') + \sum_{i=1}^r \SP^{m_i}(D_i') & \SP^{m_0}(C) + \sum_{i=1}^r\SP^{m_i}(D_i') \\
	{\SP^{m_0}(D_0')} & {\SP^{m_0}(C)}
	\arrow[from=1-1, to=1-2]
	\arrow[from=2-1, to=1-1]
	\arrow[from=2-1, to=2-2]
	\arrow[from=2-2, to=1-2]
\end{tikzcd}\]
where the bottom and vertical arrows are stratified homotopy equivalences, hence also the top arrow.
\end{proof}
\begin{cor}\label{cor:extended_disk_type}
        For $\bigcup_{i=1}^s V_i \in \DD_{\le n}(C^*)$ and $l_1,\dots,l_s >0$ we denote $l_0=n-\sum_{i=1}^s l_i$.
        Any inclusion $\sum_{i=0}^r \SP^{m_i}(D_i) \subset \sum_{i=1}^s \SP^{l_i}(V_i) + \SP^{m_0}(C)$ of a basic open
        $B=\sum_{i=0}^k \SP^{m_i}(D_i) \in \B$ with $m_0=l_0$ is an $\NN$-stratified homotopy equivalence.
\end{cor}
\begin{proof}
        By Lemma \ref{lem:shrink_disks} and Lemma \ref{lem:extend_zero_disk} we can find a conical basic $B' \subset B$ such that
        \[
                B' \subset \sum_{i=1}^s \SP^{l_i}(V_i) + \SP^{m_0}(C)
        \]
        is a stratified homotopy equivalence. By Lemma \ref{lem:filtration_basics_homotopy} $B' \subset B$ is
        a stratified homotopy equivalence. Thus so is the claimed inclusion.
\end{proof}
\begin{lem}
        \label{lem:intersection_inclusions}
        Given unions $\bigcup_{i=1}^k D_i, \bigcup_{j=1}^r V_j \in \DD_{\le n}(C^*)$
        of disjoint disks and multiplicities $m_1, \dots, m_k \geq 1$, $l_1, \dots,l_r \geq 1$, we define $m_0 =n- \sum_i m_i$
         and $l_0 = n- \sum_j l_j$. There is an inclusion
         \[
                 \sum_{i=1}^k \SP^{m_i}(D_i) +\SP^{m_0}(C)  + \subset \sum_{j=1}^r \SP^{l_j}(V_j)+\SP^{l_0}(C)
         \]
         in $\SP^n(C)$ if and only if for every $1 \le j \le r$:
         \[
                 \sum_{i:D_i \subset V_j} m_i \geq l_j.
         \]
         The inclusion is an $\NN$-stratified homotopy equivalence if and only if
         we have equalities in all the above.
\end{lem}
\begin{proof}
        The "if" part of the first equivalence is clear. Conversely, if one of the above
        inequalities does not hold, then there is no inclusion because the $V_j$ are disjoint and $* \notin V_j$.

        The minimal stratum of the above opens is determined by $m_0$ and $l_0$ respectively. If the inclusion is a stratified
        homotopy equivalence, we must have $m_0 =l_0$ and therefore equalities.

        Conversely, let $D_i' \subset D_i$ and $V_j' \subset V_j$ be open disks as in Lemma \ref{lem:shrink_disks},
        such that $D_i' \subset V_j'$ if $D_i \subset V_j$.
        We further choose cones $D_0',V_0' \subset C$ containing $*$ as in Lemma \ref{lem:extend_zero_disk}. There is a
        square
\[\begin{tikzcd}
	{\sum_{i=1}^k \SP^{m_i}(D_i) + \SP^{m_0}(C)} & {\sum_{j=1}^r \SP^{l_j}(V_j) + \SP^{l_0}(C)} \\
        {\sum_{i=1}^k \SP^{m_1}(D_i') + \SP^{m_0}(D_0')} & {\sum_{j=1}^r\SP^{l_j}(V_j')+ \SP^{l_0}(V_0')}
	\arrow[from=1-1, to=1-2]
	\arrow[from=2-1, to=1-1]
	\arrow[from=2-1, to=2-2]
	\arrow[from=2-2, to=1-2]
\end{tikzcd}\]
where the vertical inclusions are stratified homotopy equivalences by Corollary \ref{cor:extended_disk_type} and Lemma.
The bottom arrow is an $\NN$-stratified homotopy equivalence by Lemma \ref{lem:filtration_basics_homotopy}.
It follows that the top arrow is also a stratified homotopy equivalence as desired.
\end{proof}
\begin{lem}
        \label{lem:union_inclusion}
        Consider unions $\bigcup_{i=1}^k D_i, \bigcup_{j=1}^r V_j \in \DD_{\le n}(C^*)$.
        There is an inclusion
        \[
                \bigcup_{i=1}^k D_i +\SP^{n-1}(C) \subset \bigcup_{j=1}^r V_j +\SP^{n-1}(C)
        \]
        if and only if, for each $1 \le i \le k$ there is some $1 \le j \le r$ such that
        $D_i \subset V_j$.
        The inclusion is a stratified homotopy equivalence if and only if every $V_j$ contains one of
        the $D_i$, so that the inclusion relation defines a bijection.
\end{lem}
\begin{proof}
        The statement about inclusions is an immediate consequence of the connectedness of disks.

        If the inclusion is a stratified homotopy equivalence, then its restriction to the $1$-stratum
        is a homotopy equivalence, so inclusion induces an isomorphism $\pi_0(\bigcup_{i=1}^k D_i) \cong \pi_0(\bigcup_{j=1}^r V_j)$.

        Conversely, choose compact disks $D_i' \subset D_i$ and $V_j' \subset V_j$ as in Lemma \ref{lem:shrink_disks} and
        such that (after possibly reordering) $D_i' \subset V_i'$.
        There is a diagram of inclusions
        \[\begin{tikzcd}
	{\bigcup_iD_i + \SP^{n-1}(C)} & {\bigcup_iV_i + \SP^{n-1}(C)} \\
	{\bigcup_iD_i' + \SP^{n-1}(C)} & {\bigcup_iV_i' + \SP^{n-1}(C)}
	\arrow[from=1-1, to=1-2]
	\arrow[from=2-1, to=1-1]
	\arrow[from=2-1, to=2-2]
	\arrow[from=2-2, to=1-2]
\end{tikzcd}\]
        The vertical arrows are stratified homotopy equivalences by the second part of Lemma \ref{lem:shrink_disks}.
        The bottom inclusion is also a stratified homotopy equivalence. Witnessing homotopies are produced
        by glueing homotopies for the inclusions $D_i' \subset V_i'$ via Lemma \ref{lem:symmetric_homotopy} and taking the
        symmetrization of the resulting homotopies.
\end{proof}

\newpage
\bibliographystyle{alpha}
\bibliography{refs}

\end{document}

%% file: preamble.tex
\usepackage{fullpage}
\usepackage[dvipsnames]{xcolor}  
\usepackage{tipa}
\usepackage{dsfont}
\usepackage{tikz-cd}
\usepackage{tikz}
\usetikzlibrary{decorations.markings, patterns, decorations.pathmorphing}

\tikzset{->-/.style={decoration={
markings,
mark=at position #1 with {\arrow{>}}},postaction={decorate}}}

\usepackage{graphicx}
\usepackage{mathtools}
\usepackage{amsxtra}
\usepackage{amsmath}
\usepackage{amssymb}
\usepackage{amsfonts}
\usepackage{mathrsfs}
\usepackage{amsthm}
\usepackage[all]{xy}
\usepackage{enumitem}
\usepackage{hyperref}
\usepackage{subfigure}
\usepackage{mathabx}
\usepackage{euscript}
\usepackage{xargs}                      

\usepackage[colorinlistoftodos,prependcaption,textsize=tiny]{todonotes}
\newcommandx{\note}[2][1=]{\todo[linecolor=Plum,backgroundcolor=Plum!25,bordercolor=Plum,#1]{#2}}

\newtheoremstyle{example}{\topsep}{\topsep}%
     {}
     {}
     {\bfseries}
     {.}
     {2pt}
     {\thmname{#1}\thmnumber{ #2}\thmnote{ #3}}

\theoremstyle{example}
\newtheorem{exa}[equation]{Example}
\newtheorem{exas}[equation]{Examples}
\newtheorem{ex}[equation]{Example}
\newtheorem{rem}[equation]{Remark}

\newtheorem{defi}[equation]{Definition}

\newtheorem{thm}[equation]{Theorem}

\newtheorem{cor}[equation]{Corollary}
\newtheorem{lem}[equation]{Lemma}
\newtheorem{prop}[equation]{Proposition}

\hypersetup{%
    colorlinks,%
    linkcolor={red!50!black},%
    citecolor={blue!50!black},%
    urlcolor={blue!80!black}%
}

\setlist[enumerate,1]{label=(\arabic{*})}
\setlist[enumerate,2]{label=(\roman{*})}
\setlist[enumerate,3]{label=(\alph{*})}

\def\A{ {\EuScript A}}
\def\AA{ {\mathbb A}}

\def\Aut{\on{Aut}}
\def\An{\mathbb A}

\def\B{ {\EuScript B}}
\def\BB{\mathbb{B}}

\def\Bop{ {\operatorname {B}}}

\def\coSh{\operatorname{coSh}}
\def\c{\operatorname{c}}
\def\C{{\EuScript C}}

\def\CC{\mathbb{C}}

\def\Cat{{\EuScript Cat}}

\def\Ch{\operatorname{Ch}}

\def\colim{\operatorname*{colim}}

\def\cof{\operatorname{cof}}

\def\Conf{{\on{Conf}}}

\def\pD{{^{\mathfrak p}}D}
\def\D{{\EuScript D}}

\def\DD{{\mathbb{D}}}
\def\DDr{{\mathbb{D}^{\RR_{+}}}}
\def\DDrn{{\mathbb{D}^{\RR_{+}}_{\le n}}}

 \def\Disk{{\on{Disk}}}

\def\Fan{F^{\an}}
\def\an{\on{an}}

\def\fib{\operatorname{fib}}
\def\cofib{\operatorname{cofib}}
\def\tfib{\operatorname{tfib}}

\def\Fun{\operatorname{Fun}}

\def\G{{ \EuScript G}}

\def\Hyp{\operatorname{Hyp}}
 
\def\id{\on{id}}

 \def\im{\on{im}}

 \def\iso{{ \operatorname{iso}}}

\def\J{{\mathcal J}}
\def\J{ {\EuScript J}}

\def\lla{\longleftarrow}
\def\llra{\longleftrightarrow}
\def\lra{\longrightarrow}

\def\m{\operatorname{m}}

\def\Mfld{\on{Mfld}}

\def\NN{\mathbb{N}}

\def\O{{\mathcal O}}

\def\on{\operatorname}

\def\op{{\operatorname{op}}}

 \def\P{{\EuScript P}}
 \def\PP{{\mathbb P}}

\def\Perv{{\on{Perv}}}

\def\phi{{\varphi}}

\def\pt{ {\on{pt}}}

\def\Rg{{\mathbb R}_{\ge 0}}

\def\Re{\operatorname{Re}}

\def\RR{\mathbb{R}}

\def\SP{{\operatorname{SP}}}
\def\SPi{{\operatorname{SP}^{\infty}}}

\def\Sh{\on{Sh}}
\def\Shc{\on{Sh}_{\NN}}

\def\Shan{\Sh^{\an}}

\def\Shanr{\Sh^{\an,\mathbb{\RR}_{+}}}
\def\Shk{\Sh^{\le k}}
\def\Shkr{\Sh^{\le k, \mathbb{R}_{+}}}
\def\Shnr{\Sh^{\le n, \mathbb{R}_{+}}}

\def\Sp{{\on{Sp}}}

 \def\ShnRgC{\Sh_{[n]}(\SP^n(\mathbb{R}_{\ge 0});\EuScript{D})}
 \def\ShPnRgC{\Sh_{\mathbb{P}_{\le n}}(\SP^n(\mathbb{R}_{\ge 0});\EuScript{D})}

\def\U{ {\EuScript U}}
\def\UU{ {\mathbb U}}

\def\V{ {\EuScript V}}

\def\ZZ{\mathbb{Z}}

 \def\={{\,\, \simeq\,\, }}
 \def\-{{\setminus}}
\def\<<{\langle {}\hskip -.1cm {}\langle}
\def\>>{\rangle \hskip -.1cm \rangle}

%% file: symdk.bbl
\begin{thebibliography}{DHST26}

\bibitem[AC20]{arone-ching:handbook}
Gregory Arone and Michael Ching.
\newblock Goodwillie calculus.
\newblock In {\em Handbook of homotopy theory}, pages 1--38. CRC Press, Boca Raton, FL, 2020.

\bibitem[AFT17]{aft17}
David Ayala, John Francis, and Hiro~Lee Tanaka.
\newblock Local structures on stratified spaces.
\newblock {\em Advances in Mathematics}, 307:903--1028, 2017.

\bibitem[BBD82]{beilinson1982faisceaux}
A.~A. Beilinson, J.~Bernstein, and P.~Deligne.
\newblock Faisceaux pervers.
\newblock {\em Ast\'erisque}, 100:5--171, 1982.

\bibitem[Bei87]{beilinson:howtoglue}
A.~A. Beilinson.
\newblock How to glue perverse sheaves.
\newblock In {\em {$K$}-theory, arithmetic and geometry ({M}oscow, 1984--1986)}, volume 1289 of {\em Lecture Notes in Math.}, pages 42--51. Springer, Berlin, 1987.

\bibitem[DHST26]{dyckerhoff2026laxdescentfoliatedverdier}
Tobias Dyckerhoff, Till Heine, Simon Schneider, and Kostiantyn Tolmachov.
\newblock Lax descent and foliated verdier duality, 2026.

\bibitem[DK85]{DK85}
W.~G. Dwyer and D.~M. Kan.
\newblock Normalizing the cyclic modules of {C}onnes.
\newblock {\em Comment. Math. Helv.}, 60(4):582--600, 1985.

\bibitem[DK15]{DK15}
T.~Dyckerhoff and M.~Kapranov.
\newblock Crossed simplicial groups and structured surfaces.
\newblock In {\em Stacks and categories in geometry, topology, and algebra}, volume 643 of {\em Contemp. Math.}, pages 37--110. Amer. Math. Soc., Providence, RI, 2015.

\bibitem[DKS23]{DKS:milnor}
T.~Dyckerhoff, M.~Kapranov, and Y.~Soibelman.
\newblock Perverse sheaves on {Riemann} surfaces as {Milnor} sheaves.
\newblock {\em Forum Math. Sigma}, 11:50, 2023.
\newblock Id/No e92.

\bibitem[Hei23]{heine:dk}
T.~Heine.
\newblock {A categorified Dwyer-Kan correspondence}.
\newblock {\em \href{https://arxiv.org/abs/2303.03653}{\sf arXiv:2303.03653}}, 2023.

\bibitem[HPT23]{haine2023homotopyinvariance}
Peter~J. Haine, Mauro Porta, and Jean-Baptiste Teyssier.
\newblock The homotopy-invariance of constructible sheaves.
\newblock {\em Homology, Homotopy and Applications}, 25(2):97--128, 2023.

\bibitem[KSW25]{KSW25assembly}
Eilind Karlsson, Claudia~I. Scheimbauer, and Tashi Walde.
\newblock Assembly of constructible factorization algebras, 2025.

\bibitem[Lur09]{lurie:htt}
J.~Lurie.
\newblock {\em Higher topos theory}, volume 170 of {\em Annals of Mathematics Studies}.
\newblock Princeton University Press, Princeton, NJ, 2009.

\bibitem[{Lur}14]{lurie:tamagawa}
J.~{Lurie}.
\newblock {Tamagawa Numbers via Nonabelian Poincare Duality}.
\newblock {\em Lecture notes}, 2014.

\bibitem[{Lur}17]{lurie:ha}
J.~{Lurie}.
\newblock {Higher Algebra}.
\newblock {\em preprint}, September 2017.

\bibitem[ORW26]{tashi:cores}
Viktoriya Ozornova, Martina Rovelli, and Tashi Walde.
\newblock Cores and localizations of $(\infty,\infty)$-categories, 2026.

\bibitem[Sch03]{Schurmann2003}
J{\"o}rg Sch{\"u}rmann.
\newblock {\em Topology of Singular Spaces and Constructible Sheaves}.
\newblock Monografie Matematyczne. Birkh{\"a}user Basel, 1 edition, 2003.

\bibitem[Vol25a]{volpe2025six}
Marco Volpe.
\newblock The six operations in topology.
\newblock {\em Journal of Topology}, 18(4), 2025.

\bibitem[Vol25b]{volpe2025verdier}
Marco Volpe.
\newblock Verdier duality on conically smooth stratified spaces.
\newblock {\em Algebraic \& Geometric Topology}, 25(2):919--950, 2025.

\end{thebibliography}
